\documentclass[reqno,11pt]{amsart}
\usepackage{standard}
\usepackage[margin=1.05in,marginparwidth=1in]{geometry}
\usepackage{graphics,xcolor}
\usepackage{comment}
\usepackage{hyperref}
\usepackage{mathtools}
\usepackage{xspace}
\usepackage{musicography}

\newcommand{\loc}{\text{loc}}

\newcommand{\XPhi}[1]{{}^{#1}\Phi}

\newcommand{\hyp}{\mathcal{H}}
\newcommand{\gl}{\mathcal{G}}
\newcommand{\ellip}{\mathcal{E}}

\newcommand{\bv}{\mathbf{v}}
\newcommand{\action}{S}
\newcommand{\rc}{\mathsf{r}}
\newcommand{\rtot}{\mathsf{R}}

\newcommand{\kk}{{k_0}}
\newcommand{\vf}{\mathsf{v}}
\newcommand{\hamvf}{\mathsf{H}}
\newcommand{\pib}{\pi_b}
\newcommand{\Tbstar}{{}^{b}T^*}
\newcommand{\Tb}{{}^{b}T}
\newcommand{\Psibh}{\Psi_{b,h}}
\newcommand{\lspec}{\mathrm{l}\text{-}\mathrm{Spec}}
\DeclareMathOperator{\Tr}{Tr}
\newcommand{\Comp}{\mathrm{comp}}
\newcommand{\REF}{\mathsf{R}}
\newcommand{\INC}{\mathsf{I}}
\newcommand{\TRAN}{\mathsf{T}}
\def\smath#1{\text{\scalebox{1}{$#1$}}}
\def\sfrac#1#2{\smath{\frac{#1}{#2}}}

\newcommand{\WFh}{\operatorname{WF}_h}
\newcommand{\Psih}{\Psi_h}
\newcommand{\WFbh}{\mathrm{WF}_{\mathrm{b},h}}

\newcommand{\opWF}{\mathrm{WF}'}

\newcommand{\qcal}{\mathcal{Q}}
\newcommand{\pcal}{\mathcal{P}}
\newcommand{\Pcal}{\mathcal{P}}
\newcommand{\Ccal}{\mathcal{C}}

\newcommand{\F}{\mathcal{F}}
\newcommand{\shorttime}{T_1}

\newcommand{\indexset}{\mathcal{I}}
\newcommand{\indexY}{\pa}
\newcommand{\indexint}{\circ}
\newcommand{\gammabase}{{\gamma}}

\newcommand{\patan}{\pa_\flat}
\DeclareMathOperator{\fun}{Fun}

\newcommand{\Omegaq}{\Omega_{\text{per}}}
\DeclareMathOperator{\ind}{ind}
\DeclareMathOperator{\sgn}{sgn}

\DeclareMathOperator{\Id}{Id}

\title[Singular Gutzwiller]{A Gutzwiller trace formula for singular potentials}

\author{Jared Wunsch}
\email{jwunsch@math.northwestern.edu}
\address{Department of Mathematics, Northwestern University, Evanston, IL, USA}

\author{Mengxuan Yang}
\email{yangmx@princeton.edu}
\address{Department of Operations Research \& Financial Engineering, Princeton University, Princeton, NJ 08544 USA}

\author{Yuzhou Joey Zou}
\email{yzou@oakland.edu}
\address{Department of Mathematics and Statistics, Oakland University, Rochester, MI, USA}

\begin{document}

\maketitle

\begin{abstract}
  The Gutzwiller trace formula relates the asymptotic spacing of
  quantum-mechanical energy levels in the semiclassical limit to the
  dynamics of periodic classical particle trajectories.  We generalize this
  result to the case of non-smooth potentials, for which there is
  partial reflection of energy from derivative discontinuities of the
  potential.  It is the periodic trajectories of an associated branching
  dynamics that contribute to the trace asymptotics in this more general
  setting; we obtain a precise description of their
  contribution.
    \end{abstract}

\section{Introduction}
\label{sec:intro}

The Gutzwiller trace formula \cite{Gu:71} relates the asymptotic
spacing in the classical limit of the energy
levels $E_{j,h}$ of a quantum system with quantum
Hamiltonian \begin{equation}\label{P} P_h=-h^2\Lap_g+V\end{equation} to the dynamics of classical particles
moving according to the Hamilton flow of the principal symbol
$\abs{\xi}^2_g+V$.  This result, a semiclassical counterpart to the
celebrated Duistermaat--Guillemin trace formula for the trace of the
half-wave propagator \cite{DuGu:75} (cf.\ \cite{Ch:74}
and \cite{Co:72}), has received rigorous mathematical treatment by a
number of authors \cite{GuUr:89}, \cite{PaUr:91}, \cite{Me:92},
\cite{CoRaRo:99}, \cite{StZe:21}. The formula concerns the
asymptotics of a smoothed  Fourier mode 
of the
local density of states:
$$
g_\rho(E,h) := \sum_j \chi(E_{j,h}) \rho\big( \frac{E-E_{j,h}}h\big),
$$
where $\hat\rho$ is supported near the length of a family of closed classical
orbits (i.e., is isolated near a particular frequency); the leading order asymptotics of $g_\rho$ as $h \to 0$
are influenced by dynamical invariants of these orbits.

If we allow non-smooth coefficients, there are changes to the
quantum dynamics due to \emph{diffractive effects} that are not
visible to the naivest prescriptions of geometric optics.  If the
potential $V$ is non-smooth but still, say, $\mathcal{C}^2$, then the
solutions to Hamilton's equations of motion exist and are unique, but
it is known that  energy propagating in phase space, as measured
by \emph{semiclassical wavefront set}, may partially reflect off singularities
of $V$ \cite{GaWu:23a}, \cite{GaWu:21}.  In this paper, we establish a
Gutzwiller trace formula for a class of non-smooth $V$ (\emph{conormal
  potentials}), which shows
that closed trajectories that are allowed to reflect off
singularities of $V$ along a \emph{branching} flow do contribute to the oscillations in the
density of states, with amplitudes that are smaller (in terms of
powers of $h$) for smoother potentials and for
orbits with more reflections.  The detailed description of
the contributions of these orbits (Theorem~\ref{thm:trace} below)
has classical dynamical ingredients including a linearized Poincar\'e
map restricted to the symplectic orthocomplement of a closed orbit
cylinder and a Maslov factor whose interpretation as the Morse index
for a periodic (reflected) variational problem relies on the authors'
previous work \cite{WuYaZo:24}.

\emph{Our main results are as follows} (with precise statements of the
two main theorems following as Theorem~\ref{thm:poisson} and
Theorem~\ref{thm:trace} below).  For $T$ not the length of a closed
branching orbit with energy in $\supp \chi$, given any $M \in \NN$, if
$\hat \rho$ is supported sufficiently near
near $T$, 
$$
g_\rho(E,h)=O(h^M).
$$
This is a ``Poisson relation,'' that tells us that 
the trace of the Schr\"odinger propagator is nontrivial as $h
\downarrow 0$ only at times given by lengths of closed branching orbits.  The next
results concerns the asymptotics at such times.

If $T=T(E)$ is the length of a single nondegenerate closed branching
orbit
$\gamma=\gamma_E$ at each
energy $E \in \supp \chi$ then
$$
        g_\rho(E,h) \sim \frac{1}{2\pi} h^{\kk N} i^{-\sigma_\gamma}\chi(E)  \hat{\rho}(T_\gamma(E)) e^{\frac{i}{h}(ET(E)+S_\gamma)}
      \frac{T^\sharp_\gamma}{\abs{\det(I-\Pcal)}^{\frac{1}{2}}}
     \prod_{j=1}^N \frac{i^\kk J_j}{(2\xi_N^j)^{\kk+2}},
$$
        where
        \begin{itemize}
          \item $N$ is the number of reflections along $\gamma$.
        \item $\Pcal$ is the linearized Poincar\'e map.
        \item $S_\gamma$ is the classical action along $\gamma$.
        \item $T^\sharp$ it the primitive length of $\gamma$.
          \item $\sigma_\gamma$ is the Morse index of the periodic
            variational problem along $\gamma$ (see
            Appendix~\ref{appendix:periodic} and \cite{WuYaZo:24}).
            \item $\xi^j_N$ is the normal momentum at the $j$'th
              reflection.
              \item $\kk$ is given by regularity of the potential
                $V\in \mathcal{C}^{\kk-1}\backslash \mathcal{C}^\kk$ (see the discussion below).
            \item $J_j$ is a reflection coefficient given by the
                jump in the $\kk$'th normal derivative of $V$ at the
                $j$'th reflection point.
          \end{itemize}
The method of proof is a close analysis of the propagation of
singularities for and trace asymptotics of the
frequency-localized Schr\"odinger propagator
$$
\chi(P_h) e^{-itP_h/h};
$$
the connection with $g_\rho$ is given by
$$
g_\rho(E,h)= \frac{1}{2\pi} \int \hat\rho(t) \Tr (\chi(P_h) e^{-itP_h/h}) e^{iEt/h} \, dt.
$$

We now describe our hypotheses in detail, with particular
attention to the singularities of $V$, which are required to lie along
a smooth hypersurface.

Let $(X,g)$ be a smooth Riemannian manifold of dimension $n$, either compact and
without boundary or else the interior of a
scattering manifold in the sense of \cite{Me:94} (e.g.,
Euclidean space or a manifold with Euclidean ends).
Let $V: X \to \RR$ be
a potential defined on $X.$ Let $Y \subset X$ be a (possibly
disconnected) smooth compact embedded hypersurface, and assume
(as in \cite{GaWu:23a}) that
$$
V \in I^{[-1-\kk]}(Y)\subset \mathcal{C}^{\kk-1}(X)
$$
meaning that $V$ is \emph{conormal} to $Y$ and is locally given by the
inverse Fourier transform of a Kohn--Nirenberg symbol of order
$-1-\kk$ in a transverse direction to $Y$. 
Assume further that each point of $Y$ is contained in a small metric ball $U$
such that $U \backslash Y$ consists of two components $\Omega_\pm$ and
$$
V \in \CI(\overline{\Omega_\pm}).
$$
In particular, a transverse $\kk$'th derivative of $V$ exists from each side of $V$, but
may jump across $V$; tangential derivatives are all continuous. \emph{Throughout, we will take $\kk \geq
  2.$}  (An instructive example is $V=x_+^\kk
V_0$, locally near $Y$, with $V_0\in \CI(X)$ and $x$ a defining function for $Y$.) We further assume, when $X$ is
a scattering manifold, that $V$ is a symbol of positive order,
tending to $+\infty$ at spatial infinity; this makes the energy
surfaces compact.


For $y\in Y$, given $(y,\vf) \in SN (Y)\simeq Y \times \{-1,1\},$ let
$J(y,\vf)$ be the difference of $\kk$'th normal derivatives across the
interface from above and below (with orientation specified by $\vf$):
take Riemannian normal coordinates $(x_1,x')$ around $Y$, hence $Y=\{x_1=0\}$, oriented so
that $\pa_{x_1}\rvert_y=\vf$.
Then set
$$
J(y,\vf)=\pa_{x_1}^\kk(V)_+-\pa_{x_1}^\kk(V)_-
$$

Let $\hamvf_p$ denote the Hamilton vector field of
$p=\abs{\xi}_g^2+V=\sigma_h(P_h)$ with $P_h$ given by \eqref{P}.  For $E \in \RR$, let $\Sigma_{E}$
denote the characteristic set (energy surface) \begin{equation}\label{SigmaE}\Sigma_E=\{(x,\xi) \in T^*X: p(x,\xi)-E=0\}.\end{equation}

Let $\pi: T^*X\to X$ denote the projection to the base and $\pib:
T^*X\to \Tbstar X$ denote the projection map to the ``b-cotangent bundle''
discussed below in Section~\ref{section:b}, given in normal
coordinates $x=(x_1,x')$, by
$$
\pib\colon (x_1,x',\xi_1,\xi')\mapsto (x_1,x',x_1\xi_1, \xi').
$$

\begin{definition}
\label{def:bnb}
  A \emph{branching null bicharacteristic} (with energy $E$) is a potentially discontinuous
  curve $\gamma(s)$ in $\Sigma_{E} \subset T^* X$ such that at each
  $s_0$ either  
  \begin{itemize}
      \item $\gamma$ is differentiable with $\dot \gamma=\hamvf_p$ (i.e., $\gamma$ is a null bicharacteristic), or else
      \item $\pi(\gamma(s_0)) \in Y$ and there exists $\ep>0$ such
        that $\gamma$ is a null bicharacteristic on $(s_0-\ep,
        s_0)\cup (s_0, s_0+\ep),$ with $\pib\circ \gamma$ continuous across $s=s_0.$
  \end{itemize} 
  \end{definition}
  Thus at points over $Y$ where the curve is moving transversely to $Y$, the normal momentum $\xi_1$ (which is not
  constrained by the continuity of $\pib\gamma$) may jump in a manner
  consistent with the conservation of $\xi'$ and $p=\sigma_h(P)$: this
  is specular reflection (see Figure \ref{fig:branchingflow}).  Note that at points of tangency with $Y$, these are
  just ordinary bicharacteristic curves.
  
  \begin{figure}[h]
    \centering
    \includegraphics[width=0.46\linewidth]{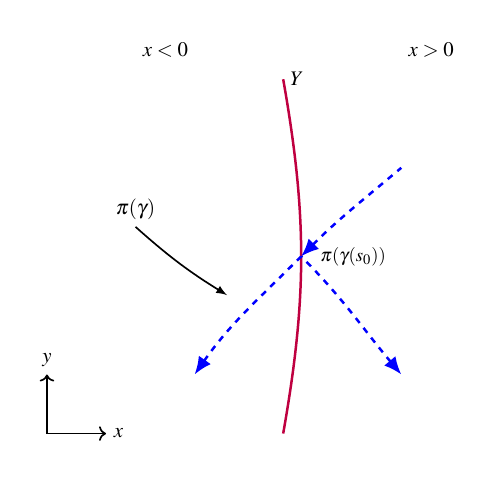}
    \quad
    \includegraphics[width=0.49\linewidth]{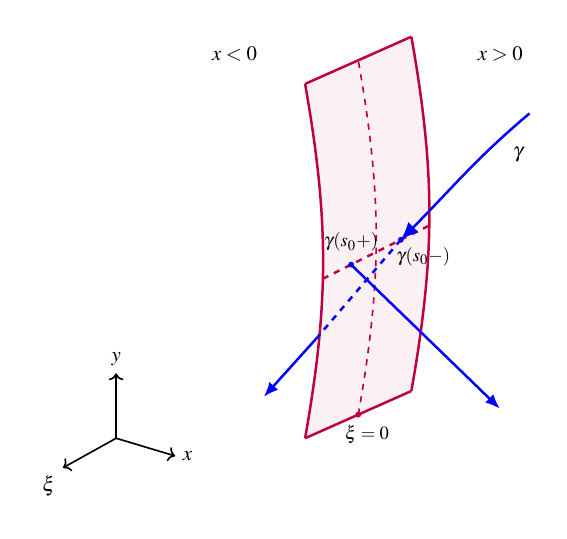}
    \caption{The picture on the left illustrates branching null bicharacteristic projected to the physical space, where $\pi(\gamma(s_0-)) = \pi(\gamma(s_0+))\in Y$. The picture on the right illustrates branching null bicharacteristic in phase space with $(x,\xi,y)$-coordinates, where there is a jump in $\xi$-coordinates in the reflective part of the branching bicharacteristic (from $\gamma(s_0-)$ to $\gamma(s_0+)$) at time $t=s_0$; note that $\eta$-variable are projected out (as there is no jump in $\eta$-variable).}
    \label{fig:branchingflow}
\end{figure}
  
  \begin{definition}
A \emph{closed branching orbit} is a periodic branching null bicharacteristic.  The \emph{segments} of a branching closed orbit are the closures of the maximal open intervals along the curve on which it is an integral curve of $\hamvf_p.$
\end{definition}

For $E \in \RR,$ $I \subset \RR$, let 
\begin{align*}
\lspec_{E} &:= \{0 \} \cup\{ \pm L: L \text{ is the period of a
        closed branching orbit in } \Sigma_E\},\\
          \lspec_{I} &:= \bigcup_{E \in I} \lspec_{E};
        \end{align*}
for $N \in \NN$
\begin{align*}
\lspec^N_{E} &:= \{0 \} \cup\{ \pm L: L \text{ is the period of a
        closed branching orbit in } \Sigma_E\\  &\quad\quad\text{with no more than $N$ reflections}\},\\
          \lspec^N_{I} &:= \bigcup_{E \in I} \lspec^N_{E}.
\end{align*}

In what
follows, we will employ a compactly supported frequency cutoff
$\chi(P_h)$ (or, as we will show is equivalent, $\chi(-hD_t)$). We will assume throughout that the
energy cutoff function $\chi$ satisfies $dp \neq 0$ on $\Sigma_E$, $E
\in \supp \chi$, as well as the following commonly-invoked dynamical
hypothesis, which seems likely to be merely technical, but vastly simplifies
the analysis by ruling out bicharacteristics moving along $Y$:
\begin{equation}\label{contact}
  \text{For every } E \in \supp \chi \text{ the bicharacteristic flow
    on } \Sigma_E \text{ makes finite order contact with } Y.
\end{equation}
This condition means that in coordinates, $\hamvf_p^\ell(x_1) \neq 0$ for some
$\ell \in \NN$.
One important consequence of this assumption is that every bicharacteristic lies over
$Y$ only at a discrete set of times.

We begin by stating the \emph{Poisson relation} for the Schr\"odinger
propagator. The notation $-0$ means $-\ep$ for all $\ep>0$.
    \begin{theorem}
      \label{thm:poisson}
      Assume the dynamical assumption \eqref{contact}
      holds for $\chi \in \CI_c(\RR).$
      If $T \notin \lspec_{\supp\chi}$ then for each $M$ there exists an open
      interval $I \ni T$ such that
      $$
\Tr \chi(P_h) e^{-itP_h/h}=O(h^M) \text{ on } I.
      $$
      If $T \notin \lspec^N_{\supp\chi}$ there exists an open
      interval $I \ni T$ such that
      $$
\Tr \chi(P_h) e^{-itP_h/h}=O(h^{(N+1)\kk-n-0}) \text{ on } I.
      $$
    \end{theorem}
\noindent (Recall that $n$ is the spatial dimension.)
    
Now we Fourier analyze the asymptotic singularities of the energy-localized
propagator via semiclassical Fourier transform, localized near a
single point in the length spectrum. For any branching
bicharacteristic $\gamma,$ let $S_\gamma$ denote the classical action
along $\gamma$---see \eqref{actiondef} below.  We say that a
closed branching orbit $\gamma$ is \emph{nondegenerate} if
the multiplicity of $1$ as an eigenvalue of the linearized first
return map is exactly $2$ (which is the smallest value possible, owing
to the existence of orbit cylinders); see Section~\ref{sec:poincare} for details.

For the following analysis of the singularities of the trace, we
assume further that:
\begin{enumerate}
  \item For a time interval 
  $J \subset (0,\infty)$
  the only closed branching orbits in
  $\Sigma_E$ for $E \in\supp \chi$ with period in $J$ are nondegenerate
  (hence, by compactness of the energy surface, finite in number), and
  are simple or iterates of a simple bicharacteristic.
\item
  There is at most one closed branching orbit of
  energy in $\supp\chi$ and length in $J$ passing through any
  given point in $\Tbstar X$.
\item
  No point along one of these trajectories is conjugate to itself in
  the sense of Definition~4.6 of \cite{WuYaZo:24}.
\end{enumerate}
    \begin{theorem}
    \label{thm:trace}  Assume the dynamical assumption \eqref{contact}
    holds.
      Let $\hat\rho \in \CI_c(\RR)$ be supported in $J$ containing
      only lengths of orbit cylinders satisfying the hypotheses
      enumerated above.
      Then the inverse Fourier transform of the localized wave
      trace\footnote{We use the normalization conventions of
        \cite{Ho:90} for the Fourier transform.}
      is
      \begin{equation*}
      \begin{split}
          g_\rho(E,h):= & \frac{1}{2\pi} \int \hat\rho(t) \Tr (\chi(P_h) e^{-itP_h/h}) e^{\frac{i}{h}Et} \, dt
        \\ \sim & \sum_{\gamma}
            \frac{1}{2\pi} h^{\kk N} i^{-\sigma_\gamma}\chi(E)  \hat{\rho}(T_\gamma(E)) e^{\frac{i}{h}(ET(E)+S_\gamma)}
            \frac{T^\sharp_\gamma}{\abs{\det(I-\Pcal)}^{\frac{1}{2}}}
            \prod_{j=1}^N \frac{i^\kk J_j}{(2\xi_N^j)^{\kk+2}}
      \end{split}
      \end{equation*}
      where the sum is over all orbit cylinders $\gamma$ with lengths in $J$; 
      $N$ is the number of diffractions along $\gamma$; $\gamma_1,\dots, \gamma_N$ are segments of the closed branching orbit 
      $\gamma$ with length $T_\gamma$ and primitive length
      $T^\sharp_\gamma$; $k_0$ is given by the regularity of the
      potential; $\sigma_{\gamma}$ denotes the Morse index of the
      closed branching orbit; $\Pcal$ is the Poincar\'e map defined in Section \ref{sec:poincare};
      $\xi_N^j$ denotes the normal momentum to $Y$ at the $j$'th
      reflection and $J_j$ denotes the value of $J(z,v)$ at the
      point and direction of contact.
\end{theorem}
Note that since the bicharacteristic lives on the energy surface, we
may replace $\xi^j_N$ by
$$
(E-V(x^j))^{1/2}\cos \theta_j
$$
where $V(x^j)$ denotes the potential evaluated at the $j$'th point of contact
with $Y$ and $\theta_j$ denotes the angle made with the normal at the
point of contact.

See Remark~\ref{remark:mismatch} below for a discussion of the
relationship (or lack thereof) of the powers of $h$ in Theorems~\ref{thm:poisson} and \ref{thm:trace}.

Of the dynamical hypotheses made on the orbit cylinders, the nondegeneracy is
essential, while the additional hypotheses of non-self-conjugacy and
simplicity seem to be merely technical here, and could probably be relaxed.
The latter hypothesis, as well as the requirement that there be at
most one orbit cylinder passing through a given point, could be
relaxed by employing microlocalized propagators (as in
Section~\ref{sec:multiplereflection}) rather than passing to the
ordinary propagator microlocalized only at start- and end-points in
computing traces. The former could be addressed by extending our
ansatz for the propagator to allow for non-projectable semiclassical
Lagrangian distributions.  We leave both of these extensions for
future work.

\vspace{-1.5mm}

\section*{Related work}
This work relies on the previous analysis of semiclassical Schr\"odinger equations with
conormal coefficients of the first author and Oran Gannot
\cite{GaWu:23a}, \cite{GaWu:21}, which was in turn influenced by prior work of De
Hoop--Uhlmann--Vasy on wave equations with such coefficients
\cite{DeUhVa:15}.  Prior results in the setting of Schr\"odinger
operators were only available in one dimension \cite{Be:82}.
Recently, Demanet--Lafitte \cite{demanet2023reflection} have obtained
results on reflection coefficients for semiclassical problems at a
conormal interface that are closely related to the results of
Section~\ref{sec:ref-pgtr} below; their paper also provides
some discussion of various physical motivations of such models.  The
third author \cite{Zo:24} has moreover done a calculation in the explicit example
of ``bathtub potentials'' in one dimension, with singularities of the
form $cx_+^2$, which illustrates the influence of potential
singularities on spectral asymptotics analyzed here in a more general setting.

One might hope to prove a trace formula in this context using the
global Fourier integral operator tools of Duistermaat--Guillemin
\cite{DuGu:75}, as was done in the smooth case e.g.\ by Meinrenken
\cite{Me:93}.  The singularities of the propagator at the interface
$Y$, however, makes this approach seem quite difficult: the branching
propagation of singularities means that the description of the
propagator over $Y$ is considerably more complicated.  We have a more
or less explicit description at hyperbolic points, where we compute
reflection coefficients, but near glancing points all we have is
energy estimates.  We therefore employ an alternate approach that
decomposes the trace via a partition of unity in phase space, and
depends only on piecing together the descriptions of the propagator
for short times.  Crucially, we are able to use the cyclicity of the
trace to ``push away'' the boundary contributions to the trace into
the interior, where we are able to analyze it as a
semiclassical Fourier integral operator.  

Among the technical
novelties here are: the description of the propagation of
singularities for the Cauchy problem, which requires establishing the
b-microlocality of $\chi(P_h)$; the computation of the reflection
coefficients for our problem via a delicate analysis of high-order
transport equations; the analysis of the dynamics of a novel branching
flow, which has to be filtered by the number of branching points; the
employment of the propagation argument to push the trace computation
away from the boundary; and the recognition of the coefficients
arising in the final trace computation in terms of dynamical
quantities, especially the linearized Poincar\'e map (which here,
unlike in \cite{DuGu:75} is for an \emph{orbit cylinder} rather than
an isolated orbit) and the Maslov/Morse index (which employs the
authors' previous work on the dynamics of reflected bicharacteristics
in \cite{WuYaZo:24}).

\vspace{-1.5mm}

\section*{Outline}
The paper is structured as follows.  We begin
(Section~\ref{sec:prelim}) with a review of
propagation of singularities results for the Schr\"odinger equation
with conormal potentials.  This entails the introduction of techniques
involving the b-calculus of pseudodifferential operators.  We
obtain some purely dynamical results on the branching flow along which
singularities globally propagate.  We then
obtain results for propagation of singularities for the Cauchy
problem (rather than spacetime singularities).
Section~\ref{sec:poincare} contains further dynamical preliminaries:
an account of the relationship of the Poincar\'e map for an orbit
cylinder to the Hessian of the action that will eventually occur in
our stationary phase computations.  In Section~\ref{sec:int-pgtr} we
study the structure of the semiclassical Schr\"odinger propagator,
first recalling its form in the smooth case (i.e., away from $Y$), and then obtaining a
parametrix for the propagator along a singly-reflected
bicharacteristic near $Y$.  Section~\ref{sec:multi} extends this
construction to allow for long times and multiple reflections.  In
Section~\ref{sec:poisson} we obtain the Poisson relation
(Theorem~\ref{thm:poisson}) by decomposing the trace using a
microlocal partition of unity.  Finally in Section~\ref{sec:tracenew} we
prove the trace formula, Theorem~\ref{thm:trace}, by using a more
refined partition of unity that enables us to push the trace
computations away from $Y$.  Appendices cover
some further dynamical ingredients: the modifications to the results
of \cite{WuYaZo:24} required to deal with the variational problem of
periodic branching bicharacteristics without fixed period, and the
composition of van Vleck determinants arising in stationary phase computations.

\vspace{-1.5mm}

\section*{Acknowledgements} The authors thank Dmitri Vassiliev for helpful conversations,
and for introducing them to the terminology ``branching bicharacteristics.''
The contributions of Oran Gannot to this paper are fully at the level
of a coauthor, as he is responsible for the computation in
Section~\ref{sec:ref-pgtr} of the reflection coefficients, and the
authors are grateful for his permission to use this work.

The first author thanks the NSF for partial support from the
grants DMS-2054424 and DMS-2452331, and the Simons Foundation, for partial support
from the grant MPS-TSM-00007464. The second author acknowledges support from the AMS-Simons Travel Grant and NSF grant DMS-2509989.

\section{Propagation of singularities}
\label{sec:prelim}

\subsection{Geometry of branching rays and propagation of semiclassical singularities}

Propagation of singularities for the stationary Schr\"odinger equation
$$
(-h^2\Lap+V-E)u=0
$$
with a potential $V$ as given above was previously treated in
\cite{GaWu:23a}.  This was then extended to the setting of the
time-dependent Schr\"odinger equation as discussed here in
\cite{GaWu:21}.  In this section, we review the definitions and propagation results
that will be needed below.

We define the spacetime manifolds $$M = \RR\times X,\quad Y_M=\RR\times Y.$$

Let $$u\in L^\infty (\RR; L^2(X))$$ solve the time-dependent Schr\"odinger equation
$$
h \pa_t u =-iP_h u,
$$
i.e.,
$$
Q_h u := (hD_t +P_h )u=0.
$$
Let $$q=\sigma_h(Q_h) = \tau +\abs{\xi}_g^2+V(x)$$ denote the
symbol of the time-dependent semiclassical Schr\"odinger operator, and
let
$\Sigma=\{q=0\}$ denote its characteristic set.

We consider the semiclassical wavefront set
$$
\WFh u\subset T^*X.
$$
Away from $Y$, where $V$ is smooth, standard results in semiclassical analysis (cf.~\cite[Appendix~E]{DyZw:19}) constrain this wavefront
set: we know that
$$
\WFh u \subset \Sigma
$$
and that $\WFh u$ is invariant under the flow generated by the
Hamilton vector field
$$
\hamvf_q = \pa_t +\hamvf_p
$$
with $$p=\abs{\xi}^2_g +V(x)$$ the symbol of the stationary operator.  The same holds for $\WFh^k u$ for each
$k$, where we recall that by definition,
$$
\mu \notin \WFh^k u
$$
iff there exists $A \in \Psih(\RR\times X)$, elliptic at $\mu$ with
$$
A u =O_{L^2}(h^k).
$$
Note that $\tau$ is conserved under the
$\hamvf_q$-flow.  Let us study the portion of the characteristic set in $\tau=-E$,
and turn to the question of what happens over $Y_M$.

Fix Riemannian normal coordinates $(x_1,x')$ near $Y$ so that locally $$g =dx_1^2 + k_{ij}(x_1,x')
dx_i' d x_j'.$$
Thus, setting
$$
r(x,\xi',E) = E-V(x) -\ang{K(x)\xi',\xi'}
$$
with $K$ the inverse matrix to the positive-definite matrix $k_{ij}$,
we have
$$
\begin{aligned}
p(x,\xi)-E&= (\xi_1)^2 + \ang{K(x) \xi',\xi'} + V(x)-E= (\xi_1)^2
-r(x,\xi',E),\\
q(x,\xi) &= \tau+ (\xi_1)^2 + \ang{K(x) \xi',\xi'} + V(x)= (\xi_1)^2
-r(x,\xi',-\tau)
\end{aligned} $$
We respective define the \emph{elliptic}, \emph{glancing}, and
\emph{hyperbolic} sets in $T^*Y$ as
\begin{gather}
    \hyp_E=\{(x',\xi')\colon r(0,x',\xi',E)>0\}\subset T^*Y,\\
    \gl_E=\{(x',\xi')\colon r(0,x',\xi',E)=0\}\subset T^*Y,\\
    \ellip_E=\{(x',\xi')\colon r(0,x',\xi',E)<0\}\subset T^*Y.
\end{gather}
Note that these are projections from $T_{Y}^*X$ of covectors that are
respectively transverse to $Y$ and in $\Sigma_E$, tangent to $Y$ and in
$\Sigma_E$, and outside $\Sigma_E$ (with $\Sigma_E$ given by \eqref{SigmaE}).

Define the lift of $\hyp_E$ to $T^*Y_M$ by 
\[
\hat \hyp_E = \{(t,x',-E,\xi') : (x',\xi') \in \hyp_E\} \subset T^*Y_M,
\]
with the analogous definition for $\hat \gl_E$. 
We also write $\iota^* : T^*_{m} M \rightarrow T^*_m Y_M$ for the
canonical projection whenever $m \in Y_M$.

The main result about propagation across the interface $Y$ is
given by the following theorem.  The first part say that $h^s$ wavefront set
at a hyperbolic point can be caused by direct propagation of $h^s$ wavefront set, or by a
reflection of a point in $h^{s-\kk}$ wavefront set; thus, the
solution gains $h^{\kk}$ regularity upon reflection.  The second
part of the theorem says that at glancing points, where the flow is
tangent to $Y$, singularities propagate along ordinary
bicharacteristics, unaffected by diffraction (see Figure~\ref{fig:bichar}).  Note in particular that
this rules out the possibility that semiclassical singularities stick
to $Y$ as it curves.
\begin{theorem}{(\cite{GaWu:21})}\label{theo:diffractiveimprovements}
	Let $w = w(h)$ be $h$-tempered in $H^1_{h,\loc}(M)$ such that $Q_h w = \mathcal{O}(h^\infty)_{L^2_\loc}$.
	\begin{enumerate} \itemsep6pt 
		\item 
		If $\mu_0 \in \hat \hyp_E$, let $\mu_\pm \in \{q=0,
                \tau = -E\}$ be the preimages of $\mu_0$ under $\iota^*$ with opposite normal momenta. If $\mu_+ \in \WF_h^s(w)$
                for some $s\in \RR$, then there exists $\varepsilon >
                0$ such that 
		\[
		\exp_{-t'\hamvf_{q}}(\mu_+) \subset \WF_h^s(w), \text{ or } \exp_{-t'\hamvf_{q}}(\mu_-) \subset \WF_h^{s-\kk}(w),
		\]
		or both, for all $ t'\in (0,\varepsilon)$.
		
              \item If $\mu_0 \in \hat \gl_E$, let
                $\mu \in \{q=0, \tau = -E\}$ be the unique preimage of
                $\mu_0$ under $\iota^*$ (necessarily with vanishing normal
                momentum). If $\mu \in \WF_h^s(w)$ for some
                $s \in \RR$, then there exists $\varepsilon > 0$ such
                that
		\[
		\exp_{-t'\hamvf_{q}}(\mu) \subset \WF_h^s(w)
		\] 
		for all $t' \in (0,\varepsilon)$.
	\end{enumerate}
      \end{theorem}

\begin{figure}[h]
    \centering
    \includegraphics[width=0.45\linewidth]{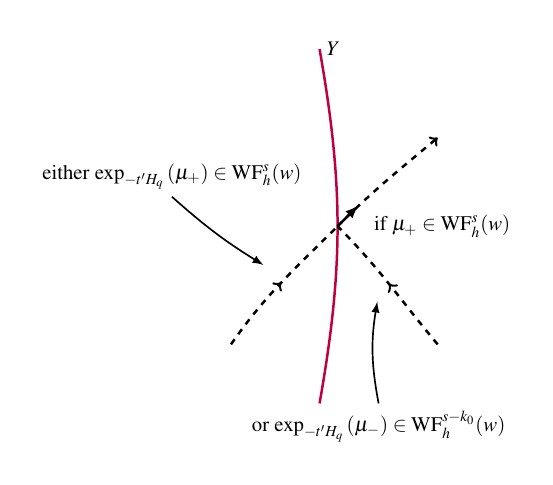}
    \quad
    \includegraphics[width=0.45\linewidth]{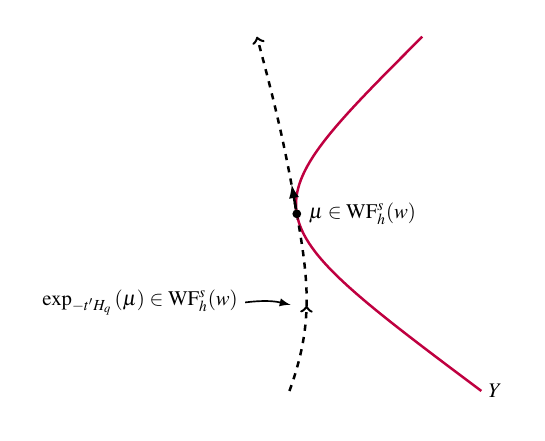}
    \caption{The picture on the left illustrates part (1) of the theorem, and the picture on the right illustrates part (2) of the theorem. Red curves are the hypersurface $Y$ where the conormal singularities are. }
    \label{fig:bichar}
\end{figure}

(Note that we also have an \emph{elliptic} estimate on $\ellip_E$,
analogous to Proposition~7.5 of \cite{GaWu:23a}; here the regularity
does not propagate, but we simply have an a priori estimate on $u$
limited by the regularity of $V$.)

A consequence is that regularity globally propagates forward along ordinary
bicharacteristic flow except at such times as it hits $\hat\hyp_E,$
where reflected singularities from other bicharacteristics may come
into play (i.e., branching may occur).
\begin{corollary}{(\cite{GaWu:21})}\label{cor:propagationawayfromhyp}
	Let $w = w(h)$ be $h$-tempered in $H^1_{h,\loc}(M)$ such that $Q w = \mathcal{O}(h^\infty)_{L^2_\loc}$. Let $s\in \RR$ and $T \geq 0$. If 
 \[
 \mu \in \{q = 0, \, \tau = -E\}
  \]
 is such that $\exp_{-t'\hamvf_{q}}(\mu)$ is disjoint from $(\iota^*)^{-1}(\hat \hyp_E)$ for each $t' \in [0,T]$, then
 \[
 \exp_{-T\hamvf_{q}}(\mu) \notin \WF_h^s(v) \Longrightarrow \mu \notin \WF_h^s(v).
 \]
\end{corollary}



  \subsection{\mbox{b}-Geometry and propagation of b-wavefront set}\label{section:b}

Semiclassical wavefront set is the natural tool for describing the
smoothing effect of reflected propagation.
There are, however, limitations to our ability to study global effects
by ordinary semiclassical microlocalization in this setting. In part, this
arises because, as alluded to above, the elliptic estimate (which we
have not stated here, but is
analogous to Proposition~7.5 of \cite{GaWu:23a})
 would not allow us to conclude that for solutions to
$Qu=0$, there is no wavefront set over elliptic points over $Y$: the
conormal singularities of the potential create semiclassical
singularities in the solution.  This defect is remedied, however, if
we study instead the \emph{semiclassical b-wavefront set}, which is
adapted to the conormal singularities that we deal with here.
  
We start by recalling the basics of \emph{b-geometry} and analysis relative to the
interior hypersurface $Y \subset X$; the more common setting, as described e.g.\
in \cite{Me:93}, would have $Y=\pa X$.  The concepts of b-geometry and
analysis were introduced by Richard Melrose, and we refer the reader to
his \cite{Me:93} for a comprehensive introduction, together with
applications in elliptic PDE.  The semiclassical b-calculus used here
received its first treatment (to our knowledge) in \cite[Appendix
A]{HiVa:18}.
  See \cite[Section
3]{GaWu:23a} for an introduction to the semiclassical b-calculus with
respect to an interior hypersurface relevant to the presentation here.

We let $\mathcal{V}_b(M,Y_M)$ denote the Lie algebra of vector fields on $M$
tangent to $Y_M$; in normal coordinates, these are just the
$\mathcal{C}^\infty(M)$-span of $x_1 \pa_{x_1}, \pa_{x'},\pa_t$.  They are the
space of sections of a vector bundle, denoted $\Tb(M,Y_M).$  The dual
bundle, denoted $\Tbstar(M;Y_M),$ has as smooth sections the
$\mathcal{C}^\infty(M)$-span of $dx_1/x_1, dx', dt$.  Every smooth one-form is
a b-one-form, so there is a canonical projection map
$$
\pib: T^* M \to \Tbstar (M, Y_M);
$$
in normal coordinates near $Y_M$, then,
$$
\pib (\xi_1 \, dx_1 +\xi'\, dx'+\tau\, dt) = (x_1 \xi_1 \, \frac{dx_1}{x_1}+\xi'\, dx' +\tau\, dt ),
$$
i.e., in canonical dual coordinates (and using bars on
$T^*M$ coordinates to distinguish them from the $\Tbstar M$ coordinates)
$$\begin{aligned}
\xi_1\circ \pib(x_1,x',t,\bar\xi_1,\bar\xi',\bar\tau)&= x_1\bar\xi_1,\\
\xi'\circ \pib(x_1,x',t,\bar\xi_1,\bar\xi',\bar\tau) &= \bar\xi',\\ \tau\circ \pib(x_1,x',t,\bar\xi_1,\bar\xi',\bar\tau) &= \bar\tau.
\end{aligned}$$

An analogous story holds in the time-independent setting, with $\pib:
T^*X \to \Tbstar X$, and we will use this version as well.

The
\emph{semiclassical b-differential operators} are those that are sums
of smooth coefficients times products of $h x_1 D_{x_1}$, $h D_{x'}, hD_t$. The
algebra of \emph{semiclassical b-pseudo\-differential operators,}
denoted $\Psibh(M, Y_M)$, microlocalizes this algebra, and elements of
this algebra can be formally written as
$$
a(x_1,x',t,h x_1 D_{x_1}, h D_{x'}, hD_t),
$$
where $a$ lies in a suitable symbol space.  Associated to this
calculus of pseudodifferential operators is a notion of
\emph{b-wavefront set} denoted $\WFbh u$.  A quantitative version is
$\WFbh^s u$: we define
$$
\mu \notin \WFbh^s u
$$
iff there exists $A \in \Psibh(M, Y_M)$, elliptic at $\mu$, such that
$A u=O_{L^2}(h^s)$.  When $s$ is omitted, it is taken to be
$s=\infty$.  In \cite{GaWu:23a}, a variant of this is
employed, where regularity is measured with respect to $H^1_h$ instead
of $L^2$.  This is very natural from the point of view of using
propagation estimates, where we constantly use the quadratic form
$\ang{P_h u,u}$, but as remarked in Section 3.5 of that paper (prior to
Lemma 3.1), is
immaterial from the point of view of stating the main results.
Analogous constructions of course exist (and will be employed below)
with the manifolds $(M,Y_M)$ replaced by $(X,Y)$ (i.e., fixing time).

\emph{For brevity of notation, in what follows, we will drop the $Y_M$
  from notations such as $\Tbstar (M,Y_M)$ resp.\  $\Psibh(M,Y_M)$ and
  refer simply to $\Tbstar M$ resp.\ $\Psibh(M)$ without any danger of ambiguity.}

We would like to have
\begin{equation}\label{cpctWF}
\WFbh^s u=\emptyset \text{ iff } u=O_{L^2}(h^s),
\end{equation}
but this entails an extension of the wavefront set to fiber-infinity,
(so as to deal with examples such as $e^{ix/h^2}$, even for the
ordinary semiclassical calculus).  Conveniently, however, we will only
be dealing with \emph{compactly-microsupported} functions $u$.
\begin{definition}
A semiclassical family $u$ is compactly microsupported if there exists
$A \in \Psibh(M)$ with compactly supported total symbol (thus also
said to be ``compactly microsupported'') such that
$u=Aw+O_{L^2}(h^\infty)$ for some $w \in L^2$.  The definition on the
spatial manifold $X$
is analogous.
  \end{definition}
The point here is that \emph{if $u$ is compactly microsupported, then \eqref{cpctWF} does hold.}
  
Note that $Q$ is not a b-differential operator, as $h^2 D_{x_1}^2$ is not
the square of a b-vector field; this part of the operator is singular
with respect to the b-structures we have introduced.
\begin{definition}
Let $$\dot \Sigma_b(q)=\pib(\Sigma)\subset \Tbstar M$$  (recall that $\Sigma=\{q=0\}
\subset T^*M$); we call this the
\emph{compressed} characteristic set, and equip it with the subspace
  topology as a subset of $\Tbstar M$.  Likewise, for any $E \in \RR$ we let
  $$
\dot\Sigma^E_b(p)=\pib(\{p-E=0\}) \subset \Tbstar X.
  $$
  \end{definition}
     The projection is a diffeomorphism away from $Y_M$, hence in this
     region the
     compressed characteristic set is just a copy of the ordinary one,
     but on the other hand 
\begin{gather*}
\dot\Sigma_b(q)\cap \Tbstar_{Y_M} M = \big\{x_1=0,x',t ,\xi_1=0,
\xi',\tau): \tau+\abs{\xi'}^2_g +V(x) \leq 0\big\}\\
\dot\Sigma^E_b(p)\cap \Tbstar_{Y} X = \big\{x_1=0,x',\xi_1=0,
\xi'): -E+\abs{\xi'}^2_g +V(x) \leq 0\big\}.
\end{gather*}

Now we consider bicharacteristics in the b-setting.  First, we note
that we can easily extend Definition~\ref{def:bnb} to the spacetime
setting, by letting a branching spacetime null bicharacteristic be a
curve in $\Sigma \subset T^*M$ tangent to $\hamvf_q$ away from $Y_M$
and continuous at $Y_M$ after application of $\pib$.  The compressed
branching bicharacteristics, in $X$ and $M$, are just the continuous
curves obtained by
projection
to the b-cotangent bundles of the counterparts in $T^*X$ resp.\ $T^*M$:
  \begin{definition}   
  If $\gamma$ is a branching (spacetime) bicharacteristic then
  $\pib(\gamma)$ is a \emph{compressed} branching (spacetime) bicharacteristic. 
\end{definition}
One virtue of this definition is that the compressed branching bicharacteristics
are continuous, since the normal momentum (here denoted $\overline{\xi}_1$) which jumps at points of reflection
(see Figure~\ref{fig:branchingflow}),  is
zeroed out by the projection map $\pib$.\footnote{Note that the first
  author's previous work with Gannot, e.g., \cite{GaWu:23a}, used the
  terminology ``generalized broken bicharacterstics'' (``GBB'') for
  the most general curves along which b-wavefront set propagates.  The
  branching curves considered here are a special case of these GBBs
  that takes into account the lack of ``sticking'' at glancing points
  proved in \cite{GaWu:23a}; the only multivalued aspect is thus the
  bifurcation of the flow at hyperbolic points.  We have adopted the
  ``branching'' terminology to match that used by Vassiliev--Safarov \cite{VaSa:88}.}

The main propagation result in the b-setting can now be readily
described.
\begin{theorem}\label{theorem:propofsings}
Assume the dynamical assumption \eqref{contact} holds.  	Let $w = w(h)$ be $h$-tempered in $H^1_{h,\loc}(M)$ such that
        $Q w = \mathcal{O}(h^\infty)_{L^2_\loc}$.  Then $\WFbh (w)
        \subset \dot \Sigma_b(q)$.  Moreover, for $\mu \in \dot
        \Sigma_b(q)$, and for any $\ep>0$, if there does not exist $\mu'$ with
        $t(\mu')\in (t(\mu)-\ep, t(\mu))$ such that $\mu' \in \WFbh^s(w)$ and with
        $\mu'$ and $\mu$ lying on a single compressed branching
        bicharacteristic, then $\mu \notin \WFbh^s(w).$
  \end{theorem}
In other words, $\WFbh (w)$ propagates along the \emph{branching}
bicharacteristic flow: turned around in time, this says that a point
in the wavefront set of $w$ continues along \emph{some} branching
bicharacteristic.  Note that we have not included the improvement
along reflected trajectories in this crude statement; we return to
this suppression of reflected waves below.

We make some remarks on the proof, which is essentially in
\cite{GaWu:23a}, but is packaged there in different (primarily, more general) form.

\noindent\emph{Sketch of Proof.}
The proof of Theorem~\ref{theorem:propofsings} mainly follows verbatim the proofs
of the corresponding propositions in \cite{GaWu:23a} (see in
particular Theorem~1); the minor
modifications needed
 are explicitly addressed in the Appendix of \cite{GaWu:21}.  The one major
 change we have made here, however, concerns the behavior at the
 glancing set, where the geometry of rays is potentially subtle, and
 where the fact that we have stated the theorem not for generalized
 broken bicharacteristics as in \cite{GaWu:23a} but rather for the
 compressed branching bicharacteristics discussed here, makes a difference.
  Theorem~1 of \cite{GaWu:23a} applies down to quite a low regularity of
 $V$, being valid e.g.\ for $V=(x_1)_+^\alpha$ for all $\alpha>0$;
 correspondingly, the propagation at glancing points is in principle
 along all ``generalized broken bicharacteristic curves,'' which may
 be permitted to stick to the interface $Y$.  Here, however, we work
 in higher regularity so that Theorem~3 of \cite{GaWu:23a}
 additionally applies. This latter result concerns \emph{ordinary}
 semiclassical wavefront set rather than b-wavefront set, and says that
on the (unique!) bicharacteristic curve $\gamma(s)$ through a
 glancing point $\gamma(0)$, if $\WF_h^r (w) \cap \gamma((-\ep,0))$ is
 disjoint from $\WF_h^r(w)$ then $\gamma(0) \notin \WF_h^r(w)$.  Thus,
 the ordinary wavefront set propagates along ordinary
 bicharacteristics at such points, with no limits on the order $r$.

 Owing to our dynamical assumption \eqref{contact}, if $\gamma(0) \in
 \gl_E$ for some $E$ then for $\ep>0$ small enough, $\gamma((-\ep,0))$
 is away from the interface $Y$, hence in a region where $\WFbh^s(w)$
 and $\WF_h^s(w)$ agree.  Moreover, at a glancing points $\mu$ over
 $Y,$ absence of $\WF_h^s(w)$ \emph{implies} absence of $\WFbh^s(w)$
 by Proposition~7.10 of \cite{GaWu:23a} (again, cf.\ remarks in the
 Appendix of \cite{GaWu:21}).\qed
  
In order to make more global statements about propagation of
singularities, we now define a relation along the relevant compressed branching
bicharacteristic flow, both in the spatial variables and in spacetime.
\begin{definition}
For $\nu \in \Tbstar X$ let
\begin{align*}
\XPhi{E}_t(\nu) =\{\nu': &\text{ there exists a compressed branching
                        bicharacteristic in }\dot\Sigma^E_b\\ &\text{ with } \gamma(0)=\nu,\ \gamma(t)=\nu'\}.
\end{align*}
For $\mu \in \Tbstar M$ let
\begin{align*}
\Phi_T(\mu) =\{\mu': &\text{ there exists a compressed branching
                        bicharacteristic}\\ &\text{ containing
                                              }\mu,\mu',\  \text{with
                                              } t(\mu')=T \}.
\end{align*}
                                            \end{definition}
In both cases, the flow is thus allowed to branch via reflection or continuation
upon arriving transverse to the interface $Y$: these maps are
relations, not functions.  
Note that fixing the energy $E$ is
important to make the branching of the flow $\XPhi{E}$ discrete: if we did not fix
the energy, then at times when the flow is over $Y$, its location in
the b-cotangent bundle has ``forgotten'' about the normal momentum
(since we simply have $\xi_1=0$) and it is only the constraint on the
energy that determines the normal momentum of the reflected or
transmitted continuation.  By contrast, away from $Y$ the energy is
specified by a point in $\Tbstar X$, and we will omit the $E$ from the
notation in considering flowouts in this region.

We introduce a further refinement of the notation for the flow that
will be of use in our parametrix constructions below.  Near a
branching bicharacteristic $\gamma$ whose endpoint is away from $Y$
there is a locally-defined single-valued flow map, where we define the
flowout of a perturbation of $\gamma(0)$ to
be the flow along the bicharacteristic that reflects where $\gamma$
did and is transmitted where $\gamma$ was transmitted (where $\gamma$ was glancing).  We denote this
locally-defined flow $\XPhi{E}^\gamma_t$.

\begin{remark}\label{remark:inscribed}
To motivate what follows, consider the example of the Euclidean plane with $Y=S^1$ and
$V=(\abs{x}-1)_+^{\kk}.$  Regular polygons inscribed in $S^1$ are projections
of compressed branching bicharacteristics; we can find a family of these
limiting to the ``gliding'' curve that circulates around $Y.$
This latter curve is not (the projection of) a compressed branching bicharacteristic
curve, and semiclassical singularities do not propagate along it.
Correspondingly, the polygonal curves approaching $S^1$ must have more
and more reflections; since under Schr\"odinger propagation each reflection gains a factor of
$h^{\kk}$ by Theorem~\ref{theo:diffractiveimprovements}, these
approximating curves can only carry less and less energy.
\end{remark}
If we are interested in only wavefront set up to a given semiclassical
order, then, it makes sense to filter the flow lines by maximum number
of reflections.
\begin{definition}
\label{def:branchingflow}
For $\nu \in \Tbstar X$ and $N \in \NN$ let
\begin{multline*}
\XPhi{E}^N_t(\nu) =\{\nu': \text{ there exists a compressed branching
  bicharacteristic in } \dot \Sigma^E_b  \text{ with }\\ \text{at most $N$ reflections and } \gamma(0)=\nu,\ \gamma(t)=\nu' \},
\end{multline*}
and for $\mu \in \Tbstar M$ let
\begin{multline*}
\Phi^N_T(\mu) =\{\mu': \text{ there exists a compressed branching
  bicharacteristic with }\\ \text{at most $N$ reflections between }
\mu,\mu',\ t(\mu')=T \}.
\end{multline*}
\end{definition}
Then Theorem~\ref{theo:diffractiveimprovements} yields the following
description of propagation of regularity along both this finitely-reflected
flow and the full flow.  We reiterate that the importance of the former is
that limits of finitely-reflected rays include gliding rays along $Y$;
while no singularities propagate along these, they are nonetheless in
the closure of the flow if we allow infinitely many reflections, and
this complicates the study of the dynamics.
\begin{corollary}\label{cor:propofsings}
Assume the dynamical assumption~\eqref{contact} holds.  	Let $w = w(h)$ be $h$-tempered in $H^1_{h,\loc}(M)$ such that
        $Q w = \mathcal{O}(h^\infty)_{L^2_\loc}$ and
        ${w(h)} \in L^2_{\loc}$ uniformly for all $h$.
  Let $\mu \in \dot\Sigma_b(q)$ and suppose $s< (N+1)\kk$, and
  $$
  \Phi^N_{T}(\mu) \cap \WFbh^{s} (w) =\emptyset.
  $$
  Then $$\mu
\notin\WFbh^{s} (w).
$$

Likewise, if
  $$
  \Phi_{T}(\mu) \cap \WFbh (w) =\emptyset.
  $$
  Then for all $s \in \RR$, $$\mu
\notin\WFbh^s (w).
$$

\end{corollary}
\begin{remark}
  To understand the numerology of the first part of the theorem, note
that keeping
track only of the $N$-fold branching flow leaves open the possibility
that branching bicharacteristics with $N+1$ or more reflections reach
$\mu$; since we have $L^2$ background regularity, and the
singularities along these curves gain $h^{\kk}$ with each reflection,
after $N+1$ reflections these contributions are $O(h^{(N+1)\kk-0})$,
hence do not contribute to $\WFbh^s(w)$ for $s$ in the given range.

We remark that there is a subtle difference between $\mu \notin\WFbh^s
(w)$ for all $s$, and $\mu \notin\WFbh (w)$; note that
$$
\WFbh(w) =\overline{\bigcup_s\WFbh^s(w)},
$$
and that the statement is false without the closure.
\end{remark}
\begin{proof}
We will in fact prove a sharper result than the one stated above (at
the cost of a slightly more complex statement).  Let
$\Psi_T^j(\mu)$ be the subset of $\Phi_T^j(\mu)$ consisting of points
that have undergone \emph{exactly} $j$ reflections.  The sharp result
we will prove is that
\begin{equation}\label{sharptheorem}
\Psi_T^j(\mu) \cap \WFbh^{s-j\kk}w=\emptyset,\ j=0,\dots, \lceil
\frac s{\kk}\rceil-1\Longrightarrow \mu \notin \WFbh^sw.
\end{equation}
Note that  $N :=\lceil \frac s{\kk}\rceil-1$ is more easily
described as the greatest
integer strictly smaller than $s/\kk$.  Thus the stated theorem
follows from this sharper version, since (giving up the subtlety of
different regularity hypotheses on the different sets $\Psi_T^j(\mu)$) the hypothesis implies that
$$
\Phi_T^N(\mu) \cap \WFbh^{s} w=\emptyset
$$
with the given choice of $N$.  The restriction in the hypothesis that $N$ is (at least) the largest integer less
than $s/\kk$ is equivalent to $N+1>s/\kk$, which is the
restriction on $s$ in the hypothesis
of the theorem.  Note that of course $\Phi_T^N(\mu)=\bigcup_{j=0}^N \Psi_T^j(\mu).$

To prove this sharper result, let us consider the case $T<0$, as this
is the usual mode in which we apply the theorem (and the case $T>0$
follows by time-reversal symmetry).  The result holds tautologically
for $T=0$ so let $T_0$ be the supremum of $T'$ such that
\eqref{sharptheorem} holds for all $T \in [-T',0]$.  We will show that
$T_0=\infty$, by assuming it is finite and deriving a contradiction.

Given fixed $\mu$, we may now fix $\ep>0$ so that along all
backwards bicharacteristics from $\mu$ with up to $j$ reflections, at most one reflection can
occur between times $-T_1-\ep$ and $-T_1$: this follows from the
finite-order contact assumption, since at both hyperbolic and glancing
points the flow leaves the boundary immediately.  Moreover, we may
obtain this $\ep>0$ locally uniformly as $\mu$ ranges over an open
set, owing to the continuity of the finitely broken flow, a purely
dynamical result which we prove below in Section~\ref{sec:dynamics}.

Now fix any $t \in (-T_1-\ep,-T_1)$.
Suppose that
\begin{equation}\label{regularity.back}
\Psi_t^j(\mu) \cap \WFbh^{s-j\kk}w=\emptyset,\ j=0,\dots, \lceil
\frac s{\kk}\rceil-1.
\end{equation}
It may be
that $t$ is exactly the time at which one branching
bicharacteristic through $\mu$ lies exactly over $Y$, either at a
hyperbolic or glancing point.  If this is so, we note that we may
increase $t$ so as to ensure that $\Psi_t^j(\mu)$ all lies over $M
\backslash Y_M$, since the basic propagation of b-singularities
(Theorem~\ref{theorem:propofsings}) implies that the same holds for
slightly larger $t$, and our geometric assumptions ensure that we can
make these points lie away from $Y_M$.  Thus we assume without loss of
generality that no points in the finite set $\Psi_t^j(\mu)$ lie over $Y$.

Now note that
$$
\Psi_t^j(\mu) = \Psi_t^1\circ \Psi_{-T_1}^{j-1}(\mu)\cup \Psi_t^0\circ \Psi_{-T_1}^{j}(\mu),
$$
since either zero or one reflections may occur between time $t$ and
$-T_1$.  Our hypothesis \eqref{regularity.back} can thus be rewritten
as
$$
\begin{aligned}
  \Psi_t^1\circ \Psi_{-T_1}^{j-1}(\mu) \cap \WFbh^{s-j\kk}w
  &=\emptyset, j \leq \lceil
\frac s{\kk}\rceil-1,\\
  \Psi_t^0\circ \Psi_{-T_1}^{j}(\mu) \cap \WFbh^{s-j\kk}w &=\emptyset,\ j \leq \lceil
\frac s{\kk}\rceil-1.
\end{aligned}
$$
Reindexing gives
\begin{equation}\label{foo.6.9.25}
\begin{aligned}
  \Psi_t^1\circ \Psi_{-T_1}^{j}(\mu) \cap \WFbh^{s-(j+1)\kk-0}w
  &=\emptyset, j \leq \lceil
\frac s{\kk}\rceil-2,\\
  \Psi_t^0\circ \Psi_{-T_1}^{j}(\mu) \cap \WFbh^{s-j\kk}w &=\emptyset, j \leq \lceil
\frac s{\kk}\rceil-1.
\end{aligned}
\end{equation}
For $j = \lceil\frac s{\kk}\rceil-1$ we crucially note that
since $s-(j+1)\kk\leq 0$, $\WFbh^{s-(j+1)\kk}w=\emptyset$ (as $w$ is
uniformly $L^2$-bounded), hence in fact
both lines of \eqref{foo.6.9.25} apply for the whole range $j \leq \lceil
\frac s{\kk}\rceil-1$.

Now Theorem~\ref{theo:diffractiveimprovements}, applied between times $t$ and $-T_1$,
implies that
\begin{equation}\label{propagation.6.12}
\Psi_{-T_1}^j(\mu) \cap \WFbh^{s-j\kk}w,\ j \leq \lceil
\frac s{\kk}\rceil-1;
\end{equation}
note that since the points we are dealing with lie over $M \backslash
Y_M$, there is no distinction between $\WFbh^r w$ and $\WF_h^r w$, which
is why we have taken the trouble to stay away from the boundary at
this step. 

By construction of $T_1$ \eqref{propagation.6.12}
implies that\footnote{If $\mu$ itself happens to be over the boundary, we back up
slightly along the (unbroken) flow to leave the boundary and apply \eqref{propagation.6.12}, and then
employ Corollary~\ref{cor:propofsings} to find that $\mu \in \WFbh^rw$
iff these nearby interior points are in $\WF_h^r w$}.
$$
\mu \notin \WFbh^s (w).
$$
Thus, the implication \eqref{sharptheorem} holds for a range of times
$t<-T_1$. As $\ep>0$ is locally uniform in $\mu \in \Tbstar M$ and the
energy surface over $t\in [-A,A]$  is
compact for $A \gg T_1$, we may take a single $\ep>0$ for which this result holds, contradicting our assumption that $T_1$ was finite, and the
proof of the first part of the theorem is complete.

The version of the theorem with the infinitely-branching flow then
follows directly.
  \end{proof}

  \subsection{Dynamics of branching flow}
  \label{sec:dynamics}
In this section we collect some purely dynamical results, establishing a
kind of weak continuity for the finitely branching flows (already used above) as well as an
approximability result.

We begin by noting that
branching may only occur finitely many times along a given orbit:
\begin{lemma}
Let $I$ be a compact interval and let $\gamma:I \to \Tbstar M$ a
compressed branching bicharacteristic with
energy $E$.  Then $\gamma$ reflects at most finitely many times.
\end{lemma}
\begin{proof}
If the assertion fails, then the reflection times have a limit point
$t_0$.  We must have $\gamma(t_0) \in \gl_E$, since otherwise there is
a nonvanishing normal momentum and there exists $\delta>0$ such that $\pi(\gamma(t)) \notin Y$ for $t\in
(t_0-\delta, t_0+\delta)$.  But now our finite-order contact hypothesis \eqref{contact}
guarantees that
the same holds near a glancing point as well, so we have
arrived at a contradiction.
\end{proof}
We also remark that the branching flowout of a single point is always
finite:
\begin{lemma}
For all $\mu \in \dot \Sigma_b^E$ and all $t$, $\XPhi{E}_t(\mu)$ is a
finite set.
  \end{lemma}
  \begin{proof}
    The assertion is true for $\abs{t}$ sufficiently small, since
    \begin{enumerate}
    \item If $\pi(\mu)\notin Y$ then the flow is unique for short time.
      \item If $\mu$ is a hyperbolic point then there are two possible
        flowouts for sufficiently short time, lying over $X \backslash
        Y$ except at $t=0$.
        \item If $\mu$ a glancing point, the flowout is unique for
          short time (it leaves $Y$ owing to the assumption \eqref{contact}).
          \end{enumerate}
Thus, it suffices to show that the set $\mathcal{J}$ of $t$ for which the finiteness
holds is both open and closed.  That $\mathcal{J}$ is closed follows from
the continuity of the unbranching flow.  That $\mathcal{J}$ is open follows
by picking any $t \in \mathcal{J}$ and examining individually finitely many
points in the flowout at that time.  For each of these points there is
a time such that the flowout from that point for short time remains finite, by examination of the same trichotomy
discussed above.
\end{proof}

Note that the gliding
ray construction in Remark~\ref{remark:inscribed} above shows that we may have $\mu_j\to \mu$, even while
points in $\XPhi{E}_t(\mu_j)$ lie nowhere near $\XPhi{E}_t(\mu)$; the
essential problem here is that the relation of branching flow (without
control on the number of reflections) is not closed. This does not
occur for finitely branching flows.  We begin with a lemma
that establishes this fact for short times, starting at glancing points
(where the difficulties lie).

\begin{lemma}\label{lemma:continuity}
  Suppose $I$ is an open neighborhood of $[0,T]$ and  $\gammabase:I\to
  \Tbstar M$ is a non-branching bicharacteristic
  curve of energy $E$ with $\gamma(0) \in \gl_E$ and $\pi(\gammabase(t))\notin
  Y$ for $t \neq 0$. Let $V$ be an open neighborhood of $\gammabase(T)$, and
  $N\in\mathbb{N}$. Then there exists a neighborhood $U\subset \Tbstar
  M$ of
  $\gammabase(0)$ and $\epsilon>0$ such that if $\gammabase^N$ is any
  branching bicharacteristic of energy $E$ with at most $N$
  reflections and energy $E' \in (E-\ep, E+\ep)$ and with
  $\gammabase^N(0)\in U$, then
\[\gammabase^N(t)\in V\text{ for all }t\in (T-\epsilon,T+\epsilon).\]
\end{lemma}

\begin{proof}
We now prove the result by induction on $N$. For $N=1$, suppose for
contradiction that we have a sequence $\gammabase^1_j(t)$ of Hamiltonian
trajectories, which make at most one reflection, and a sequence of
times $T_j$ and energies $E_j$ such that $\gammabase^1_j(0)\to \gammabase(0)$, $T_j\to
T$, and $E_j \to E$ but $\gammabase^1_j(T_j)\not\in V$ for all $j$. We note that if, along
a subsequence, we find a sequence of Hamiltonian trajectories which do
not reflect at all, then along that subsequence the continuity of
ordinary Hamiltonian flow would guarantee that eventually
$\gammabase^1_j(T_j)\in V$, a contradiction. Thus, we may assume after
discarding finitely many trajectories that each trajectory reflects
exactly once. Let $T^*_j$ be the time of reflection. Then, along a
subsequence, $T^*_j$ converges to some $T^*\in[0,T]$.

We now claim that
$T^* = 0$.  If not, we note that $\gammabase^1_j$ follows ordinary
Hamiltonian flow on $[0,T^*_j]$, so by continuity of ordinary
Hamiltonian flow, we would have $\gammabase^1_j(T^*_j)\to \gammabase(T^*)$. By
assumption, $\pi(\gammabase(T^*))\not\in Y$, so since $Y$ is closed, 
$\pi(\gammabase^1_j(T^*_j))\notin Y$ for large enough $j$, contradicting
the assumption that $\gammabase^1_j$ \emph{reflected} at time
$T^*_j$. Thus, we must have $T^*_j\to 0$.

Let $\Phi(t,\mu)$ denote the
ordinary Hamiltonian flow on $\dot\Sigma^{E_j}
_b$; recall that this
is well-defined, up to choice of reflection or transmission, even over
$Y$ since the energy and the tangential momentum determine the
normal momentum up to sign.
Note that 
\[\gammabase^1_j(T_j) = \Phi(T_j-T^*_j,\gammabase^1_j(T^*_j)).\]
Since $T_j\to T$ and $T^*_j\to 0$, we have $T_j-T^*_j\to T$ and $\gammabase^1_j(T^*_j) = \Phi(T^*_j,\gammabase^1_j(0))\to \Phi(0,\gammabase(0)) = \gammabase(0)$. We are done if we can establish that $\gammabase^1_j(T^*_j+)\to \gammabase(0)$, since then
\[\gammabase^1_j(T_j)\to \Phi(T,\gammabase(0)) = \gammabase(T),\]
contradicting the assumption that $\gammabase^1_j(T_j)\not\in V$ for all $j$.

So it remains to establish that $\gammabase^1_j(T^*_j+)\to
\gammabase(0)$.  We note that in normal coordinates, if we let $\xi_1$
be the normal momentum of the trajectory viewed in the ordinary (not
b-) cotangent bundle, then the law of reflection is
$-\xi_1(\gammabase^1_j(T^*_j+))+\xi_1(\gammabase^1_j(T^*_j-))=2\xi_1(\gammabase^1_j(T^*_j-))$
(with all other components being continuous). This jump approaches
zero as $j\to \infty$ since $\gammabase^1_j(T^*_j-) \to
\gammabase^1_j(0)$, and the latter has zero normal component as it is
a glancing point.  Since $\gammabase^1_j(T^*_j-)\to \gammabase^1_j(0)$
as there are no intervening reflections, the convergence thus follows, and we have proved
the result for $N=1$.

Suppose now that the statement is proven for $N-1$.  Thus, let
$U^{(N-1)}$ and $\epsilon^{(N-1)}$ satisfy the claim for $N-1$. Again,
suppose for contradiction that we have a sequence $\gammabase^N_j(t)$ of
Hamiltonian trajectories, with at most $N$ reflections, and a sequence
of times $T_j$ and energies $E_j$ such that $\gammabase^N_j(0)\to\gammabase(0)$ and
$T_j\to T$, $E_j \to E$, but $\gammabase^N_j(T_j)\not\in V$ for all $j$. By the same
arguments above, we may assume that these trajectories have at least
one reflection. Let $T^*_j$ be the time of first reflection; then
along a subsequence, $T^*_j\to T^*\in[0,T]$. By the same arguments
above, we must have $T^* = 0$. Then, it follows that for all $j$ large
enough, we have $|T_j-T^*_j-T|<\epsilon^{(N-1)}$ and
$\gammabase^N_j(T^*_j)\in U^{(N-1)}$. It follows that, for these large
enough $j$, the paths $\tilde\gammabase^{N-1}_j(t) := \gammabase^N_j(t+T_j^*)$
are Hamiltonian trajectories with at most $N-1$ reflections satisfying
${\tilde\gammabase}^{N-1}_j(0) = \gammabase^N_j(T_j^*)\in U^{(N-1)}$;
furthermore we also have
$T_j - T_j^*\in (T-\epsilon^{(N-1)},T+\epsilon^{(N-1)})$. It follows
from the inductive hypothesis that
\[\gammabase^N_j(T_j) = \tilde\gammabase^{N-1}_j(T_j-T_j^*) \in V,\]
yielding the desired contradiction.
\end{proof}

It is helpful to rephrase this in term of the relations
  $\XPhi{E}^N$:
  \begin{corollary}\label{cor:continuity}
For every $\mu \in \dot\Sigma_b^E$, there exists $T>0$
sufficiently small such that for $N\in \NN$ and $V$ open containing $\XPhi{E}_T(\mu)$, there exist
a neighborhood $U$ of $\mu$ in $\Tbstar X$ and $\delta>0$ such that for
and $t \in (T-\delta, T+\delta),$ $E' \in (E-\delta, E+\delta)$,
$\mu' \in U\cap \dot\Sigma_b^{E'}$,
$$
\XPhi{E'}_t^N(\mu') \subset V.
$$
\end{corollary}

\begin{proof}
If $\mu \in \gl_E$ the result is just a restatement of
Lemma~\ref{lemma:continuity}, while if $\mu\in \hyp$, or $\pi(\mu)
\notin Y$, the result
follows from continuity of the ordinary Hamilton flow, since for short
enough time the flow splits at most once. (This happens
instantaneously for $\mu \in \hyp$.)
      \end{proof}

      More generally, we have a continuity result for the long-time flow:
\begin{corollary}\label{cor:continuity2}
For every $\mu \in \dot\Sigma^E_b$, $T>0$,
$N\in \NN$, and $V$ open containing $\XPhi{E}^N_T(\mu)$, there exist
a neighborhood $U$ of $\mu$ and $\delta>0$ such that for 
and $t \in (T-\delta, T+\delta),$ $E' \in (E-\delta, E+\delta)$, $\mu'
\in U\cap \dot\Sigma_b^{E'}$
$$
\XPhi{E'}_t^N(\mu') \subset V.
$$
  \end{corollary}
  \begin{proof}
Let $\gamma_1,\dots, \gamma_m$ be all possible branching trajectories
starting at $\mu$ on $t \in [0, T]$.  We proceed by induction on the
maximum number of times $k$ that one of the $\gamma_j$ lies over $Y$ (which may
exceed the number of reflections, owing to the possibilities of
transmission and glancing).  If $k=0$ then the continuity is just continuity
of ordinary bicharacteristic flow, hence the base case is established.

More generally, if $k>0$ consider the first time $T_0$ at which
$\gamma_j$ lies over $Y$ (this would be the same for all $j$).  We
work with a single $\gamma_j$ at a time as we can take 
$U$ to be the intersection of the resulting neighborhoods for each.  If $\gamma_j(T)$ is a \emph{hyperbolic} point
(transverse reflection or transmission) then the continuity follows from the inductive
hypothesis coupled with the continuity of the unbroken flow on the
initial segment: for any neighborhood $V$ of $\gamma_j(T_0)$ in
$\Tbstar_Y X$ consisting only of hyperbolic points and all
$\ep>0$  there exists a neighborhood $U$ of $\mu$  and an energy
neighborhood so that the ordinary
bicharacteristic flow from any point in $U$ at nearby energy hits a point in $V$ at
time $t \in (T_0-\ep,T_0+\ep)$.  As this point furnishes the initial
conditions for the all the branching continuations, the result follows
by the inductive hypothesis.

If $\gamma_j(T_0)$ is a \emph{diffractive} point, i.e., a point of
tangency, then we again employ continuity of the free flow on the
initial segment to know that for any $U_1$ a neighborhood of
$\gamma_j(T_0)$ there exists $U$ a neighborhood of $\mu$ and an energy
neighborhood $E' \in (E-\ep, E+\ep)$ with
$\XPhi{E'}^N_{T_0}(U) \subset U_1$, since the unbroken flow is
continuous and encounters $Y$ for the first time at a time
$T'> T_0-\delta_0$ (with $\delta_0$ as small as we like by
shrinking $U$); the subsequent branching flow for time at most $\delta_0$ results in an
arbitrarily small error, since the flow-speed is locally bounded.
Now we use Corollary~\ref{cor:continuity} with starting point
$\gamma(T_0)$ to ensure that for all $\ep'>0$
and all neighborhoods $U_2$ of $\gamma_j(T_0+\ep')$ in
$\Tbstar X$  we can choose $U_1$ a neighborhood of $\gamma_j(T_0)$ and
$E' \in (E-\ep, E+\ep)$ an energy neighborhood so that
$\XPhi{E}_{\ep'}^N(U_1) \subset U_2$.  The result then follows, with
using the
inductive hypothesis to choose $U_2$ a sufficiently small neighborhood
of $\gamma_j(T_0+\ep')$ (and a sufficiently small energy neighborhood), since the flow from $\gamma_j(T_0+\ep')$ passes
over $Y$ fewer than $k$ times.
  \end{proof}
  
We also establish, for use in proving the Poisson relation, the
  following lemma about approximations of periodic orbits.  Fix any
  metric $d(\bullet,\bullet)$ on $\Tbstar X$.
  \begin{lemma}\label{lemma:approximateflow}
Let $N \in \NN$ and let $K\in \RR$ be compact.  Suppose that $T \notin
\lspec^N_K.$  There exists $\ep>0$ such that for any compressed branching bicharacteristic
$\gamma(t)$ with energy $E
\in K$ and at most $N$ reflections, $d(\gamma(0), \gamma(t))>\ep$ for
$t \in [T-\ep,T+\ep]$.
\end{lemma}
Note once again that this lemma fails if we do not control the number
of reflections: if we compactify the example of
Remark~\ref{remark:inscribed} onto a large torus, we can arrange that
$2\pi \notin \lspec_{1/4}$, while at time $t=2\pi$ there are
nonetheless many-times-reflected inscribed bicharacteristics that have
start- and end-points arbitrarily close to one another.  (Recall that
energy $1/4$ gives unit speed propagation in our normalization.)
\begin{proof}
If not, there is a sequence of $\gamma_j$ with $d(\gamma(0),
\gamma(t_j))\to 0$, $t_j \to T$.  Extracting a subsequence we may
assume that $\gamma_j(0)$ converges, the energies converge, and the
number, times, and locations of reflections all converge.  By
continuity of the unbroken flow, the $\gamma_j$ must then converge to
a limit that is again a compressed branching bicharacteristic of
energy $E \in K$ and length $T$, a contradiction.
  \end{proof}

\subsection{Cauchy data}
  
Finally, we turn from the spacetime description of singularities to
the \emph{Cauchy problem}, so as to understand the mapping properties
of operators of the form
$$
A U(t) B
$$
with $A,B \in \Psibh(X).$

Fix a time $t$ and let $\iota_t: X \to M$ be given by
$$
\iota_t(x,y) =(x,y,t).
$$
Then
$$
\iota_t^*:  \Tbstar M \to \Tbstar X
$$
is projection of one-forms onto their spatial components:
$$
\iota_t^*(\xi dx/x+\eta dy + \tau dt)=\xi dx/x+\eta dy.
$$
As usual we also let $\iota_t^*$ denote the pullback on functions
(i.e., restriction to a fixed time).

We now address the relationship between the spacetime (semiclassical, b-) wavefront set of the
solution to the Cauchy problem and the wavefront set of the Cauchy
data.  The imposition of the dynamical assumption \eqref{contact} here
seems purely a technical convenience in what follows; but without it, we see no simple
alternative to revisiting the basic propagation of singularities
results of \cite{GaWu:23a} in the context of evolution of  Cauchy data for
the time-dependent equation (as opposed to the spacetime approach of \cite{GaWu:21}).
\begin{proposition}\label{proposition:cauchydata}
Assume the dynamical assumption \eqref{contact} holds.
  Let $w = w(h)$ be compactly microsupported such that
        $Q w = \mathcal{O}(h^\infty)_{L^2_\loc}$. Then for each $s \in \RR \cup
        \{+\infty\}$ and $t_0 \in \RR$,
        $$
\iota_{t_0}^* \WFbh^s (w) =\WFbh^s (\iota_{t_0}^*
w),\quad \WFbh^s(w) \subset \dot \Sigma_b.
$$
  \end{proposition}
  \begin{proof}
We work near $Y$, since the argument in the interior of $X \backslash
Y$ is the same (but simpler).  Since we will be much concerned with
Cauchy data, we use the simplified notation $u(t_0):=
u\rvert_{t=t_0}=\iota_{t_0}^* u$.

We will take $s=+\infty$ throughout; the proof goes through verbatim
for finite values as well.

The containment $\WFbh (w) \subset \dot \Sigma_b$ is the
content of the elliptic estimates in \cite[Section 7.2]{GaWu:23a} (see
\cite[Appendix A]{GaWu:21} for modifications necessary in this
time-dependent setting).

Assume now that $\WFbh^s(w) \cap (\iota_{t_0}^*)^{-1} (\mu)\neq \emptyset$ for some $\mu
\in \Tbstar X$.  By compact microsupport, we may choose
$\Upsilon \in \Psibh(M)$ with compact microsupport and with
$\WF'(I-\Upsilon)\cap \WFbh w=\emptyset$ so that
$$
w=\Upsilon w+O_{L^2}(h^\infty),
$$
(locally in $t$).

Then for $\psi$ a cutoff function supported near $t=t_0$, and a
spatial-variables-only pseudodifferential operator $A \in
\Psibh(X)$ microsupported close to $\mu$,
$$
\psi(t) Aw =\psi(t) A\Upsilon w+O(h^\infty).
$$
Unlike $A$, which is only pseudodifferential in the spatial variables,
$\psi(t) A \Upsilon$ lies in $\Psibh(M)$ (as is easily seen by writing
$A$ as a left quantization and $\Upsilon$ as a right
quantization). Its microsupport lies in the union of
$(\iota_t^*)^{-1}(\WF' A)$ for $t \in \supp \psi(t)$, hence by the assumptions on $w$, if $A$ is
taken to have microsupport sufficiently close to $\mu$ and $\psi$
support sufficiently close to $t_0$, then the
microsupport of this operator is disjoint from $\WFbh(w)$, hence
$$
\psi(t) Aw =O(h^\infty).
$$
Now consider time derivatives of this quantity: we find that
$$
\pa_t^k \psi(t) Aw=h^{-k} (h \pa_t)^k \psi(t) Aw =O(h^\infty)
$$
as well, by the same reasoning. Thus in particular (using $k=0,1$), we find that the
restriction  of $\psi(t) Aw$ to $t=t_0$ is $O(h^\infty)$, and this
yields the desired absence of $\mu$ from $\WFbh (w(t_0))$.  
We have thus established the containment
$$
 \WFbh (\iota_t^* w) \subset\iota_t^* \WFbh (w).
$$
(Note that this did not use the fact that $w$ solves
the PDE: it is a general fact about restrictions.)

We now turn to the reverse containment: we need to know that if $\mu
\notin \WFbh  (w(t_0))$ then $(\iota_{t_0}^*)^{-1} (\mu) \cap \WFbh (w)=\emptyset.$
Over $X\backslash Y$, since $\WFbh(w)$ agrees with the usual
semiclassical wavefront set, this follows from standard results for
hyperbolic systems as in \cite[Section 23.1]{Ho:85}; see in
particular Theorem 23.1.4 and the remarks following it.  An essential
ingredient is the fact that the projection on the characteristic set, $\iota^*\rvert_{\Sigma(q)}$, is
1--1 over $X \backslash Y$.

Over $Y$,
however, we face the complication that $\iota^*$ is no longer 1--1,
since (in the product coordinates of Section~\ref{sec:prelim}) $\xi_1$ is identically zero on $\dot\Sigma_b,$ hence choosing any
value of $\tau$ with $\tau+\abs{\xi'}_k^2 +V\leq 0$ gives a point in $\dot
\Sigma_b$ with the given projection to the $\xi'$ variables.  Here is where the dynamical assumption~\eqref{contact} is
convenient.  If there is $\mu' \in \dot\Sigma_b$ with
$\iota_{t_0}^*\mu'=\mu$ and $\mu' \in \WFbh(w),$ then the propagation
of spacetime singularities results from \cite{GaWu:21} coupled with
the assumption on the dynamics mean that $\mu'$ is a limit point of
points $\mu_j'$ over $M\backslash Y_M$ that lie in $\WFbh (w)$.  Since the
desired result
holds over $X\backslash Y$, $\iota_{t(\mu_j')}^*(\mu_j')\in \WFbh
w(t(\mu_j')).$ Since the
projection $\iota^*$ is continuous, $\iota_{t(\mu_j')}^*(\mu_j') \to
\mu$.

We will thus obtain a contradiction with the hypothesis that $\mu
\notin \WFbh (w(t_0))$ if we can show that $\mu \notin
\WFbh (w(t_0))$ implies that there exists a neighborhood $U$ of $\mu$ in
$\Tbstar X$ and an $\ep>0$ such that
\begin{equation}\label{WFclosed}
\abs{t-t_0}<\ep \Longrightarrow U \cap \WFbh(w(t))=\emptyset.
\end{equation}
Another way of phrasing \eqref{WFclosed} is that $\WFbh (w(t))$ is closed as a
subset of $\RR_t \times \Tbstar X$; this is a very weak form of propagation of (time-parametrized, Cauchy
data) singularities.  To prove it, we will use some commutator
arguments that are simple analogues of the more sophisticated
constructions of \cite{GaWu:23a}; we refer the reader to Section 5 of
that paper for details.

To show \eqref{WFclosed}, first note that the compact
microsupport in phase space assumption shows that $h D_t w$ lies
in $L^2_\loc$ hence $P_h w\in L^2_\loc$ as well, which
in particular shows that $w$ is $h$-tempered with values in
$H^1_h$, locally uniformly in $t$.  Thus Proposition~5.2 of
\cite{GaWu:23a} applies, and shows that $\xi_1=0$ on $\WFbh(w(t))$ for each $t$.

Thus we may assume without loss of generality that
$$
\mu= (x_1=0, x'=y, \xi_1=0, \xi'=\eta),
$$
since there is no wavefront set except at $\xi_1=0$.
For simplicity we also translate to set $t_0=0$.
Let $A\in \Psi_h(Y)$ be a pseudodifferential operator in the $x'$
variables only, given by semiclassical quantization in the $x'$ variables of the symbol
$$
\big(\chi(\abs{x_1}+\delta^{-1} t) \chi(\abs{x'-y}+\delta^{-1} t)\chi(\abs{\xi'-\eta}+\delta^{-1} t)\big)^{1/2},
$$
with $\chi(s)$ a cutoff equal to $1$ for $s<\delta_\chi/2$, and
supported in $s<\delta_\chi$, having smooth square root and with
$(-\chi')^{1/2}$ also smooth.
Such an operator is not (quite) in the semiclassical b-calculus, as
its symbol, which is independent of $\xi_1$, is
therefore of $S(1)$ type rather than
Kohn--Nirenberg, but it turns out we can treat it for practical
purposes as if it were in the calculus: really it is a smooth family
of tangential pseudodifferential operators.  See Section~7.3 of
\cite{GaWu:23a} for details.

If $\delta_\chi$ is taken sufficiently small, $Aw(0) =O(h^\infty)$ since
$w$ has no wavefront set on its microsupport, viewed as a b-operator
(see Lemma~7.8 of \cite{GaWu:23a}).  Writing
$$
P_h=(hD_{x_1})^*(hD_{x_1}) -h^2 \Lap_Y+V(x),
$$ 
where the adjoint indicates use of the metric inner product, we further compute
$$\begin{aligned}
\pa_t\ang{ A^* A w(t),w(t)} &= \ang{(\pa_t (A^*A) + (i/h)[P_h,A^*A])w,w}\\ &=\ang{Bw,w}+\ang{Cw,w} + h\ang{Rw,w}
\end{aligned}
$$
where the operators on the RHS have the following properties (cf.\ Lemma 3.7 of \cite{GaWu:23a}):
\begin{itemize}\item
$\displaystyle
\sigma(B) =
\chi(\abs{x_1}+\delta^{-1} t)(\pa_t +\hamvf_Y)\chi(\abs{x'-y}+\delta^{-1} t)\chi(\abs{\xi'-\eta}+\delta^{-1} t),
$
with $$\hamvf_Y
= 2 \xi'_jk^{j\ell} \pa_{\xi'_\ell} -\frac{\pa k^{j\ell}}{\pa
  x_m'}\xi_j\xi_\ell \pa_{\xi'_m}-\frac{\pa V}{\pa {x'_j}}  \pa_{\xi'_j}
$$ the Hamilton vector field in $(x',\xi')$ of
$\abs{\xi'}^2_k+V(x_1,x')$,
\item
$\displaystyle
C=\big((i/h)[(hD_{x_1})^*(hD_{x_1}), \chi(\abs{x_1}+\delta^{-1}
t)]+\pa_t\chi(\abs{x_1}+\delta^{-1} t) )A'
$
with $A'$ a tangential operator with symbol
$\chi(\abs{x'-y}+\delta^{-1} t)\chi(\abs{\xi'-\eta}+\delta^{-1} t)$.

\item  the $C$ term is supported away from $Y$ for $\abs{t}$ small, and can be viewed as an
ordinary (i.e., not b-) semiclassical pseudodifferential operator (again with an
$S(1)$ symbol, owing to the global support in $\xi_1$).  Better yet,
we can consider it a sum of compositions of tangential pseudodifferential
and normal differential operators.
\end{itemize}

For $\delta$ sufficiently small, since $\chi' \leq 0$ we can
arrange that $\sigma(B)$ is negative, and indeed minus a sum of
squares, so $B=-\sum G_j^*G_j+hR'$;
for simplicity we lump the $hR'$ remainder into the $hR$ term in what follows.  We can also
arrange that the symbol of $C$ be strictly negative on a neighborhood
of $\xi_1=0$ (and hence on $\WFbh(w(t))$): on $\abs{\xi_1}<1$ (say), taking $\delta>0$ sufficiently
small makes the symbol of this operator positive, with the
$\delta^{-1}$ term outweighing $\xi_1/x_1$ terms from the commutator
with $(hD_{x_1})^*(hD_{x_1})$; recall that $\abs{x_1}$ is bounded
below on the support for $\abs{t}$ small.  Hence we can split
the symbol of $C$ into $-f^2 +e$ where $e$ is supported on
$\abs{\xi_1}>1/2,$ and $\abs{x_1}<\delta_\chi$.  Thus on the operator side,
$$
C=-F^*F+E+hR''
$$
with $E w=O(h^\infty)$ uniformly in (small) $t$, since it is an
operator\footnote{Cf.\ Lemma~7.8 of
\cite{GaWu:23a} for the mild subtleties involved in employing
tangential pseudodifferential operators here.} supported away
from $Y$, with $\WF' E$
disjoint from $\WF_h w(t)$ for all $t$. 
 (Again, we will lump
 the $hR''$ remainder term with $hR$ below.)

Assembling the above computations yields
$$
\pa_t\ang{A^*A w,w}\leq Ch
$$
for small time; since the initial data at $t=0$ is $O(h^\infty)$,
this yields
$$
\ang{A^* Aw,w}=O(h)
$$
for $t \in [0,\ep)$ for some $\ep>0$.  Now as usual in positive
commutator arguments
we work iteratively,
shrinking the cutoffs and also $\ep$, while keeping them larger
than some fixed open set in phase space and some $\ep_0$ respectively, to
show that in fact for some $\tilde{A}$ of the same form as $A$ above,
$$
\ang{\tilde{A}^* \tilde{A} w,w}=O(h^\infty),\ t \in [0, \ep_0).
$$
This shows that\footnote{Once again, technically $\tilde A$ is not in the
  calculus, again owing to the global support of its symbol in
  $\xi_1$; however, composition with the quantization of an honest
 semiclassical b-pseudodifferential operator that is elliptic at $\mu$
 does give an operator in $\Psibh(X)$, elliptic at $\mu$, which
 uniformly maps $w(t)$ to be $O(h^\infty)$ for $\abs{t}$ small.  Again
see the proof of Lemma~7.8 of \cite{GaWu:23a} for similar
manipulations that prove the deeper converse result that lack of
$\WFbh (w(t))$ gives boundedness under tangential elements of $\Psi_h(X)$.}
  the elliptic set of $\tilde{A}$ is disjoint from $\WFbh
(w(t))$ for sufficiently small positive $t$, hence we have established
\eqref{WFclosed} for $t \in [0, \ep)$.  A similar, time-reversed, argument then takes
care of $t \in (-\ep, 0]$.  This establishes \eqref{WFclosed}, which
then give the desired contradiction with our assumption that $\mu' \in
\WFbh (w)$.
\end{proof}

In order to use the results above, relating Cauchy data to solution
wavefront sets for compactly microsupported solutions in $H^1_h$,
we will need to show that our spectral cutoff produces such solutions
from $L^2$ data, and characterize its mapping properties.  In
particular, we need to know that it is microlocal in the the context
of the semiclassical b-calculus, despite not globally lying in that
calculus.  (That it is \emph{locally} a pseudodifferential operator away from
$Y$ is also part of the following result.)
\begin{proposition}\label{prop:chi}
  Let $\chi \in \mathcal{C}_c^\infty(\RR).$  For any $f \in L^2(X),$
\begin{equation}\label{wfcontainment}
\chi(P_h) f \in H^1_h(X) \quad \text{and} \quad \WFbh(\chi(P_h) f) \subset \WFbh(f)\cap \bigcup_{E \in \supp \chi}\dot\Sigma^E_b.
\end{equation}
Moreover, for $\psi_1,\psi_2\in \mathcal{C}_c^\infty$ supported away from
  $Y$, $\psi_1 \chi(P_h) \psi_2 \in \Psi_h(X)$.
  \end{proposition}
  \begin{proof}
The fact that $\chi(P_h) :L^2 \to H^1_h$ simply follows from the fact
that $P_h \chi(P_h)$ is bounded on $L^2(X)$ by the functional calculus,
hence
$$
\ang{P_h \chi(P_h) f , \chi(P_h) f} \lesssim \norm{f}^2.
$$
Thus for some $C>0$, the $H^1_h$ norm of $\chi(P_h) f$ is controlled by $\norm{\chi(P_h)
  f} +C\norm{f}$, which is itself
bounded by $\norm{f}$ by the functional calculus.

We now turn to the
inclusion.
$$
\WFbh(\chi(P_h) f) \subset \WFbh(f).
$$
To evaluate the mapping properties on wavefront set, we proceed using
the Helffer-Sj\"ostrand functional calculus---see \cite{HeSj:89} and
\cite{DiSj:99} for a pedagogical treatment.  Our strategy of proof is
to begin by
establishing estimates for the resolvent applied to $f$ and then
recall that
\begin{equation}\label{HS}
\chi(P_h) =  \frac 1{2\pi i} \int \bar{\pa} \tilde \chi(z)
(z-P_h)^{-1} dz \wedge d\bar{z},
\end{equation}
where $\tilde \chi$ is an almost-analytic extension of $\chi$, so that
$\tilde \chi$ agrees with $\chi$ on the real axis, and $\bar{\pa}
\tilde \chi(z)=O(\abs{\Im z}^\infty)$.  Recall that we may further
take $\tilde \chi$ to be compactly supported.

Thus, we begin with resolvent estimates: suppose
$$
(P_h-z)u=f.
$$
By self-adjointness of $P_h$, pairing with $u$ and taking imaginary part
gives as usual the $L^2$ estimate off
the spectrum
$$
\norm{u}\leq \frac{1}{\abs{\Im z}} \norm{f}.
$$
Moreover, examination of the real part of the pairing gives
$$
\norm{hD_{x_1} u}^2 + \norm{h\nabla_{x'} u}^2 \lesssim \norm{u}^2 +\norm{f}^2
$$
hence
\begin{equation}
\norm{hD_{x_1} u}^2 + \norm{h\nabla_{x'} u}^2 \lesssim\big(1+ \abs{\Im z}^{-2}\big)
\norm{f}^2,
\end{equation}
i.e., we have an $H^1_h$ estimate as well: on $\abs{\Im z}<1$,
\begin{equation}\label{H1est}
\norm{(P_h-z)^{-1} f}_{H^1_h}\lesssim \abs{\Im z}^{-1}
\norm{f}.
  \end{equation}

Now pick any $B_1 \in \Psibh^0(X)$ and compute
\begin{equation}\label{ellipticcommutator}
(P_h-z) B_1 u = B_1 f+[P_h,B_1] u.
\end{equation}
The commutator term above is \emph{not} in the pseudodifferential
calculus $\Psibh(X)$, as $P_h$ contains $hD_{x_1}$ terms that don't lie in the
calculus, but Section 3.3 of \cite{GaWu:23a} shows that these
commutators can still be written in terms of $hD_{x_1}$ and
b-pseudodifferential operators with unchanged microsupport.  In
particular the difficult part of the commutator can be written in the
form $[(hD_{x_1})^* (hD_{x_1}), B_1]$, where the adjoint is with
respect to the metric inner product.  We compute (cf.\ Lemma 3.6 and Lemma 3.7 of \cite{GaWu:23a}))
\begin{equation}\label{diffbcommutator}
h^{-1} [(hD_{x_1})^* (hD_{x_1}), B_1]=(hD_{x_1})^* C_1
(hD_{x_1})+C_2 (hD_{x_1})+C_3
\end{equation}
where $C_1 \in \Psibh^{-1}(X)$,  $C_2 \in \Psibh^{0}(X)$,  $C_3 \in
\Psibh^{1}(X)$, all with microsupport contained in $\WF'(A)$.  (Here
we have crudely lumped together various terms whose principal symbols we can in fact compute
more explicitly as in \cite{GaWu:23a}).
Thus (cf.\ \cite[Lemma 3.9]{GaWu:23a}),
\begin{equation}\label{foo.5.13}
\abs{\ang{[P_h,B_1] u,B_1 u}} \leq \big( h\norm{B_2
  u}_{H^1_h}+O(h^\infty) \norm{u}_{H^1_h}\big) \norm{B_1 u}_{H^1_h},
\end{equation}
where $B_2\in \Psibh^0(X)$ has slightly expanded microsupport, so that
it is elliptic on $\WF' B_1$.  (See Section 5.1 of \cite{GaWu:23a} for
similar computations.)  Now pairing
\eqref{ellipticcommutator} with $B_1 u$ and taking imaginary resp.\
real parts yields
$$
\abs{\Im z} \norm{B_1 u}^2\lesssim \text{RHS},\quad 
\norm{h \nabla B_1 u}^2 \leq C \norm{B_1 u}^2+\text{RHS}
$$
where in both cases, RHS denotes
$$
\abs{\ang{B_1f+[P_h,B_1] u, B_1 u}}.
$$
Putting these facts together gives, for $\abs{\Im z}<1$,
$$
\norm{B_1 u}^2_{H^1_h} \lesssim \abs{\Im z}^{-1} \text{RHS}.
$$
Meanwhile, using \eqref{foo.5.13} and Cauchy--Schwarz yields
$$
\text{RHS}  \leq \big(\norm{B_1 f} + h\norm{B_2 u}_{H^1_h}+O(h^\infty)\norm{u}_{H^1_h}\big)\norm{B_1 u},
$$
so that we now obtain on $\abs{\Im z}<1$
$$
\norm{B_1 u}_{H^1_h} \lesssim \abs{\Im z}^{-1}  \big(\norm{B_1 f} + h\norm{B_2 u}_{H^1_h} +O(h^\infty) \norm{u}_{H^1_h}\big).
$$
Iterating this estimate, using operators with slightly expanding
microsupports (and estimating all the $f$ terms with the single
term $B_k f$ via an elliptic estimate) yields
\begin{multline}
\norm{B_1 u}_{H^1_h} \lesssim \abs{\Im z}^{-1} \big( 1+h\abs{\Im z}^{-1}+\dots + h^{k-1} \abs{\Im z}^{-k+1}\big)
\norm{B_kf}\\ + h^k\abs{\Im z}^{-k} \norm{B_{k+1} u}_{H^1_h}+
O(h^\infty) \abs{\Im z}^{-k} \norm{u}_{H^1_h}
\end{multline}
Finally, we can terminate the iteration by using \eqref{H1est} to get
$$
\norm{B_{k+1} u}_{H^1_h} \lesssim \norm{u}_{H^1_h} \lesssim \abs{\Im z}^{-1}
\norm{f};
$$
hence (controlling the $O(h^\infty)$ term this way as well),
$$
\norm{B_1 u}_{H^1_h} \lesssim \abs{\Im z}^{-1}\big( 1+h \abs{\Im z}^{-1}+\dots + h^{k-1} \abs{\Im z}^{-k+1}\big)
\norm{B_k f} + h^k \abs{\Im z}^{-k-1} \norm{f}.
$$
Consequently, taking $\tilde \chi$ supported in $\abs{\Im z}<1$ we
may then estimate, for $z \in \supp \tilde \chi \backslash \RR$,
$$
\norm{B_1 (P_h-z)^{-1} f}_{H^1_h} \lesssim \abs{\Im z}^{-k}
\norm{B_kf} + h^k\abs{\Im z}^{-(k+1)} \norm{f}.
$$
Inserting this estimate into \eqref{HS} yields a convergent integral
on the RHS owing to the rapid vanishing of $\bar \pa \tilde \chi$ at
$\RR$, hence for any $k$,
$$
\norm{B_1\chi(P_h) f}_{H^1_h} \lesssim \norm{B_k f} + O(h^k) \norm{f}.
$$
This implies the desired mapping property: if $\alpha \notin \WFbh (f)$
we choose the microlocalizers as above so that $B_1$ is elliptic at
$\alpha$ and $\WF' B_k$ is contained in a neighborhood of $\alpha$
disjoint from $\WFbh (f)$.  Then $B_k f=O(h^\infty)$ and the estimate
yields $B_1 \chi(P_h) u=O_{H^1_h}(h^k)$ (for any $k$).

The inclusion
$$
\WFbh(\chi(P_h) f) \subset \bigcup_{E \in \supp \chi}\dot\Sigma^E_b(p)
$$
follows from the combination of the corresponding elliptic estimate in spacetime,
\cite[Proposition 5.2]{GaWu:23a} as revisited in \cite[Appendix
A]{GaWu:21}, together with Proposition~\ref{proposition:cauchydata}
above, which allows us to convert this to a result about Cauchy data.

Finally, the assertion that $f(P_h)$ is a pseuoifferential operator away
from its singularities at $Y$ follows from the results of Section~4 of \cite{Sj:97}.
\end{proof}
  \begin{remark}
The inclusion $\WF\chi(P_h) \subset \bigcup \dot\Sigma^E_b$
can also be proved directly using the the functional calculus.  For
instance, to
show that the wavefront set is contained in $\xi_1=0$ over $Y$,  given $A$
microsupported at $\xi_1\neq 0$ for $x_1$ small, we may factor $P_h-z$
out of $A$ mod $O(h^\infty)$ uniformly in $z \in \supp \tilde \chi$, and then integrate
by parts to move $\bar \partial$ in \eqref{HS} to land on what
is then a holomorphic operator family.
    \end{remark}

    We record for later use the fact that we may regard the cutoff
    $\chi(P_h)$ as being in the time variable instead:
\begin{lemma}\label{lemma:timeorspace}
  For $\chi \in \CcI(\RR)$,
  $$
  \chi(-hD_t) U(t) = \chi(P_h) U(t).
 $$
\end{lemma}
\begin{proof}
Since $U(t)$ is unitary, for $f \in L^2$ we can view $U(t)f \in
L^\infty(\RR_t; L^2)$, and apply the distributional semiclassical
Fourier transform in $t$
to write (using the functional calculus)
\begin{equation*}
\begin{split}
\chi(-hD_t) U(t) f &= (\F_h^{-1} \chi(-\bullet))* (U(\bullet) f)\\
&=
\int (\F_h^{-1} \chi(-\bullet))(t') (U(t-t') f)\, dt'\\
&=
\int (\F_h^{-1} \chi(-\bullet))(t') e^{it'P_h/h}e^{-itP_h/h} f)\, dt'\\
&=
e^{-itP_h/h} \int (\F_h^{-1} \chi(-\bullet))(t') e^{it'P_h/h} f)\, dt'\\
&= e^{-itP_h/h} \chi(P_h) f. \qedhere
\end{split}
\end{equation*}
\end{proof}

We are now finally in a position to describe the mapping properties of the
finite-energy propagator with respect to b-microlocalizers.
\begin{proposition}
\label{prop:wf-relation}
Let $\chi \in \mathcal{C}_c^\infty(\RR)$ and let $A,B \in \Psibh(X)$
have compact microsupport.
Let
$$
S(t) =A \chi(P_h) U(t) B: L^2 \to L^2.
$$
Suppose that for each $E \in \supp \chi$, 
$$
\WF'A \cap \XPhi{E}(\WF'(B))=\emptyset.
$$
Then
$$
S(t)=O_{L^2}(h^\infty).
$$
Suppose that for each $E \in \supp \chi$,
$$
\WF'A \cap \XPhi{E}^N(\WF'(B))=\emptyset.
$$
Then
$$
S(t)=O_{L^2}(h^{(N+1)\kk-0}).
$$
\end{proposition}
\begin{proof}
We describe the $\XPhi{E}_t$ result, with the proof of the $\XPhi{E}^N_t$
version being analogous.

  Given any $f\in L^2,$ let $w(t) =\chi(P) U(t) Bf=\chi(-hD_t) U(t)
 Bf$.  This is a compactly-microsupported solution to the
 Schr\"odinger equation, since its microsupport lies in $-\tau \in
 \supp \chi$ owing to the frequency cutoff, while the elliptic
 estimate in \eqref{wfcontainment} keeps the spatial
 fiber variables in $\Tbstar M$ in a compact set as well.

The Cauchy data of $w(t)$ is $\chi(P_h) Bf$, hence by
  Proposition~\ref{prop:chi}, the
  wavefront set of the Cauchy data is contained in $\WF' B$.  Thus, by Proposition~\ref{proposition:cauchydata},
  the spacetime wavefront set of $w$ at $t=0$ is contained in
$(\iota_0^*)^{-1}(\WF'B) \cap \dot \Sigma_b$; it also lies in $\{-\tau
\in \supp \chi\}$ since we can write this as $\chi(-hD_t) U(t) Bf$.  Propagation of
singularities (Corollary~\ref{cor:propofsings}) now shows that a
neighborhood of $(\iota_0^*)^{-1}(\WF' A)$ is disjoint from $\WFbh
(w)$.  Hence another application of
Proposition~\ref{proposition:cauchydata} then implies that
\begin{equation*}
    A w(t)=O_{L^2}(h^\infty).\qedhere
\end{equation*}
\end{proof}

\section{The Poincar\'e map}
\label{sec:poincare}

We now turn to some dynamical preliminaries that will be necessary to
identify the term in the trace formula corresponding to the linearized
Poincar\'e map.  (Recall that our analysis of the terms in this formula
will also rely essentially on prior work in \cite{WuYaZo:24} for
the interpretation of the Maslov factor.)

Let $\gamma$ be a closed branching orbit with period
$T$ and energy $-\tau=E_0$, and fix $\alpha
\in T^*(X\backslash Y)$ a
point along $\gamma.$ In this section, as we work at points over
  $X\backslash Y$, we will slightly abuse notation by
  letting $$\Phi_t(\mu)= \XPhi{p(\mu)}_t^\gamma(\mu)$$ denote the branching flow from a point in $T^*(X\backslash
  Y)$ with the energy fixed by its value $p(\mu)$ at the initial point.
  (Remember that the flow in $\Tbstar X$ without a specified energy a priori ``forgets'' its precise
  energy at times of interaction with the boundary.) This flow is thus
perhaps better viewed as a (discontinuous!) branching flow in the
ordinary cotangent bundle, where normal momentum may simply change
sign at reflected points.  Recall that the $\gamma$ superscript means
we view this flow as being the locally single-valued flow close to the
given branching bicharacteristic $\gamma$.

Fix a Hamiltonian flow box chart (cf.~\cite[Section~8.1]{abraham2008foundations}) near $\alpha,$
i.e., a symplectic coordinate system $(H,s,q,p)$ for $T^*X$ in which
the Hamiltonian is simply $p=H$ and where $\alpha=(E_0, 0, 0, 0)$ in
these coordinates. Then, as $\gamma$ is a periodic branching orbit with period $T$ passing through $\alpha$, we have 
$$
\Phi_T(E_0, 0, 0,0)=(E_0, 0, 0,0),
$$
i.e., $\alpha$ is a fixed point of the Hamiltonian flow $\Phi_t$.
(Recall that we are abusing notation by letting $\Phi_t$ denote the
\emph{locally} single-valued flow near the given trajectory $\gamma$,
away from branching points.
At points close to $\alpha$, we write
$$
\Phi_T(H, s, q,p)=(H, s+\tau(H,q,p), Q(H,q,p), P(H,q,p))
$$
for some functions $\tau,Q,P$.  Note that since the Hamilton vector
field is simply given by $\hamvf_P = \pa_s$, if $(Q(H',q',p'),
P(H',q',p'))=(q',p')$ then the point $(H',0,q',p')$ lies along another
closed trajectory with, in general, a distinct period
$T-\tau(H',q',p')$. 

Our nondegeneracy assumption is that 
$$
\frac{\pa (Q,P)}{\pa (q,p)}-\Id
$$
is nonsingular at $\alpha$.
We note that since
\[d\Phi_T = \begin{pmatrix} 
\begin{matrix}
1 & 0 \\
* & 1
\end{matrix} &
\begin{matrix}
0 & 0 \\
* & *
\end{matrix} \\
\begin{matrix}
* & 0 \\
* & 0
\end{matrix} &
\frac{\pa (Q,P)}{\pa (q,p)}
\end{pmatrix},\]
it follows that
\[\begin{aligned}\det\left(\lambda-d\Phi_T\right) = \det\begin{pmatrix}
\lambda-1 & 0 & \begin{matrix} 0 & 0 \end{matrix} \\
* & \lambda-1 & \begin{matrix} * & * \end{matrix} \\
\begin{matrix}
* \\ *
\end{matrix} &
\begin{matrix}
0 \\ 0
\end{matrix} &
\lambda-\frac{\pa (Q,P)}{\pa (q,p)}
\end{pmatrix} &= (\lambda-1)\det\begin{pmatrix}
\lambda-1 & \begin{matrix} * & * \end{matrix} \\
\begin{matrix} 0 \\ 0 \end{matrix} & \lambda-\frac{\pa (Q,P)}{\pa (q,p)}
\end{pmatrix} \\ &= (\lambda-1)^2\det\left(\lambda-\frac{\pa (Q,P)}{\pa (q,p)}\right).\end{aligned}\]
Thus $\frac{\pa (Q,P)}{\pa (q,p)}-\Id$ is nonsingular iff the multiplicity of $1$ in $d\Phi_T$ is exactly $2$.

Given nondegeneracy, we remark that by the Implicit Function
Theorem we may solve the equations
$$
(Q(H,q,p),P(H,q,p))=(q,p)
$$
locally for $(q,p)$ in terms of $H$ close to $E_0$; let
$(\qcal(H),\pcal(H))$ denote the resulting functions of $H$.
In other words, we have obtained an
\emph{orbit cylinder}, a family of closed orbits with, in general,
varying periods; the cylinder is transverse to the energy surface
since its tangent space at $\alpha$ is spanned by $\pa_s$ (which lies tangent to
individual orbits) and $V=\pa_H+ (\pa_H \qcal) \pa_q+(\pa_H \pcal) \pa_p$.
Let $\Ccal$ denote the orbit cylinder.  We note further that
differentiating the equation
$$
\Phi_T(H, 0,\qcal(H), \pcal(H)) =(H, \tau(H,\qcal(H),\pcal(H)),  \qcal(H), \pcal(H))
$$
yields
$$
d\Phi_T(V) =V+c\pa_s
$$
for some $c$; hence
$d\Phi_T$ acting on $T_\alpha\Ccal=\operatorname{span}(\pa_s, V)$ has matrix representation
$$
\begin{pmatrix}
  1 & *\\
  0 & 1
  \end{pmatrix},
$$
i.e.\ $T_\alpha\Ccal$ is the generalized eigenspace for $d\Phi_T$
corresponding to eigenvalue $1$.

\begin{figure}
    \centering
    \includegraphics[width=0.6\linewidth]{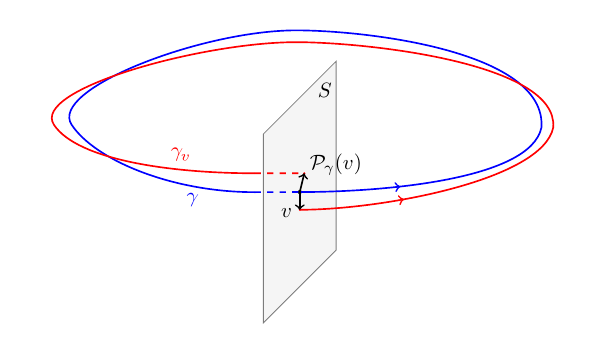}
    \caption{Poincare map $\mathcal P$ of the closed branching orbit $\gamma$.}
    \label{fig:poincare}
\end{figure}

We can now define the Poincar\'e map contribution to the trace
formula, which is
$$
\det (\Id-\Pcal)
$$
where $\Pcal$ is the map
$$
\Pcal=(d\Phi_T)\rvert_{T\Ccal^\perp},
$$
and $T\Ccal^\perp$ denotes the symplectic orthocomplement of the orbit
cylinder, which is conserved by $d\Phi_T$ since this map is symplectic
and preserves $T\Ccal$.

We now turn to computations that relate the linearized Poincar\'e
determinant on the symplectic orthocomplement of the orbit cylinder to
derivatives of the classical action.

Suppose $\alpha=(x_0,\xi_0)$ lies on the closed branching
orbit $\gamma$.  For $x,y$ near $x_0$ and $t$ near the
period $T$, there is a unique trajectory going from $y$ to $x$ in time
$t$ which is close to the given periodic trajectory. Denote $S(t,x,y)$
the classical action along such a trajectory. 

Using canonical coordinates on $T^*X$, let us write
\[\Phi_t(y,\eta) = (X(t,y,\eta),\Xi(t,y,\eta)).\]
Then as noted above,
\[(d\Phi_T)|_{x_0,\xi_0}\]
has $1$ as an eigenvalue, with eigenspace of dimension $2$ under non-degeneracy assumptions. Such a space can be identified with $T\Ccal$, the tangent space to an orbit cylinder. We would like to compute
$$
\det (\Id-\Pcal)
$$
where $\Pcal$ is the map
$$
\Pcal=(d\Phi_T)\rvert_{T\Ccal^\perp},
$$
and $T\Ccal^\perp$ denotes the symplectic orthocomplement of the orbit
cylinder, which is conserved by $d\Phi_T$ since this map is symplectic
and preserves $T\Ccal$.

Since $d\Phi_T$ has $1$ as an eigenvalue with algebraic multiplicity 2, if we write
\begin{equation}
\label{eq:qlambda}
\det(\lambda-d\Phi_T) = (\lambda-1)^2q(\lambda),
\end{equation}
then $q(1) = \prod(1-\lambda_i) = \det(\Id-\Pcal)$ where $\lambda_i$ are the eigenvalues of $d\Phi_T$ not equal to $1$, and hence $q(1)$ is the desired Poincar\'e determinant. So we aim to compute $\det(\lambda-d\Phi_T)$, factor out $(\lambda-1)^2$, and find the value of the remaining polynomial at $\lambda=1$.

Writing
\[d\Phi_T = \begin{pmatrix} \partial_yX & \partial_{\eta}X \\ \partial_y\Xi & \partial_{\eta}\Xi\end{pmatrix}\]
in block form, the Schur complement formula yields
\begin{align*}
\det(\lambda-d\Phi_T) &= \det\begin{pmatrix} \lambda-\partial_yX & -\partial_{\eta}X \\ -\partial_y\Xi & \lambda-\partial_{\eta}\Xi\end{pmatrix} \\
&= (-1)^n\det\begin{pmatrix} \partial_{\eta}X & \partial_yX-\lambda \\ \partial_{\eta}\Xi-\lambda & \partial_y\Xi\end{pmatrix} \\
&= (-1)^n\det\partial_{\eta}X\det\begin{pmatrix} \partial_y\Xi - \left(\partial_{\eta}\Xi-\lambda\right)\left(\partial_{\eta}X\right)^{-1}\left(\partial_yX-\lambda\right)\end{pmatrix}.
\end{align*}
The last matrix can be written as
\[A + (\lambda-1)B + (\lambda-1)^2C\]
where
\begin{equation}
\label{eq:ABC}
\begin{aligned}
A &= \partial_y\Xi - \left(I-\partial_{\eta}\Xi\right)\left(\partial_{\eta}X\right)^{-1}\left(I-\partial_yX\right),  \\
B &= -\left(\partial_{\eta}X\right)^{-1}\left(I-\partial_yX\right) -\left(I-\partial_{\eta}\Xi\right)\left(\partial_{\eta}X\right)^{-1}, \\
C &= -\left(\partial_{\eta}X\right)^{-1}.
\end{aligned}
\end{equation}
It follows that
\[\det(\lambda-d\Phi_T) = \left(\det C\right)^{-1}\det(A+(\lambda-1)B+(\lambda-1)^2C).\]

We now proceed via a series of lemmas whose proofs are postponed to
later in this section.
We begin by employing the following general fact about determinants:
\begin{lemma}
\label{lem:detexp}
Suppose $A$, $B$, and $C$ are $n\times n$ matrices, where $A_{1j} = A_{i1} = 0$ for all $i,j\in\{1,\dots,n\}$, and $B_{11} = 0$. Then
\[\det(A+tB+t^2C) = t^2\left(\det\begin{pmatrix} C_{11} & B_{12} & \dots & B_{1n} \\
B_{21} & A_{22} & \dots & A_{2n} \\
\vdots & \vdots & \ddots & \vdots \\
B_{n1} & A_{n2} & \dots & A_{nn} \end{pmatrix} + O(t)\right).\]
\end{lemma}
(That is: it equals $t^2$ times the determinant of the matrix where the top-left entry is that of $C$, the other entries in the first row and column are that of $B$, and the remaining entries are that of $A$, up to $O(t^3)$.)

We now give alternative expressions for the matrices in \eqref{eq:ABC}:
\begin{lemma}
\label{lem:ABC}
We have
\begin{align*}
A &= \partial^2_{xx}S + \partial^2_{xy}S + \partial^2_{yx}S + \partial^2_{yy}S = \frac{\partial^2}{\partial x^2}\Big|_{x=x_0}[S(T,x,x)], \\
B &= \partial^2_{xx}S + 2\partial^2_{yx}S + \partial^2_{yy}S,\\
C &= \partial^2_{yx}S.
\end{align*}
Here, the notation $\partial^2_{xy}S$ refers to the $n\times n$ matrix whose $(i,j)$ entry is $\partial^2_{x^iy^j}S$, and similarly for the other subscripts, and all second derivative matrices not explicitly evaluated are evaluated at $(T,x_0,x_0)$.
\end{lemma}
We now employ Fermi normal coordinates along $\gamma$ near $x_0$,
given by $(x^1,\dots,x^n)$, which enjoy the following properties:
\begin{itemize}
\item $(x^1,\dots,x^n)(x_0) = (0,\dots,0)$, and for $|t|<\ep$ (and hence for $|t-T|<\ep$ as well), the trajectory $X(t,x_0,\xi_0)$ lies in the curve $\{x^2=\dots=x^n=0\}$, i.e., it only has a nonzero $x^1$ component.
\item $g_{ij}|_{\{x^2=\dots=x^n=0\}} = \delta_{ij}$, i.e.\ the Riemannian metric is Euclidean along the trajectory.
\end{itemize}
Note that, with such choice of coordinates, we have
\[\dot X^i = 2g^{ij}\Xi_j = 2\Xi_i\]
along the flow, and since $X^i = 0$ for $i\ge 2$, this implies that $\Xi_i = 0$ for $i\ge 2$. Hence
\[\dot X = (\dot X^1)e_1,\quad \dot\Xi = (\dot\Xi_i)e_1.\]
We can then check that the matrices $A$, $B$, and $C$ in \eqref{eq:ABC} satisfy the assumptions in Lemma \ref{lem:detexp}. Indeed, $\frac{\partial}{\partial x}|_{x=X(s,x_0,\xi_0)}[S(T,x,x)] = 0$ for all $s$ (near $0$) due to the periodicity of the bicharacteristic, from which we have
\[0 = \frac{\partial}{\partial s}\Big|_{s=0}\left(\frac{\partial}{\partial x}\Big|_{x=X(s,x_0,\xi_0)}[S(T,x,x)]\right) = \frac{\partial^2}{\partial x^2}\Big|_{x=x_0}[S(T,x,x)]\cdot \dot{X} = (\dot X^1)Ae_1.\]
Hence, $Ae_1=0$, i.e. the first column of $A$ is zero. Since $A$ is symmetric, the first row is also zero. Finally,
\begin{align*}
B_{11} &= \partial^2_{x^1x^1}S + 2\partial^2_{y^1x^1}S + \partial^2_{y^1y^1}S \\
&=\partial^2_{x^1x^1}S + \partial^2_{x^1y^1}S + \partial^2_{y^1x^1}S + \partial^2_{y^1y^1}S=0,
\end{align*}
the last equality following since it equals $A_{11}$. It follows, from Lemma \ref{lem:detexp}, that
\[\det(A+(\lambda-1)B+(\lambda-1)^2C) = (\lambda-1)^2\left(\det\begin{pmatrix} C_{11} & B_{12} & \dots & B_{1n} \\
B_{21} & A_{22} & \dots & A_{2n} \\
\vdots & \vdots & \ddots & \vdots \\
B_{n1} & A_{n2} & \dots & A_{nn} \end{pmatrix} + O(\lambda-1)\right).\]
Thus, with $q(\lambda)$ as in \eqref{eq:qlambda}, we have
\[q(\lambda) = \left(\det C\right)^{-1}\left(\det\begin{pmatrix} C_{11} & B_{12} & \dots & B_{1n} \\
B_{21} & A_{22} & \dots & A_{2n} \\
\vdots & \vdots & \ddots & \vdots \\
B_{n1} & A_{n2} & \dots & A_{nn} \end{pmatrix} + O(\lambda-1)\right).\]
Hence, with $\det(I-\Pcal) = q(1)$, it follows that
\begin{equation}
\label{eq:poincare1}
\det(I-\Pcal) = (\det C)^{-1}\det\begin{pmatrix} C_{11} & B_{12} & \dots & B_{1n} \\
B_{21} & A_{22} & \dots & A_{2n} \\
\vdots & \vdots & \ddots & \vdots \\
B_{n1} & A_{n2} & \dots & A_{nn} \end{pmatrix}.
\end{equation}
We now look for a different way to express the above matrix when $A$, $B$, and $C$ are as in Lemma \ref{lem:ABC}. Indeed, we now consider the Hessian of the function
\[(t,x^2,\dots,x^n)\mapsto S(t,(0,x^2,\dots,x^n),(0,x^2,\dots,x^n))\]
at $t=T$, $x^2=\dots=x^n=0$.
In matrix form, this is
\begin{align*}
\frac{\partial^2}{\partial(t,\hat x)^2}\Big|_{t=T,x=0}\left[S(t,x,x)\right] = 
\begin{pmatrix}
\partial^2_{tt}S & \partial^2_{tx^2}S + \partial^2_{ty^2}S & \dots & \partial^2_{tx^n}S + \partial^2_{ty^n}S  \\
\partial^2_{x^2t}S + \partial^2_{y^2t}S & A_{22} & \dots & A_{2n} \\ 
\vdots & \vdots & \ddots & \vdots \\
\partial^2_{x^nt}S + \partial^2_{y^nt}S & A_{n2} & \dots & A_{nn}
\end{pmatrix}
\end{align*}
(recalling $A = \frac{\partial^2}{\partial x^2}\Big|_{x=x_0}[S(T,x,x)]$; here, $\hat x = (x^2,\dots,x^n)$).
\begin{lemma}
\label{lem:partialtS}
We have
\[\partial^2_{tt}S = -(\dot X^1)^2C_{11}\]
and
\[\partial^2_{xt}S+\partial^2_{yt}S = B^T\dot X = (\dot X^1)B^Te_1,\]
where $B$, $C$ are as in Lemma \ref{lem:ABC}.
\end{lemma}
Thus this matrix is
\[\begin{pmatrix} -(\dot X^1)^2C_{11} & (\dot X^1)B_{12} & \dots & (\dot X^1)B_{1n} \\
(\dot X^1)B_{12} & A_{22} & \dots & A_{2n} \\
\vdots & \vdots & \ddots & \vdots \\
(\dot X^1)B_{1n} & A_{n2} & \dots & A_{nn} \end{pmatrix}.\]
We note that $B_{1j} = -B_{j1}$ for $1\le j\le n$. Indeed,
\[B_{ij} = A_{ij} + \partial^2_{y^ix^j}S - \partial^2_{x^iy^j}S,\]
so using that $A_{1j} = A_{j1} = 0$ for $1\le j\le n$, it follows that $B_{1j}+B_{j1} = 0$. It follows that
\begin{align*}
\det\frac{\partial^2}{\partial(t,\hat x)^2}|_{t=T,x=0}\left[S(t,x,x)\right] &= \det\begin{pmatrix} -(\dot X^1)^2C_{11} & (\dot X^1)B_{12} & \dots & (\dot X^1)B_{1n} \\
-(\dot X^1)B_{21} & A_{22} & \dots & A_{2n} \\
\vdots & \vdots & \ddots & \vdots \\
-(\dot X^1)B_{n1} & A_{n2} & \dots & A_{nn} \end{pmatrix} \\
&= -(\dot X^1)^2\det\begin{pmatrix} C_{11} & B_{12} & \dots & B_{1n} \\
B_{21} & A_{22} & \dots & A_{2n} \\
\vdots & \vdots & \ddots & \vdots \\
B_{n1} & A_{n2} & \dots & A_{nn} \end{pmatrix}.
\end{align*}
Combining this with \eqref{eq:poincare1}, we arrive at the result
that allows us to identify Hessian quantities arising in our
stationary phase computations as dynamical invariants:
\begin{proposition}
\[\det\frac{\partial^2}{\partial(t,\hat x)^2}|_{t=T,x=0}\left[S(t,x,x)\right] = -(\dot X^1)^2\left(\det\partial^2_{yx}S\right)\det(I-\Pcal).\]
\end{proposition}

We now turn to the proofs of the lemmas used above.

\begin{proof}[Proof of Lemma \ref{lem:detexp}]
The assumptions give $A$ and $B$ to be of the form
\[A = \begin{pmatrix} 0 & 0 & \dots & 0 \\ 
0 & A_{22} & \dots & A_{2n} \\
\vdots & \vdots & \ddots & \vdots \\
0 & A_{n2} & \dots & A_{nn}
\end{pmatrix},\quad B = \begin{pmatrix} 0 & B_{12} & \dots & B_{1n} \\ 
B_{21} & B_{22} & \dots & B_{2n} \\
\vdots & \vdots & \ddots & \vdots \\
B_{n1} & B_{n2} & \dots & B_{nn}
\end{pmatrix}.\]
It follows that
\[A+tB+t^2C = \begin{pmatrix} t^2C_{11} & tB_{12} + O(t^2) & \dots & tB_{1n}+O(t^2) \\ 
tB_{21}+O(t^2) & A_{22} + O(t) & \dots & A_{2n} + O(t) \\
\vdots & \vdots & \ddots & \vdots \\
tB_{n1}+O(t^2) & A_{n2} + O(t) & \dots & A_{nn} + O(t)
\end{pmatrix}.\]
Factoring $t$ out of the first row, and the first column, we thus have
\[\det(A+tB+t^2C) = t^2\det\begin{pmatrix} C_{11} & B_{12} + O(t) & \dots & B_{1n}+O(t) \\ 
B_{21}+O(t) & A_{22} + O(t) & \dots & A_{2n} + O(t) \\
\vdots & \vdots & \ddots & \vdots \\
B_{n1}+O(t) & A_{n2} + O(t) & \dots & A_{nn} + O(t)
\end{pmatrix},\]
the last determinant equaling the determinant of the matrix in the statement, up to $O(t)$.
\end{proof}

\begin{proof}[Proof of Lemma \ref{lem:ABC}]
We use
\[\partial_yS(t,X(t,y,\eta),y) = -\eta.\]
Taking derivatives in $\eta$ yields
\[\partial^2_{yx}S\partial_{\eta}X = -I.\]
This shows that
\[C = -\left(\partial_{\eta}X\right)^{-1} = \partial^2_{yx}S.\]
Taking derivatives in $y$ yields
\[\partial^2_{yx}S\partial_yX + \partial^2_{yy}S = 0\implies \partial^2_{yy}S = -\partial^2_{yx}S\partial_yX = \left(\partial_{\eta}X\right)^{-1}\partial_yX.\]
This gives
\[-\left(\partial_{\eta}X\right)^{-1}\left(I-\partial_yX\right) = -\left(\partial_{\eta}X\right)^{-1} + \left(\partial_{\eta}X\right)^{-1}\partial_yX = \partial^2_{yx}S + \partial^2_{yy}S.\]
We also have
\[\partial_xS(t,X(t,y,\eta),y) = \Xi(t,y,\eta).\]
Taking derivatives in $\eta$ yields
\[\partial^2_{xx}S\partial_{\eta}X = \partial_{\eta}\Xi\implies \partial^2_{xx}S = \partial_{\eta}\Xi\left(\partial_{\eta}X\right)^{-1}.\]
Thus
\[-\left(I-\partial_{\eta}\Xi\right)\left(\partial_{\eta}X\right)^{-1} = \partial^2_{yx}S + \partial^2_{xx}S.\]
Hence
\begin{align*}
B &= \left(-\left(\partial_{\eta}X\right)^{-1}\left(I-\partial_yX\right)\right) + \left( -\left(I-\partial_{\eta}\Xi\right)\left(\partial_{\eta}X\right)^{-1}\right) \\
&= \left(\partial^2_{yx}S + \partial^2_{yy}S\right) + \left(\partial^2_{yx}S + \partial^2_{xx}S\right),
\end{align*}
giving the desired statement. Finally, taking derivatives in $y$ yields
\[\partial^2_{xx}S\partial_yX + \partial^2_{xy}S = \partial_y\Xi.\]
Thus
\begin{align*}
A &=\partial_y\Xi - \left(I-\partial_{\eta}\Xi\right)\left(\partial_{\eta}X\right)^{-1}\left(I-\partial_yX\right)\\
&=\partial^2_{xx}S\partial_yX + \partial^2_{xy}S + \left(\partial^2_{yx}S + \partial^2_{xx}S\right)\left(\partial^2_{yx}S\right)^{-1}\left(\partial^2_{yx}S + \partial^2_{yy}S\right) \\
&=\partial^2_{xx}S\left(-\partial^2_{yx}S\right)^{-1}\partial^2_{yy}S + \partial^2_{xy}S + \partial^2_{yx}S + \partial^2_{yy}S + \partial^2_{xx}S + \partial^2_{xx}S\left(\partial^2_{yx}S\right)^{-1}\partial^2_{yy}S \\
&=\partial^2_{xx}S + \partial^2_{xy}S + \partial^2_{yx}S + \partial^2_{yy}S,
\end{align*}
as claimed.
\end{proof}

\begin{proof}[Proof of Lemma \ref{lem:partialtS}]
We use that $\partial_tS = -H$, i.e.
\begin{equation}
\label{eq:partialtS}
\partial_tS(t,X(t,y,\eta),y) = -H(y,\eta).
\end{equation}
Taking derivatives in $t$ in \eqref{eq:partialtS} gives
\[\partial^2_{tt}S + \partial^2_{tx}S\dot X = 0.\]
Taking derivatives in $\eta$ in \eqref{eq:partialtS} gives
\[\partial^2_{tx}S\partial_{\eta}X = -\frac{\partial H}{\partial\eta} = -\dot X^T.\]
Thus
\[\partial^2_{tx}S = -\dot X^T\left(\partial_{\eta}X\right)^{-1} = \dot X^T\partial^2_{yx}S.\]
In particular, this gives
\[\partial^2_{tt}S = -\partial^2_{tx}S\dot X = -\dot X^T\partial^2_{yx}S\dot X = -(X^1)^2\frac{\partial^2S}{\partial y^1\partial x^1} = -(X^1)^2C_{11}.\]
Taking derivatives in $y$ in \eqref{eq:partialtS} gives
\[\partial^2_{tx}S\partial_yX + \partial^2_{ty}S = -\frac{\partial H}{\partial y} = \dot\Xi^T\]
and hence
\[\partial^2_{ty}S = \dot\Xi^T-\partial^2_{tx}S\partial_yX = \dot\Xi^T-\dot X^T\partial^2_{yx}S\partial_yX = \dot\Xi^T+\dot X^T\partial^2_{yy}S.\]
It follows that
\[\partial^2_{tx}S+\partial^2_{ty}S = \dot\Xi^T + \dot X^T\left(\partial^2_{yx}S + \partial^2_{yy}S\right),\]
or equivalently
\[\partial^2_{xt}S+\partial^2_{yt}S = \dot\Xi + \left(\partial^2_{xy}S + \partial^2_{yy}S\right)\dot X.\]
Finally, from
\[\partial_xS(t,X(t+s,y,\eta),X(s,y,\eta)) = \Xi(t+s,y,\eta),\]
taking the derivative in $s$ at $s=0$ yields
\[\left(\partial^2_{xx}S+\partial^2_{xy}S\right)\dot X = \dot \Xi.\]
Hence
\[\partial^2_{xt}S+\partial^2_{yt}S = \left(\partial^2_{xx}S+\partial^2_{xy}S + \partial^2_{xy}S + \partial^2_{yy}S\right)\dot X,\]
which gives the desired result.
\end{proof}

\section{Interior and short-time reflective propagator}

In this section, we recall the oscillatory integral description of
interior propagators following Meinrenken \cite{Me:92} and compute the
short-time reflective propagator near $Y$ using a WKB-type
method. Recall that
$$U(t) := e^{-itP_h/h}$$
denotes the Schr\"odinger propagator of $P_h$. As is usual in
  such computations, we will switch now to taking the propagator to
  act on half-densities in $X$, as this simplifies much of the
  bookkeeping without affecting the trace (as we may trivialize the
  half-density bundle with a choice of global half-density that does
  not affect the trace.)
  
\subsection{Interior propagators}
\label{sec:int-pgtr}
Let $A,B \in \Psi_h(X)$ be compactly microsupported, with the microsupport  
$\opWF_h(A),\opWF_h(B)\in T^*( X\backslash Y)$. Assume that for any
$(z,\xi) \in \WF'(B)$, $(w,\eta) \in \WF'(A)$,
and $t \in [t_0-\ep, t_0+\ep]$, there
exists at most one branching null bicharacteristic
$\gamma=\gamma_{z,w}$ of $P_h$ with $\gamma(0)=(z,\xi)$,
$\gamma(t)=(w,\eta)$, and that $z,w$ are not conjugate points.
\emph{Assume that $\gamma_{z,w}$ lies entirely over $X \backslash Y$}.  We summarize
\cite[Theorem~1]{Me:92} as the following lemma on the interior
microlocal propagator.  Here $\mu_\gamma$ is the number of conjugate
points along $\gamma$ (hence by the Morse index theorem is the Morse index of the variational problem with fixed
endpoints), and $S_\gamma$ is the classical action along the trajectory
$\gamma$:
\begin{equation}\label{actiondef}
S_\gamma(z,w) = \int_{\gamma_{z,w}} \frac 14 \abs{\dot{z}}^2-V \, dt.
\end{equation}
The $\Delta_{\gamma}$ term in the formula below is known as the van
Vleck determinant, and is given in terms of the action by
$$
\Delta_{\gamma} = \det \frac{\pa^2 S_{\gamma_{z,w}}}{\pa x \pa y};
$$
it can be alternatively interpreted in terms of the derivative of the
exponential map, and its square root frequently arises in Fourier
integral operator constructions as a half-density factor.

\begin{lemma}
\label{lem:int-pgtr}
    For any $A,B \in \Psi^0_h(X)$ satisfies the above conditions, the Schwartz kernel of interior microlocal propagator $AU(t)B$ is given by
    \begin{gather*}
    {(2\pi i h)^{-\frac{n}{2}}} e^{iS_\gamma/h}
 \abs{\Delta_\gamma}^{1/2} e^{-i\frac{\pi}{2}\mu_\gamma} a(x,\pa_x S_{\gamma}) b(y,-\pa_y S_{\gamma}) \abs{dz\, dw}^{1/2} \big(1+O(h)\big),
\end{gather*}
where the $O(h)$ error term has a
full asymptotic expansion in powers of $h$.
\end{lemma}

\subsection{Reflective propagators}
\label{sec:ref-pgtr}

We now derive the form of the reflective and transmitted propagators
near hyperbolic points over $Y$ (i.e., transverse interaction with the
interface).  The
parametrix construction below is due to Oran Gannot \cite{gannot}, to whom we are
grateful for permission to use this computation.

We study the structure of the Schr\"odinger propagator microlocally
near an interaction with a hyperbolic point. Recall that 
in Riemannian normal coordinates $(x_1,x')$ near $Y$ with respect to $g$,
\[
p(x,\xi) = \xi_1^2 + \langle K(x)\xi',\xi' \rangle + V(x)
\]
for a positive definite matrix $K(x)$. Fix a point $(0,x',\xi_1,\xi')$ with $\xi_1 \neq 0$, and then let $E_0 = p(q_0)$.

Let $B \in \Psi_h^\Comp(X)$ have support disjoint from $T^*_YX$
  and wavefront set near a point $q_0 \in T^*X$ close to $\hyp$ satisfying $x_1(q_0) < 0$. We seek an oscillatory integral representation for 
\[
U_B(t) = e^{-itP/h}B,
\]
at least for small $|t|$. We construct the kernel of $U_B(t,x,y)$ separately for $x_1 < 0$ and $x_1 > 0$, subject to a matching condition along $Y = \{x_1 = 0 \}$.

Our ansatz for the propagator will be defined piecewise as follows:
\begin{equation}\label{propagatoransatz}
U_B(t,x,y)  = \begin{cases} \sum_{\bullet \in \{\INC,\REF\}} e^{\sfrac{i}{h}
  \phi^{\bullet}(t,x, y)}a^\bullet(t,x,y)\, |dx\, dy|^{1/2},
&x_1<0,\\
e^{\sfrac{i}{h}
  \phi^{\TRAN}(t,x, y)}a^\TRAN(t,x,y)\, |dx\, dy|^{1/2}, &x_1>0.
\end{cases}\end{equation}
The amplitudes are denoted $a^\bullet$ with $\bullet=\INC,\REF,\TRAN$
for ``incident, reflected, transmitted;'' they are assumed to have asymptotic expansions
\[
a^\bullet \sim a^\bullet_0 + h a^\bullet_1 + \ldots.
\]
Thus the phases $\phi^\bullet$ and amplitude $a^\bullet$ must solve the usual eikonal and transport equations.
The continuity conditions across $\{x_1 = 0\}$ required
to make the ansatz $\mathcal{C}^1$ are
\begin{equation}
    \label{eq:trans}
\begin{cases} 
U_B^\mathsf{I}(t,0,x',y)  + U_B^\mathsf{R}(t,0,x',y) = U_B^\mathsf{T}(t,0,x',y),\\
\pa_{x_1} U_B^\mathsf{I}(t,0,x',y)  + \pa_{x_1} U_B^\mathsf{R}(t,0,x',y) = \pa_{x_1} U_B^\mathsf{T}(t,0,x',y).
\end{cases} 
\end{equation}

\subsubsection{Construction of phase function}

Fix any point $\overline{y} \in Y$.
Let $U_1\subset X$ be a small neighborhood of $\overline{y}$ and let $U_0$ be a connected
subneighborhood such that there exists $T_0 \ll 1$ so that
for all $t \in (0, T_0)$, for all $x=(x_1,x'),\ y=(y_1,y')$ both in
$U_0$  there exists exactly
one classical bicharacteristic (i.e., an ordinary solution to Hamilton's equations) $\gamma(t)$ such that
\begin{equation}\label{unique}\begin{aligned}
  \pi(\gamma(0))=y,\quad  
  \pi(\gamma(t))=x,\quad
\pi(\gamma(s)) \in U_1,\ s \in [0, t].
\end{aligned}\end{equation}
We may further assume that if $x_1$ and $y_1$ are both negative, then
there exists at most one \emph{reflected} bicharacteristic satisfying
\eqref{unique}; recall that this is a concatenation of ordinary
bicharacteristics along which $(x_1,x',\xi')$ are continuous, while
$\xi_1$ jumps (between nonzero values) by switching sign.
Such a bicharacteristic may of course fail to exist,
depending on the convexity of $Y$ with respect to the
flow.

We now take $S$ to be the classical action
\begin{equation}
\label{eq:act}
    S(t,x,y)=t\cdot \tau (t, x, y)+ \int_{\gamma}\xi dx
\end{equation}
where $\gamma$ is the unique ordinary (i.e., unbroken) integral curve of Hamiltonian vector
field $H_p$ with $p = |\xi|^2_g +V(x) = -\tau$ from $T^*_yX$ to
$T^*_xX$ (cf.\ \cite{Ch:80}). The action thus satisfies the eikonal equation 
\begin{equation}
    \label{eq:eik}
    \partial_t S + |\nabla S|^2 + V(x) = 0.
  \end{equation}
We further decorate $S$ with the superscript $\TRAN$ or $\INC$ to
denote ``transmitted'' or ``incident'' according to whether the signs
of $x_1,y_1$ are identical ($\INC$) or opposite ($\TRAN$).  Likewise,
if $x_1,y_1<0$ are sufficiently small and $x'$ and $y'$ are close,
we define $S^\REF(t,x,y)$ to be the corresponding action integral
along the unique \emph{reflected} bicharacteristic connecting $x$ and $y$
(which can equivalently be defined as stationary point of the action along
broken trajectories with fixed endpoints as in \cite{WuYaZo:24}).
In particular, then,
\begin{equation}
    \label{eq:ref-act}
    S^{\REF}_{\gamma} (t, x, y)= t\cdot \tau (t, x, y)+ \int_{\gamma_1}\xi dx + \int_{\gamma_2}\xi dx
\end{equation}
where $\gamma_1$ is the integral curve of $H_p$ from $(y,\eta)$ to
$w_-=\gamma\cap T^*_{Y}X\cap \{\xi_1> 0\}$ and $\gamma_2$ is the
integral curve of $H_p$ from
$w_+=\gamma\cap T^*_{Y}X \cap \{\xi_1< 0\}$ to $(x,\xi)$. The reflected phase function \eqref{eq:ref-act}
satisfies the eikonal equation \eqref{eq:eik} along the generalized
branching null-bicharacteristic $\gamma$ as
$|\xi_{w_+}|^2=|\xi_{w_-}|^2$, hence energy is conserved in the
reflection, and the differential of the action still lies in
the characteristic set at each time.

Taking
$$
\phi^\bullet=S^\bullet,\ \bullet=\INC,\TRAN,\REF
$$
in the ansatz thus solves the eikonal equation \eqref{eq:eik} for each piece of the
propagator; we note further that this choice interacts well with the
matching conditions \eqref{eq:trans}.  In particular, we certainly obtain
\begin{equation} \label{eq:phaseequality}
\action^\INC(t,0,x',y) = \action^\REF(t,0,x',y) = \action^\TRAN(t,0,x',y),
\end{equation}

Since 
$\partial_{x_1}S^{\INC} = \xi_{w-,1} = -\xi_{w+,1} =
-\partial_{x_1}S^{\REF}$,
we also have the simple relationship
\begin{equation}\label{incref}
    \pa_{x_1}\action^\REF(t,0,x',y) = -\pa_{x_1}\action^\INC(t,0,x',y).
  \end{equation}
  Since $S^\INC$ and $S^\TRAN$ both satisfy the same eikonal
    equation at $x_1=0$ and have the same value there, we likewise
    obtain
    \begin{equation}\label{inctran}
    \pa_{x_1}\action^\TRAN(t,0,x',y) = \pa_{x_1}\action^\INC(t,0,x',y).
      \end{equation}

\subsubsection{Construction of Schr\"odinger kernels}
Now we consider the amplitude
equations. By matching the values of $U_B(t,x,y)$ along $\{x_1 = 0\}$
as required by \eqref{eq:trans},
  we obtain the initial condition at $Y$
\begin{equation}\label{ampvalue}
a^\INC_k(t,0,x',y) + a^\REF_k(t,0,x',y) = a^\TRAN_k(t,0,x',y)\text {
  for all } k.
\end{equation}
By matching normal derivatives (the second requirement of
\eqref{eq:trans}), we likewise obtain by \eqref{incref}, \eqref{inctran}
\begin{align}
\pa_{x_1}\action^\INC(t,0,x',y) & (a^\INC_k(t,0,x',y) - a^\REF_k(t,0,x',y)) - i \pa_{x_1} a^\INC_{k-1}(t,0,x',y) - i \pa_{x_1} a^\REF_{k-1}(t,0,x',y) \nonumber \\
= & \pa_{x_1}\action^\INC(t,0,x',y)a^\TRAN_k(t,0,x',y) - i \pa_{x_1}
a^\TRAN_{k-1}(t,0,x',y)\text { for all } k \label{ampderiv}
\end{align}
(with the convention that $a^\bullet_{-1}=0$).
We remark, crucially, that we can now produce a parametrix by solving transport
equations to any desired order: we solve for $a^{\INC}$ up to
$Y=\{x_1=0\}$, giving smooth data on this hypersurface at every order
in $k$.  This gives initial data for both $a^{\TRAN}_k$ and
$a^{\REF}_k$ for each $k$ by \eqref{ampvalue}, \eqref{ampderiv} (see
also the reformulation \eqref{abc} below); the transport equations
may then be solved in turn on $x_1\geq 0$ resp.\ $x_1\leq 0$.  Borel
summing the resulting series gives a solution to the Schr\"odinger
equation modulo $O(h^\infty)$, valid across $Y$.  In what follows we
elucidate the structure of this parametrix, and in particular, the
threshold $k$ up to which $a^{\INC}_k$ and $a^{\TRAN}_k$ agree across $Y$ and
$a^{\REF}_k$ vanishes. There is of course no a priori guarantee that
this parametrix is a good approximation to the actual solution simply
by virtue of solving the equation modulo $O(h^\infty)$, but we 
take up this question below in the proof of Proposition~\ref{prop:ref-pgtr},
the main result of this section, where we use the parametrix construction to deduce the
microlocal structure of the propagator.

Now recall that $\pa^{k}_{x_1} V(0-,x') = \pa^{k}_{x_1}V(0+,x')$
for $k < \kk$, but that in general,
\[
\pa^{\kk}_{x_1} V(0-,x') \neq \pa^{\kk}_{x_1}V(0+,x').
\]
\emph{Our immediate goal is to show that $a^\REF_k(t,x,y) =0$ for $0
  \leq k < \kk$, and that incident and transmitted coefficients match
at these orders.} We begin with
for $k=0$.  Subtracting \eqref{ampvalue} from
\eqref{ampderiv} with $k=0$, we obtain  
\[
a^\REF_0(t,0,x',y) = 0.
\]
We now turn to the transport equations to extend this equation into
$X\backslash Y$.

We record the form of the transport equations in general for a
Schr\"odinger equation whose Laplacian is with respect to an arbitrary
metric. Recalling that $Q := hD_t -h^2\Delta_g + V$, we have
 \[
 e^{-i\action/h}Q e^{i\action/h} = (hD_t + \pa_t \action) +
 \frac{1}{\sqrt{g}}((h/i) \pa_i + \pa_i \action)g^{ij}\sqrt{g} ((h/i) \pa_j +
 \pa_j \action) + V 
 \]
 Expanding out the spatial derivative part gives
 \[
 -h^2\Lap_g + 2(h/i)g^{ij} (\pa_i \action) \pa_j + (h/i)\Lap_g\action
 +g^{ij}\pa_i \action \pa_j\action.
 \]
As $$\pa_t \action + g^{ij}\pa_i \action \pa_j\action+V=0$$ by the eikonal
 equation, the transport equations giving the vanishing of the
 asymptotic expansion of the remaining terms are 
\begin{align*}
&(\pa_t + 2g^{ij} (\pa_i \action)\pa_j+ \Delta_g \action)a_0 = 0, \\
&(\pa_t + 2g^{ij} (\pa_i \action)\pa_j + \Delta_g \action)a_k = i \Delta_g a_{k-1};
\end{align*}
here we have dropped (and will continue to omit) the superscripts on
the $S$ and $a_j,$ as we will be working at the interface where the
incident and transmitted actions agreed, and each
  $a_j^\bullet$, $\bullet\in \{\TRAN,\INC,\REF\}$ must satisfy the
  transport equation (on the relevant side of the interface $\{x_1=0\}$).

We introduce two pieces of notation to streamline the
  bookkeeping what follows.  First, let $\patan f$ stand for all
  possible products of vector fields $\pa_t,
  \pa_{x'}$ applied to $f$, i.e., all derivatives of $f$ in
  tangential variables only.  We will not track orders of such
  operators; they are merely taken to be finite.  (Recall that
  tangential derivatives of $V$ are all continuous.) Second, we further
  write
  $$
\mathrm{Fun}(f_1,\dots f_n)
$$
to denote some function of the arguments $\patan f_j$, i.e., \emph{we allow
dependence on tangential derivatives} without writing it explicitly in
the notation.  (The point is to track just the crucial $\pa_{x_1}$
derivatives.)
\begin{lemma}\label{lemma:phiderivs}
	Let $j\geq 2$. Then
	\begin{equation}\label{woodshole8.5}
	\partial_{x_1}^{j} \action = -\frac{\pa_{x_1}^{j-1}
          V}{2\pa_{x_1}\action} + G_j(\action, V, \pa_{x_1} V,\ldots,
       \pa_{x_1}^{j-2}V)=\mathrm{Fun}(\action, V, \pa_{x_1} V, \dots, \pa_{x_1}^{j-2}V).
	\end{equation}	
\end{lemma}
\begin{proof}
	The proof is by induction, beginning with $j=2$. First consider the eikonal equation
\[
\partial_t \action + (\partial_{x_1}\action)^2 + k^{\alpha\beta}(x)\partial_{x^\alpha}\action \cdot  \partial_{x^\beta}\action + V(x)  = 0.
\] 
First note that we can explicitly solve this equation for $\pa_{x_1}\action$, yielding
\begin{equation}\label{pa1phi}
\partial_{x_1}\action = (-\partial_t \action - k^{\alpha\beta}(x)\partial_{x^\alpha}\action \cdot  \partial_{x^\beta}\action - V)^{1/2}.  
\end{equation}
Thus we can express $\pa_{x_1}\action$ as a smooth function of
$(\patan \action,\patan V)$ in a neighborhood of $(\action(x_0), V(x_0))$. 

Now to compute the $j$'th normal derivative of $\action$ for $j\geq 2$, differentiate the eikonal equation in $x_1$. Thus, to handle the term $j=2$, we obtain
\[
\pa_{x_1}\pa_t \action  + 2 \pa^2_{x_1} \action \cdot \pa_{x_1}\action + 2k^{\alpha\beta} \pa_{x^\alpha}\action \cdot \pa_{x_1}\pa_{x^\beta}\action + (\pa_{x_1}k^{\alpha\beta})\partial_{x^\alpha}\action \cdot  \partial_{x^\beta}\action + \pa_{x_1}V.
\]
Now we need to solve for $\pa_{x_1}^2\action$. This easily yields an expression of the form
\[
\pa_{x_1}^2\action = -\frac{ \pa_{x_1}V}{2\pa_{x_1} \action} +
\fun(x,\action,\pa_{x_1}\action)
\]
for an appropriate smooth function.
Note now that we can eliminate the $\pa_{x_1}\action$ dependence in $F_2$ in favor of $(\action,V)$ dependence using the equation above. Thus we indeed have
\[
\pa_{x_1}^2\action = -\frac{ \pa_{x_1}V}{2\pa_{x_1} \action} +
G_2(\patan\action,\patan V).
\]
In the inductive step, suppose that
\[
\partial_{x_1}^{j} \action = -\frac{\pa_{x_1}^{j-1}
  V}{2\pa_{x_1}\action} + G_j(\patan\action, \patan V, \pa_{x_1}
\patan V,\ldots, \patan\pa_{x_1}^{j-2}V).
\]	
Differentiating both sides with respect to $x_1$,
\begin{equation}
    \begin{split}
\partial_{x_1}^{j+1} \action & = -\frac{\pa_{x_1}^{j} V}{2\pa_{x_1}\action} 
+ \frac{\pa_{x_1}^{j-1} V}{2(\pa_{x_1}\action)^2} \pa_{x_1}^2 \action + \pa_{x_1}\action \cdot \pa_{t_0} G_j(\patan\action, \patan V, \patan\pa_{x_1} V,\cdots, \patan\pa_{x_1}^{j-2}V) \\ 
& \quad \quad \quad \quad \quad+ F_j(\patan\action, \patan V, \patan\pa_{x_1} V,\ldots, \patan\pa_{x_1}^{j-2}V, \patan\pa_{x_1}^{j-1}V).
    \end{split}
\end{equation}
In the second line we can use the inductive hypothesis to exchange
$\pa_{x_1}\action$ and $\pa_{x_1}^2\action$ for $(\action, V,\pa_{x_1}V)$
dependence, thus completing the proof, since the third line is already
in the desired form (and \eqref{pa1phi} yields the cruder functional
dependence given by the second inequality in \eqref{woodshole8.5}).
\end{proof}

Next we look at the structure of the transport equations. In our normal coordinates they take the form 
\begin{equation}\label{transport1}
(\pa_t + 2\pa_{x_1} \action \cdot \pa_{x_1} + 2k^{\alpha\beta} \pa_{x^\alpha} \action \cdot  \pa_{x^\beta} + \Delta_g \action)a_k = i\Delta_g a_{k-1}.
\end{equation}
Consider the structure of the first equation (i.e., $k=0$). Here we are interested in computing $\pa_{x_1}a_0$. Notice that we can write
\begin{equation}\label{Deltaphi}
\Delta_g \action = E(x) \pa_{x_1} \action + \pa_{x_1}^2 \action +  \Delta_k \action, \quad E(x) = (\det k)^{-1/2} \pa_{x_1} (\det k)^{1/2}.
\end{equation}
Thus if we solve for $\pa_{x_1} a_0$, we get
\begin{equation}\label{a0}
\pa_{x_1} a_0 = -\frac{\pa_{x_1}^2 \action}{2\pa_{x_1}\action}a_0 +
\fun( a_0,\action).
\end{equation}
We also need to compute higher order derivatives of $a_0$. For this,
we note inductively that
\begin{equation}\label{a0j}
\pa_{x_1}^ja_0 = -\frac{\pa_{x_1}^{j+1} \action}{2\pa_{x_1}
  \action}a_0 +\fun(a_0, \action,\cdots,\pa_{x_1}^{j}\action),
\end{equation}
where $j \geq 1$. Notice that we can also write this in the form
\[
\pa_{x_1}^ja_0 = \fun(a_0, \action,\cdots, \pa_{x_1}^{j+1}\action).
\]

Now we compute $\pa_{x_1}a_1$. Returning to \eqref{transport1} and
proceeding as for $a_0,$ we find that
\begin{equation}\label{pa1a1}
\pa_{x_1} a_1 =  \fun(a_1,\action, \pa_{x_1}\action, \pa_{x_1}^2\action)+ \frac{i \Delta_g a_{0}}{2 \pa_{x_1}\action}.
\end{equation}
Consider the Laplacian term on the right hand side. Note that
\begin{align*}
\frac{i\Delta_g a_0}{2\pa_{x_1}\action}  &= \frac{i\pa_{x_1}^2
                                        a_0}{2\pa_{x_1}\action} +
                                        \mathrm{Fun}(a_0,\pa_{x_1}a_0,
                                        \pa_{x_1}\action) \\ 
&= \frac{i\pa_{x_1}^2 a_0}{2\pa_{x_1}\action} + \mathrm{Fun}(a_0,\action,\pa_{x_1}\action,\pa_{x_1}^2\action) \\
& = -\frac{i\pa_{x_1}^3\action}{(2\pa_{x_1}\action)^2}a_0 + \mathrm{Fun}(a_0,\action,\cdots, \pa_{x_1}^2\action)
\end{align*}
by \eqref{a0}, \eqref{a0j}.
Returning to \eqref{pa1a1} we now obtain
\[
\pa_{x_1}a_1 = -\frac{i\pa_{x_1}^3\action}{(2\pa_{x_1}\action)^2}a_0 +
\mathrm{Fun}(a_0,a_1,\action,\pa_{x_1} \action,\pa_{x_1}^2\action).
\]
We now continue inductively to show
\begin{equation}\label{pa1ak}
\begin{aligned}
\pa_{x_1} a_k &= -\frac{i^k \pa_{x_1}^{k+2}\action}{(2\pa_{x_1} \action)^{k+1}}a_0 + \mathrm{Fun}(a_0,a_1,\cdots a_k,\action,\cdots,\pa_{x_1}^{k+1}\action) \\
&= \mathrm{Fun}(a_0,a_1,\cdots a_k,\action,\cdots,\pa_{x_1}^{k+2}\action). 
\end{aligned}
\end{equation}
Notice in particular that the inductive hypothesis \eqref{pa1ak} implies that
\[
\pa_{x_1}^2 a_k = -\frac{i^k \pa_{x_1}^{k+3}\action}{(2\pa_{x_1} \action)^{k+1}}a_0 + \mathrm{Fun}(a_0,\cdots,a_k,\pa_{x_1}a_0,\cdots\pa_{x_1}a_k,\action,\cdots,\pa_{x_1}^{k+2}\action).
\]
The last term can be written as
\[ \mathrm{Fun}(a_0,\cdots,a_k,\action,\cdots,\pa_{x_1}^{k+2}\action).
\]
To establish \eqref{pa1ak} inductively, note that the transport equation
\eqref{transport1} yields
\[
\pa_{x_1}a_{k+1} = \mathrm{Fun}(a_{k+1},\action,\pa_{x_1}\action,\pa_{x_1}^2\action) + \frac{i \Delta_g a_{k}}{2 \pa_{x_1}\action}.
\]
We use the inductive hypothesis for the last term to write
\begin{align*}
\frac{i\Delta_g a_k}{2\pa_{x_1}\action}  &= \frac{i\pa_{x_1}^2 a_k}{2\pa_{x_1}\action} + \mathrm{Fun}(a_k,\pa_{x_1}a_k) \\ 
&= \frac{i\pa_{x_1}^2 a_k}{2\pa_{x_1}\action} + \mathrm{Fun}(a_0,\ldots,a_k,\action,\ldots,\pa_{x_1}^{k+2}\action) \\
& = -\frac{i^{k+1}\pa_{x_1}^{k+3}\action}{(2\pa_{x_1}\action)^{k+2}}a_0 + \mathrm{Fun}(a_0,\ldots,a_k,\action,\cdots,\pa_{x_1}^{k+2}\action).
\end{align*}
This completes the proof by induction. The final step is to replace
the dependence on $\action$ with dependence on $V$ using
Lemma~\ref{lemma:phiderivs}. This tells us that
\begin{equation}
	\label{eq:d1ak}
	\pa_{x_1} a_k =
        \frac{i^{k}\pa_{x_1}^{k+1}V}{(2\pa_{x_1}\action)^{k+2}}a_0
        +
        \mathrm{Fun}(a_0,a_1,\cdots,a_k,\action,V,\ldots,\pa_{x_1}^kV). 
\end{equation}

Now we come back to the matching conditions \eqref{ampvalue}, \eqref{ampderiv}.
Let $\psi$ denote the restriction of
$\pa_{x_1}\action^\INC$ to $\{x_1 =0\}$. Then the conditions read
\[
a^{\INC}_k + a^{\REF}_k = a^{\TRAN}_k, \quad \psi(a^{\INC}_k -
a^{\REF}_k) - i\pa_{x_1}(a^{\INC}_{k-1} + a^{\REF}_{k-1}) = \psi
a^{\TRAN}_k - i\pa_{x_1} a^{\TRAN}_{k-1},\text{ on } Y.
\]

If we multiply the first equation by $\psi$ and add, resp. subtract the second equation we obtain
\begin{equation}\label{abc}
\begin{aligned}
2\psi (a^{\INC}_k - a^{\TRAN}_k) &= i\pa_{x_1}(a^{\INC}_{k-1} + a^{\REF}_{k-1} - a^{\TRAN}_{k-1}), \\
2\psi a^{\REF}_k &= i\pa_{x_1}(a^{\TRAN}_{k-1} - a^{\INC}_{k-1}- a^{\REF}_{k-1}).
\end{aligned}
\end{equation}
Now we start with $k=0$, which tells us that along $\{x_1 =0 \}$, we
have $a^{\REF}_0 = 0$ and $a^{\INC}_0 = a^{\TRAN}_0$. Observe that this implies $a^{\REF}_0= 0$
identically, since $a^{\REF}_0$ satisfies the transport equation
  \eqref{transport1} with vanishing initial data at the hypersurface
  $\{x_1=0\}$.

We now further claim that if 
$V,\dots, \pa_{x_1}^{k}V$ are
  continuous across $Y$ then $a^{\INC}_j=a^{\TRAN}_j$ and $a^{\REF}_j=0$ on $Y$ for all $j
  \leq k$.  We show this inductively, having established it above for
  $k=0$; the assumed continuity of $V$ was of course tacitly employed in this argument.

\emph{Suppose then that $V,\dots, \pa_{x_1}^{k+1}V$ are
  continuous} and that $a^{\INC}_j = a^{\TRAN}_j$ and $a^{\REF}_j = 0$ on $Y$ for $j \leq
  k$. We would like to conclude that $a^{\INC}_{k+1} = a^{\TRAN}_{k+1}$ and
  $a^{\REF}_{k+1} = 0$ on $Y$. Since the $a^{\REF}_j$ satisfy \eqref{transport1}, the inductive hypothesis implies
  that $a^{\REF}_j = 0$ identically for $j \leq k$. In particular
  $\pa_{x_1}a^{\REF}_k = 0$ identically, and thus in particular along $Y$. Equation \eqref{eq:d1ak}  then yields
\[
\pa_{x_1} a^{\INC}_k = \pa_{x_1} a^{\TRAN}_k \text{ along } Y
\]
since $\pa_{x_1}^{k+1}V$ is continuous and $a^{\INC}_j = a^{\TRAN}_j$ for $j \leq
k$. Equation \eqref{abc} then yields $a^{\INC}_{k+1} = a^{\TRAN}_{k+1}$ along $Y$ and
$a^{\REF}_{k+1}=0$  (globally, by the transport
equation). Since the continuity of normal derivatives of
  $V$ holds up to the $\kk-1$'th derivative, we
  have thus established that along $Y$,
\begin{equation}\label{eq:continuity}
a^{\INC}_{k} = a^{\TRAN}_{k}, \quad a^{\REF}_{k}=0\ \text{ for
  all } k \leq \kk-1.
\end{equation}

Consequently, no reflection or jump between incident or
  transmitted waves occurs up to and including $O(h^{\kk-1})$ terms.
    We complete this section by giving an explicit expression of the
  leading order reflection coefficient $a_{\kk}^{\REF}$ in terms of the potential $V$.

Let $J(x')$ be the function on $Y=\{x_1=0\}$ given by the jump in the
$\kk$'th normal derivative of $V$:
$$
J(x') = \pa_{x_1}^{\kk}V(0+,x')-\pa_{x_1}^{\kk}V(0-,x').
$$
As this function on
$Y$ manifestly depends on the chosen orientation of $NY$, we remark
that the more correct global notation (as used in the introduction) is
$J(x',\bv),$ where $\bv \in T_YX$ denotes a vector transverse to $Y$ and
positively oriented in the normal coordinate system in which the jump
is computed.  Hence in particular, $J(x',-\bv)=-J(x',\bv)$.

By \eqref{abc},
$$2\psi a^{\REF}_k = i\pa_{x_1}(a^{\TRAN}_{k-1} - a^{\INC}_{k-1}- a^{\REF}_{k-1}).$$
Since $a_k^{\REF}(t,x, y)=0$ for $0\leq k<\kk$, we have 
$$2\psi a^{\REF}_{\kk} = i\pa_{x_1}(a^{\TRAN}_{\kk-1} - a^{\INC}_{\kk-1}).$$
Now apply \eqref{eq:d1ak} to $a^{\INC}_{\kk-1}$ and $a^{\TRAN}_{\kk-1}$
and note that the second part in the equation \eqref{eq:d1ak} is the
same for $a^{\INC}_{\kk-1}$ and $a^{\TRAN}_{\kk-1}$. This yields
\begin{equation}
	\label{eq:ref-coeff0}
	a_{\kk}^{\REF} (t,0,x',y)= \frac{i^{\kk} J(x')}{(2 \pa_{x_1}\action)^{\kk+2}} a^{\INC}_0(t,0,x',y);
      \end{equation}
solving the transport equation of course extends this to a function
defined on $x_1<0,$ the leading order amplitude of the reflected
wave.  
Recalling that $\action=\action^{\INC}$ along $Y$ and that this is a
      generating function of the symplectomorphism given by
      bicharacteristic flow, we may of course write
      $$
\pa_{x_1} \action=\xi_1 =\xi_1(t,0,x',y),
$$
where $\xi_1(t,0,x',y)$ is the normal momentum (in the incident direction)
of the bicharacteristic connecting $y$ to $(0,x')$ in time $t$.
Thus, we finally arrive at 
\begin{equation}
	\label{eq:ref-coeff}
	a_{\kk}^{\REF} (t,0,x',y)= \frac{i^{\kk} J(x')}{(2 \xi_1)^{\kk+2}} a^{\INC}_0(t,0,x',y)
      \end{equation}
as the leading-order nonvanishing term in the reflected
      propagator.
      
   \begin{definition}
Let the \emph{reflection coefficient} be the quantity
\begin{equation}\label{reflectioncoefficient}
    \rc(t,x,y): =  \frac{i^{\kk} J(x')}{(2 \xi_1)^{\kk+2}}
\end{equation}
evaluated at the point of
reflection $(x',\xi)\in T^*_Y X$ of the reflected bicharacteristic from $y$ to $x$ in time $t$.
      \end{definition}

 We can now
collect the outcome of our parametrix construction in the following
result about the structure of the microlocalized
propagator for a single reflection or transmission:
\begin{proposition}
\label{prop:ref-pgtr} For $T>0$ sufficiently small,
	for $A,B\in \Psi^{0}_{h,\Comp}(X)$ near $Y$ with
        points in $\opWF_h(A)$ and $\opWF_h(B)$ related by at most one
        reflected branching bicharacteristic of length in $(0, T)$, the microlocalized
        reflective propagator is given on $t\in (0,T)$ by 
	\begin{equation}
		\label{eq:ref-pgtr}
        \begin{split}
        A e^{-itP_h/h} B = (2\pi &i h)^{-\frac{n}{2}}h^\kk
        e^{\sfrac{i}{h} S^\REF_{\gamma}(t,x,y) }
        \abs{\Delta_{\gamma}}^{\frac{1}{2}} \\ & a(x,\partial_x
        S^\REF_{\gamma})\rc(t,x,y) b(y,-\partial_y
        S^\REF_{\gamma})\, \abs{ dx\, dy}^{\frac{1}{2}}(1+O(h)), 
        \end{split}
	\end{equation} 
	with $\rc(t,x,y)$ the reflection coefficient
        \eqref{reflectioncoefficient}, evaluated at the point $x'$ and
        normal momentum $\xi_1$ of
        reflection of the bicharacteristic from $y$ to $x$, and $a(x,\xi)$ and $b(y,\eta)$ are symbols of $A$ and $B$ correspondingly.

        For $A,B\in \Psi^{0}_{h,\Comp}(X)$ near $Y$ and
          supported over $A\backslash Y$ with
        points in $\opWF_h(A)$ and $\opWF_h(B)$ related by at most one
        transmitted (i.e., not reflected) branching bicharacteristic of length in $(0, T)$, the microlocalized
        reflective propagator is given on $t\in (0,T)$ by the same
        expression as in Lemma~\ref{lem:int-pgtr}.
      \end{proposition}
            \begin{proof}
We must of course truncate the parametrix construction, as we cannot
solve our equations along the long-time flow if the flow is trapped.
Suppose $\abs{x_1}<\ep_1$ on $\pi\WF'A \cup\pi \WF'B.$
Let $\Upsilon(x_1) \in \CcI(X)$ be chosen to be supported in
$\abs{x_1}< 2\ep_1$ and equal to $1$ on $\abs{x_1}<\ep_1$.  Let
$$
V_{\Upsilon B}=\begin{cases} 
V_{\Upsilon B}^{\INC}(t) + V_{\Upsilon B}^{\REF}(t) ,\quad & x_1<0,\\
V_{\Upsilon B}^{\TRAN}(t) ,\quad & x_1>0,
\end{cases}
  $$
denote the parametrix construction above, multiplied on the left by
$\Upsilon (x_1)$, and applied to agree with the short-time interior
propagator from Lemma~\ref{lem:int-pgtr} for small time.  Thus,
$$
V_{\Upsilon B}^\REF=
(2\pi  h)^{-\frac{n}{2}+\kk} e^{\sfrac{i}{h} S^\REF_{\gamma}(t,x,y) }
          \abs{\Delta_{\gamma}}^{\frac{1}{2}}  \Upsilon(x_1)  a^\REF(t,x,y) b(y,-\pib\partial_y
          S^\REF_{\gamma}) \, \abs{dx\, dy}^{\frac{1}{2}}
          (1+O(h)),
$$
$$
V_{\Upsilon B}^\INC=
(2\pi  h)^{-\frac{n}{2}} e^{\sfrac{i}{h} S^\INC_{\gamma}(t,x,y) }
          \abs{\Delta_{\gamma}}^{\frac{1}{2}}  \Upsilon(x_1) b(y,-\pib\partial_y
          S^\INC_{\gamma}) \, \abs{dx\, dy}^{\frac{1}{2}}
          (1+O(h)),
$$
$$
V_{\Upsilon B}^\TRAN=
(2\pi  h)^{-\frac{n}{2}} e^{\sfrac{i}{h} S^\TRAN_{\gamma}(t,x,y) }
          \abs{\Delta_{\gamma}}^{\frac{1}{2}}  \Upsilon(x_1) b(y,-\pib\partial_y
          S^\TRAN_{\gamma}) \, \abs{dx\, dy}^{\frac{1}{2}}
          (1+O(h))
          $$
          (with $O(h)$ terms here and below denoting terms with a full
          asymptotic expansion in integer powers of $h$).  Here we
          have used the fact that we know the solutions to the
          transport equations away from $Y$ are given by the standard
          formula for the amplitude from Lemma~\ref{lem:int-pgtr}.
          
          Applying the Schr\"odinger operator to the parametrix
          yields, by the parametrix construction
          $$
(hD_t+P_h) V_{\Upsilon B} =[P_h, \Upsilon]V_B+O(h^\infty),
$$
where $V_B$ denotes the parametrix without the factor of $\Upsilon$.
Now note that $\WF_h ([P_h, \Upsilon]V_B)$ lies over $\supp \nabla
\Upsilon$ and, by the the form of the phases \eqref{incref},
\eqref{inctran}, \eqref{pa1phi},
lies in the ``outward'' direction $\xi_1 x_1>0$; note that there is no
contribution from the incident phase (which would otherwise be incoming) since $\nabla
\Upsilon=0$ on $\pi \WF' A$.  By the propagation
of singularities results of \cite{GaWu:23a}\footnote{This is overkill,
  as the singularities in question are oriented away from $Y$, so
  ordinary propagation of singularities together with a localization
  argument suffice.}(as revisited in the time
dependent setting in \cite{GaWu:21}),
  $\WF U(t) [P_h, \Upsilon]V_B$
remains disjoint from $\WF'A \cup \WF' B$ for $t\in [0, T]$ (if $T$ is
sufficiently small).
In particular, for small $\varpi>0$ and $t \in [\varpi,T]$ we note that 
\begin{equation}\label{outgoingremainder}
\int_{\varpi}^t A U(t-s) [P_h, \Upsilon]V_B \, ds=O(h^\infty).
\end{equation}
Note also that $V_{\Upsilon B}(t)$ differs from $U(t) B$ by $O(h^\infty)$ for
$t \in [0, \varpi]$ (where $\varpi>0$ is taken sufficiently small)
by Lemma~\ref{lem:int-pgtr}.

In summary, then,
\begin{align*}
(hD_t +P_h) (V_{\Upsilon B}-U(t)B)&= [P_h,
\Upsilon]V_B+O(h^\infty),\\
(V_{\Upsilon B}-U(t)B)\big\rvert_{t=\varpi} &=O(h^\infty).
\end{align*}
Hence by Duhamel's Principle and unitarity of $U(\bullet)$,
$$
U(t)B -V_{\Upsilon B}(t) = \int_{\varpi}^t  U(t-s) [P_h, \Upsilon]V_B \, ds +O(h^\infty),\quad t \in [0, T].
$$
Applying $A$ to this equation yields the desired result on the
parametrix by \eqref{outgoingremainder}.

The proof that the transmitted propagator agrees with the free one to
top order and admits its own asymptotic expansion follows from the
identification of the transmitted phase function as the ordinary
classical action (\eqref{eq:eik} et seq.); the identification of the
transmitted amplitude with the usual solution to the transport
equations modulo $O(h^\kk)$ follows from\eqref{eq:continuity}.
\end{proof}

\section{Propagators with multiple reflections}
\label{sec:multi}
In this section, we compose the form of the ``free'' propagator (i.e.,
the propagator for smooth potentials of Lemma~\ref{lem:int-pgtr})
with the short-time reflected propagator to
get (long-time) propagators with one reflection, then we iteratively
compose free propagators and reflected propagators to get the
propagators along the branching flow with multiple reflections.

\subsection{Microlocal propagator with one reflection}
\label{sec:one-ref}
We compute a parametrix for a (long-time) reflected propagator by
decomposing it to an interior propagator and a short time reflected
propagator.  This computation can also be seen very easily by
employing standard FIO techniques, but we include it for the sake of
exposition, and to introduce some tools for the
dynamical interpretations of stationary phase expansions.

In this and following sections, as we will frequently be concerned with behavior modulo $O(h^\infty)$, we denote 
 \begin{equation}
 \label{eq:equiv}
f\equiv g \Longleftrightarrow f=g+O(h^\infty).
\end{equation}

Note that for any branching null bicharacteristic $\gamma$ with length
$t_0$, $\gamma(0)= (y_0,\eta_0)$ and $\gamma(t_0)=(x_0,\xi_0)$ in
$T^*(X\backslash Y)$, we can associate it with $\ep>0$
  and two (sufficiently
small) microlocal cutoffs $A_i,A_e\in\Psi_h(X)$ such that
$A_i$ and $A_e$ are elliptic at $(x_0,\xi_0)$ and $(y_0, \eta_0)$
respectively.  In
particular, if $\gamma$ is a branching null bicharacteristic with length
in $(t-\ep,t+\ep)$ and with exactly
one reflection, we can choose
$A_i,A_e\in\Psi_h(X)$ such that any branching null bicharacteristic
starting from $\opWF_h(A_i)$ and ending in $\opWF_h(A_e)$ has exactly
one reflection, if their microsupport is small enough.  (If there
were multiply reflected trajectories arbitrarily close, passing to a
subsequence at which two reflection times coalesce would show that
$\gamma$ would have to
have a glancing point.)
 By
Proposition \ref{prop:ref-pgtr} and the stationary phase lemma, we
obtain the following lemma on the propagator with one reflection.

\begin{lemma}
  \label{lem:ref-pgtr1}
  Suppose $\gamma$ is a branching null bicharacteristic with time
  $t_0$ and exactly one reflection. Then there exist $A_i,A_e\in
  \Psi_{h}(X)$ elliptic at the start- and end-points of $\gamma$ and $\ep>0$ such that for $t\in(t_0-\ep, t_0+\ep)$, (the kernel of) the microlocal reflection propagator $A_eU(t)A_i(z,w)$ is given by
  \begin{equation}\label{onereflectionlong}
    (2\pi i h)^{-\frac{n}{2}} h^\kk e^{\frac{i}{h}S^{\REF}_\gamma}
\abs{\Delta_\gamma}^{\frac{1}{2}} e^{-i\frac{\pi}{2}\mu_\gamma} a_e(z,\partial_z S^\REF_{\gamma})\rc(t,x,y) a_i(w,-\partial_w S^\REF_{\gamma})  \abs{ dz\, dw}^{\frac{1}{2}},
  \end{equation}
  where the reflection coefficient $\rc(t,x,y)$ is evaluated at the
  unique reflected point and $\mu_{\gamma}$ is the Morse index of the reflected physical path $\gamma$.
\end{lemma}
As noted above, the result in fact follows directly from the invariance of the symbol
of the semiclassical Lagrangian distribution $U(t)$ along the Hamilton
flow generated by $p$, with the interesting subtlety being (as usual)
the
inclusion of a Maslov factor as the flow encounters conjugate points.
We include the sketch of a stationary phase proof
here for completeness; the Maslov contribution is the
subject of the authors' previous paper \cite{WuYaZo:24}.
\begin{proof}
 Recall now that we use $\XPhi{E}^\gamma_t$ to denote the time $t$ flow along branching null
    bicharacteristics \emph{near $\gamma$} (and at fixed energy $E=p(\gamma(0))$), which can be made a well defined single-valued
    flow on $\WF' (A_i)$: here since $\gamma$ undergoes reflection, we are
    requiring that the flow be reflected rather than transmitted upon
    hitting $Y$, i.e.\ we insist, in addition to the requirements of
    Definition~\ref{def:bnb}, that the sign of the defining function
    of $Y$ remain constant along the flow near the reflection.

We begin with the case where $\WF' A_i$ is contained in a small
coordinate neighborhood of $Y$:
    we shrink $\WF' A_i$ as needed, and break the propagation time into $t=t_1+t_2$ such that under the
    flow $\XPhi{E}^\gamma_t$, all reflections occur before (but close
    to) time $t_1$, so that $\XPhi{E}^\gamma_{t_1}(\WF' A_i)$ also lies in a coordinate neighborhood
of $A_i$ and the results of Proposition~\ref{prop:ref-pgtr} apply for
the short-time propagator from $\WF A_i$ to $\XPhi{E}^\gamma_{t_1}(\WF' A_i)$

    We now take $B=(B')^2 \in\Psi_{h}(X\backslash Y)$ compactly
    microsupported close to $Y$, such that its principal
    symbol $\sigma_h(B)= 1$ on $\XPhi{E}^\gamma_{-t_2}(\WF'(A_e))$ and
    $\WF'(\id-B)$ is disjoint from the incident flowout of
    $\WF'(A_i)$.  (The notions of incident and reflected
    notions make sense locally on $\WF' B$ since $B$ is microsupported near $Y$).  We
    consider the microlocal propagator $A_e U(t_2)BU(t_1)A_i$. This is
    the composition of a reflected propagator $B'U(t_1)A_i$ and an
    interior propagator $A_e U(t_2)B'$. By Proposition
    \ref{prop:ref-pgtr}, using $\equiv$ denote to equivalence mod $O(h^\infty)$,
    \begin{equation}
    \label{eq:pgtr-1}
        \begin{split}
            A_e U(t)A_i 
            &\equiv A_e U(t_2)BU(t_1)A_i\\
            &\equiv A_e U(t_2)B'(U_{B'A_i}^{\INC}(t_1) + U_{B'A_i}^{\REF}(t_1))\\
            & \equiv A_e U(t_2)B' U_{B'A_i}^{\REF}(t_1),
        \end{split}
    \end{equation}
    where the first equation is due to the identity
    $A_eU(t_2)(\id-B) = \mathcal{O}(h^{\infty})$ 
   by propagation of singularities, and the last equation
    holds since $\WF' B'$ is
    disjoint from the wavefront set of the incident propagator. We now
    use the method of stationary phase to compute the composition of
    the propagators in \eqref{eq:pgtr-1} with that in Lemma
    \ref{lem:int-pgtr}.
The stationary points are points such that
    \[\pa_{z'}S_{\gamma_2}(t_2,z,z')+ \pa_{z'}S^{\REF}_{\gamma_1}(t_1,z',w)=0.\]
In analyzing the result of stationary phase, we use two key
composition identities for van Vleck determinant and Morse index as follows:
    \[e^{-i\frac{\pi}{2}\mu_\gamma} = e^{-i\frac{\pi}{2}\mu_{\gamma_1}} \cdot e^{-i\frac{\pi}{4}n} \cdot e^{i\frac{\pi}{4}\sgn \pa^2_{z'z'}(S_{\gamma_1}+S_{\gamma_2})}, \quad \sgn (\cdot) = n - 2\ind(\cdot) \]
    \[\abs{\Delta_\gamma}^{\frac{1}{2}} = \abs{\Delta_{\gamma_1}}^{\frac{1}{2}} \abs{\Delta_{\gamma_2}}^{\frac{1}{2}} \abs{\det \pa^2_{z'z'}(S^{\REF}_{\gamma_1}+S_{\gamma_2})}^{-\frac{1}{2}}\]
    The first identity is a special case of \cite[Theorem
    5.7]{WuYaZo:24} while the second identity is proved in Lemma~\ref{lemma:vanvleck}.  Note that the symbol of $B'$ is
    identically $1$ at the critical set (in phase space), hence does
    not appear in the composition, so stationary phase shows that (the
    Schwartz kernel of) $A_e
    U(t)A_i$ is given by \eqref{onereflectionlong}.

    We now turn to the general case, where $\WF' A_i$ is not
    necessarily close to $Y$.  This is accomplished by a further
    composition with the smooth propagator from
    Lemma~\ref{lem:int-pgtr}, this time on the right; the stationary
    phase computation is identical to the one performed above.
       \end{proof}

\subsection{Microlocal propagator with multiple reflections}\label{sec:multiplereflection}
In this section, we construct a parametrix for the microlocalized propagator $A_e U_{\gamma}(t) A_i$ associated with a branching null bicharacteristic triple $(\gamma, A_i, A_e)$ under multiple reflections. 

Recall for a branching null bicharacteristic $\gamma\subset T^*X$ with
length $t$, starting and ending over $X \backslash Y$, we may choose
$A_i, A_e\in \Psi_h(X)$ supported away from $Y$ such that $A_i$ and
$A_e$ are elliptic at $\gamma(0)$ resp.\ $\gamma(t)$.  In addition, we can
demand $A_i,A_e\in\Psi_h(X)$ such that their microsupports $\opWF_h
A_i$ and $\opWF_h A_e$ are as small as we want. To construct
a parametrix for the microlocal propagator, we want to insert a microlocal cutoff between
each reflection along the branching null bicharacteristic
$\gamma$ and compose the resulting single-reflection parametrices.

Assume there are $m\in\mathbb{N}$ reflections along $\gamma$ at times
$0<S_1<S_2<\cdots<S_m<t$. Take $T_i\in (S_i,S_{i+1})$ for $1\leq i\leq
m-1$ and $T_0=S_0=0$. Define $t_i := T_i-T_{i-1}$ to be the
propagation time between $(T_{i-1},T_i)$. We first construct $B_1$
near $\gamma(T_1)$; the construction of the rest of the intermediate microlocalizers $B_i$ can be carried out inductively.

Consider two sufficiently small neighborhoods $U_1, V_1$ such that
$$\XPhi{E}^\gamma_{T_1}(\WF' A_i) \subset U_1 \subset V_1.$$  Choose $B_1\in \Psi_h(X)$ such that
$$
\WF' B_1 \subset V_1,\quad \WF' (\id-B_1) \cap U_1=\emptyset.
$$
Then assuming $B_{k-1}$ has
been constructed, $B_k$ can be constructed similarly such that
$$
\WF' B_k \subset V_k,\quad \WF' (\id-B_K) \cap U_k=\emptyset.
$$
where we have chosen $U_k\subset V_k$ with
$$
\XPhi{E}^\gamma_{T_k}(V_{k-1})\subset U_k.
$$

A \emph{microlocal propagator} associated with a branching
null bicharacteristic triple $(\gamma,A_i,A_e)$ is thus defined as
\begin{equation}
\label{eq:ML-pgtr}
A_e U_{\gamma}(t) A_i := A_e U(t-T_m)B_{m}U(t_m)B_{m-1}\cdots B_1 U(t_1) A_i.
\end{equation}
We can further arrange (and will use below) that we may take
  each $B_j$ to be a square of another operator: $B_j=(B_j')^2.$

\begin{remark}
  If there is only one branching null bicharacteristic connecting any
  pair of points in $\WF' A_i$ and $\WF' A_e$, then
  the definition of microlocalized propagators is
  independent of the interim microlocalizers modulo
  $\mathcal{O}(h^{\infty})$; it only depends on the initial and the
  final microlocal cutoffs.  More generally, though, these internal
  cutoffs can separate multiple ways of getting from $\WF A_i$ and
  $\WF A_e$ in time $t$.
\end{remark}

Now we compute the microlocal propagator associated with $(\gamma,A_i,A_e)$ using composition of FIOs. Assume that the length of the branching null bicharacteristic $\gamma$ is $L$.
\begin{proposition}
\label{prop:pgtr-m}
For $t>0$ sufficiently closed to $L$, the Schwartz kernel of the microlocal propagator $A_e U_{\gamma}(t) A_i$ (with $m$ reflections) associated with $(\gamma,A_i,A_e)$ is given by
\begin{equation}
\label{eq:pgtr-m}
  (2\pi i h)^{-\frac{n}{2}} h^{m\kk} e^{\frac{i}{h}S^{\REF}_\gamma}
\abs{\Delta_\gamma}^{\frac{1}{2}} e^{-i\frac{\pi}{2}\mu_\gamma}
a_e(z,\partial_z S^\REF_{\gamma})\rtot(t,z,w)
a_i(w,-\partial_w S^\REF_{\gamma}) \abs{dz\, dw}^{\frac{1}{2}}(1+O(h)). 
  \end{equation}
where
\begin{gather}
    \label{eq:ref-coeff-m}
     \rtot(t,z,w) \equiv \prod_{j=1}^m \tfrac{i^\kk  J_j}{(2\xi_N^j)^{\kk+2}} 
\end{gather}
is the product of all the individual reflection coefficients $\rc(t,x,y)$ of $\gamma$ given in \eqref{eq:ref-coeff}.
\end{proposition}

\begin{proof}
    We work inductively. For $\gamma$ with one reflection, this result
    is just Lemma \ref{lem:ref-pgtr1}. Assume equation
    \eqref{eq:ref-coeff-m} holds for any branching null
    bicharacteristic with at most $m-1$ reflections. We can break a
    branching null bicharacteristic $\gamma$ with $m$ reflections into
    two pieces: a piece with a single reflection and another piece
    with $m-1$ reflections, then compute their composition using the
    stationary phase lemma.  
    Note that the microlocal propagator $A_e U_{\gamma}(t) A_i$ can be written as 
    $$ (A_e U(t-T_{m})B'_{m}) ( B'_{m} U(t_{m}) \cdots B_1 U(t_1) A_i),$$
    where $B'^2_{m-1} = B_{m-1}.$ The above two propagators are given by 
    \begin{gather*}
      {(2\pi i h)^{-\frac{n}{2}}} e^{\frac{i}{h}S^{\REF}_{\gamma_1}}
      \abs{\Delta_{\gamma_1}}^{\frac{1}{2}}
      e^{-i\frac{\pi}{2}\mu_{\gamma_1}} b'_{m}(z',\partial_{z'}
      S^\REF_{\gamma_1})\rtot (T_{m},z',w) a_i(w,-\partial_w S^\REF_{\gamma_1}) h^{k(m-1)} \abs{ dz'\, dw}^{\frac{1}{2}}, \\
      {(2\pi i h)^{-\frac{n}{2}}} e^{\frac{i}{h}S^{\REF}_{\gamma_2}}
      \abs{\Delta_{\gamma_2}}^{\frac{1}{2}}
      e^{-i\frac{\pi}{2}\mu_{\gamma_2}} a_e(z,\partial_z
      S^\REF_{\gamma})\rtot(t-T_{m},z,z') b'_{m}(z',-\partial_{z'}
      S^\REF_{\gamma_2}) h^{k} \abs{ dz\, dz'}^{\frac{1}{2}}.
    \end{gather*}
Critical points in $z'$ of the phase of the composition are given by 
$$\pa_{z'}\phi = \partial_{z'}S^{\REF}_{\gamma_1}(t_1,z',w) + \partial_{z'} S^{\REF}_{\gamma_2}(t_2,z, z') =0, $$
and by construction of $B'_{m}$ we have $\sigma_h(B'_{m}) = 1$ at critical points. Therefore, the microlocal propagator $A_e U_{\gamma}(t) A_i$ is given by
\begin{equation}
    {(2\pi i h)^{-\frac{n}{2}}} h^{m\kk}
    e^{\frac{i}{h}S^{\REF}_\gamma} \abs{\Delta_\gamma}^{\frac{1}{2}} e^{-i\frac{\pi}{2}\mu_\gamma} a_e(z,\partial_z S^\REF_{\gamma})\rtot(t,z,w) a_i(w,-\partial_w S^\REF_{\gamma}) \abs{ dz\, dw}^{\frac{1}{2}},
\end{equation}
where again we used Lemma~\ref{lemma:vanvleck} to obtain the $\Delta_\gamma$ term, as well as
the result on composition of
Morse indices from Theorem~5.7 of \cite{WuYaZo:24} to identify the new
Maslov index $\mu_\gamma$ as the number of conjugate points (in the
sense of broken trajectories introduced in \cite{WuYaZo:24}
encountered between $w$ and $z$ along $\gamma$.
The reflection coefficient factors compose by definition:
\begin{equation*}
    \rtot(t,z,w) = \rtot(t-T_{m},z,z') \rtot(T_{m},z',w)=
 \prod_{j=1}^m \tfrac{i^\kk J_j}{(2\xi_N^j)^{k+2}} \qedhere
\end{equation*}
\end{proof}

In the next lemma we show that if  for each
$(w,\xi)\in\opWF_h A_i$ and $(z,\eta)\in\opWF_h A_e$, there exist at
most one branching null bicharacteristic connecting them in time
$t$, then the propagator $U(t)$ is equivalent to
microlocal propagators $U_{\gamma}(t)$. We assume that $\gamma$ is one of these bicharacteristics
  (with length $L$)
  and that the $B_j$ are constructed according to the algorithm given above.

\begin{lemma}
\label{lemma:pgtr_equiv}
Near $t=L$, 
\begin{equation}
    \label{eq:pgtr_equiv}
    A_e U(t) A_{i} \equiv  A_e U_{\gamma}(t) A_{i} \mod \mathcal{O}(h^{\infty})
\end{equation}
\end{lemma}
\begin{proof}
The proof is by induction on the number of
reflections. If $A_e U(t) A_{i}$ involves at most one reflection along
all possible branching null bicharacteristic, this is established in
Lemma \ref{lem:ref-pgtr1}:
$$A_e U(t_2) (\id-B_1) U(t_1) A_{i} = \mathcal{O}(h^{\infty}).$$

Assume we have showed \eqref{eq:pgtr_equiv} for at most $m-1$
reflections. For $m$ reflections, note that $\WF(\id-B_{m})\cap
\XPhi{E}^\gamma_{T_{m}}(\WF'A_i)$, hence by propagation of singularities,
$$A_e U(t-T_{m}) (\id-B_{m}) U_{\gamma}(T_{m}) A_{i} = \mathcal{O}(h^{\infty}).$$
Thus (employing the inductive hypotheses on the penultimate line),
\begin{equation*}
    \begin{split}
        A_e U_\gamma(t) A_i &= A_e U(t-T_{m}) B_{m} U_{\gamma}(T_{m})
A_{i}\\
&\equiv
A_e U(t-T_{m})  U_{\gamma}(T_{m}) A_{i}\\
&\equiv
A_e U(t-T_{m})  U(T_{m}) A_{i}\\
&=
A_e U(t) A_{i}.\qedhere
    \end{split}
\end{equation*}
\end{proof}

\section{Microlocal partitions and Poisson relations}
\label{sec:poisson}

In this section we introduce a simple microlocal partition of unity
that decomposes the energy-localized trace.  Using this decomposition
and propagation of singularities, we prove the Poisson relation,
Theorem 
\ref{thm:poisson}.  

Fix $\chi\in\CcI(\RR)$.  Let $W\in \Psibh(X)$ be a compactly
microsupported operator with 
\begin{equation}\label{cond-W}
\WF'(I-W) \cap K_{\chi} =\emptyset, \quad K_{\chi}\coloneqq  \bigcup_{E \in \supp \chi} \dot\Sigma^E_b.
\end{equation}
Recall that over $Y$, the set $K_{\chi}$ lives inside \{$\xi_1=0\}$, as it is contained in the compressed characteristic set. 
Note also that in obtaining compact microsupport of $W$ we are using
compactness of the energy surfaces, guaranteed by our hypotheses on
$P_h$ near infinity in case $X$ is not compact (which required $V\to
+\infty$ at infinity).  Then we have
\begin{equation}
\label{eq:W}
(I-W) \chi(P_h)=O(h^\infty) \implies \chi(P_h)U(t)\equiv W \chi(P_h) U(t),
  \end{equation}
where $\equiv$ denotes equivalence modulo $O(h^\infty)$ as defined in \eqref{eq:equiv}.
Now we take $\{A_j\} \subset \Psibh(X)$ compactly microsupported, such that
\begin{equation}
\label{eq:partition-W}
    \sum A_j^2W-W\equiv 0.
\end{equation}
Such a microlocal partition of unity over a compact set in phase
space is easily constructed via an iterative procedure: we can indeed
ask that
$$
\WF' (\id-\sum A_j^2) \cap \WF' W=\emptyset.
$$
By equation \eqref{eq:W}, we have the microlocal decomposition of the
propagator:
$$
\chi(P_h)U(t)\equiv \sum A_j^2 \chi(P_h)U(t).
$$
Hence, taking the trace and using the cyclicity of the trace, we have established the following result on microlocal partitions of the spectrally localized trace. 
\begin{proposition}\label{prop:cyclicity}
   Let $A_j \in \Psibh(X)$
be a partition of unity comprised of semiclassical b-operators, as
defined above. Then we have
\begin{equation}\label{PoI1} \Tr [\chi(P_h) U(t)] \equiv \sum_j \Tr
  [A_j^2 \chi(P_h) U(t)]=\sum_j \Tr [A_j \chi(P_h) U(t) A_j]
\end{equation}
\end{proposition}

\begin{proof}[Proof of Theorem \ref{thm:poisson}]
Using Proposition \ref{prop:cyclicity}, we prove for $T\notin \lspec^N_{\supp\chi}$
$$\Tr [\chi(P_h) e^{-itP_h/h}] = O(h^{(N+1)\kk-n-0})$$ in a neighborhood
of $T$.  The other statement of the theorem (involving $\lspec_{\supp
  \chi}$ with no restriction on the number of reflections) then
follows from this one, since certainly if $T\notin \lspec_{\supp\chi}$
then $T\notin \lspec^N_{\supp\chi}$ for every $N$.

Fixing $T\notin \lspec_{\supp\chi}$, we construct a semiclassical
microlocal partition $\{A_k\} \subset \Psibh(X)$ as in \eqref{eq:partition-W}. Hence Proposition~\ref{prop:cyclicity} applies.
By
Lemma~\ref{lemma:approximateflow}, taking the partition sufficiently
fine by shrinking $\WF'A_k$, we can arrange that for an open interval $I \ni T$, there are no
$N$-fold branching bicharacteristics starting and ending in $\WF'A_k$ for
any $k$, by Lemma~\ref{lemma:approximateflow}.

  Consequently, Corollary~\ref{cor:propofsings} shows
that
$$
A_k \chi(P_h) U(t) A_k=O_{L^2\to L^2}(h^{(N+1)\kk-0}).
$$
By compactness of microsupport, we may factor out a compactly microsupported
elliptic b-pseudodifferential operator $W'$ to obtain
$$
A_k \chi(P_h) U(t) A_k\equiv A_k \chi(P_h) U(t) A_kW',
$$
and we easily compute $\Tr W'=O(h^{-n})$ by integrating the Schwartz
kernel over the diagonal (see Theorem~C.18 of \cite{Zw:12}),
so that
\begin{equation*}
    \Tr [A_k \chi(P_h) U(t) A_k]=O(h^{(N+1)\kk-n-0}).\qedhere
\end{equation*}
\end{proof}

\begin{remark}\label{remark:mismatch}
The reader might object that we end up requiring $N$ to be large
before we know that in our Poisson relation the singularities in the
trace from closed orbits with $N$
reflections are in fact smaller than the ``main'' unreflected
Gutzwiller contributions as given e.g.\ in \cite[Theorem 3]{Me:92} or
in our Theorem~\ref{thm:trace} with $N=0$.
This is inevitable, however, as we employ no \emph{structural}
information about the propagator in the Poisson relation---merely its
crude mapping properties---while our
(and others') computations of trace asymptotics for nondegenerate
closed orbits use
strongly the semiclassical Lagrangian
structure of the propagator.
  \end{remark}

\section{The trace formula}
\label{sec:tracenew}
Recall that the trace we seek to compute is given (up to a standard $2
\pi$ normalizing factor) by
\begin{equation}
    \label{FTtrace}
    I_E \coloneqq \int_{\mathbb{R}} e^{iEt/h}\hat\rho(t) \Tr [\chi(P_h) e^{-itP_h/h}] dt 
\end{equation}
In proving the trace formula, it suffices to work with an arbitrarily small neighborhood of a fixed energy $E_0$; we will therefore use the freedom to shrink $\supp \chi$ in order to constrain the dynamics to be close to the dynamics at energy $E_0$.

We use a more refined microlocal partition of unity of $K_\chi$
(actually $\WF' W$), to decompose the trace into finitely many
microlocal traces. The microlocal traces away from $Y$ can be
obtained using the standard stationary phase method. For the microlocal
traces near $Y$, we use cyclicity of the trace to push the computation
of the trace away from the interface $Y$, so as to avoid stationary
phase computations on manifolds with boundary (which would be
necessary on the two sides of $Y$).  Similar techniques have been
previous employed in the setting of conic singularities in, e.g.,
\cite{Hi:05,Wu:02,FoWu:17,Ya:22}. The dynamics are, however, considerably more
complicated in the setting considered here.

\subsection{A refined microlocal partition of unity}
\label{sec:partitionnew}

We again build a microlocal partition of $W$ as in
Section~\ref{sec:poisson} but now require the partition to have more
sophisticated dynamical properties.

 Let $\indexset$
denote the (finite) index set for the partition; we will later split $\indexset$ into
two subsets denoted\footnote{Notwithstanding the notation, we remind
  the reader that the interface $Y$ is an interior hypersurface, not a
  boundary.} 
  $$\indexset =\indexY \sqcup \indexint,$$ 
  according to
whether or not the element of the partition overlaps $Y$. Let
$\{A_k\}_{k\in\indexset}\subset\Psi_{b,h}(X)$ be a finite collection
of semiclassical b-pseudodifferential operators on $X$ satisfying the
following basic properties.  In what follows we use the notation ``$A
\equiv B$ on $K$'' as a more human-readable shorthand for $\WF' (A-B)\cap
K=\emptyset$, i.e. as meaning that the two operators are microlocally
equivalent on $K$.

First, we fix a nonnegative cutoff function $\psi_Y$ 
supported close to $Y$ and having smooth square root, such that
$\psi_Y=1$ on a neighborhood of $Y$.  The partition we take has the
following properties:
\begin{enumerate}\renewcommand{\theenumi}{\alph{enumi}}
    \item Each $\WF' A_k$ is compact and $A_k$ has
      real principal symbol.
    \item There exists a fixed small constant $\delta$ such that $\opWF A_k\subset \prescript{b}{}{T}^*X$ is contained in a small ball of radius $\delta$ with respect to a metric on $\prescript{b}{}{{T}^*X}$;
    \item $\sum_{j \in \indexset} A_j^* A_j \equiv I$ on $\WF' W$.
      \item $\sum_{j \in \indexY} A_j^* A_j\equiv \psi_Y$ on
        $\WF' W$ and $\sum_{j \in \indexint} A_j^* A_j\equiv (1-\psi_Y)$ on $\WF' W$
\end{enumerate}
Note that if $k \in \indexint$, then $A_k$ is a usual semiclassical $\Psi$DO. 
Such microlocal partition of unity can be constructed in the following
way: we first construct a microlocal partition of unity $C_j \in
\Psibh(X)$ over $\WF' W$, so that $C_j$ are self-adjoint and
$$
\sum C_j^2\equiv I \text{ on } \WF'W.
$$
This is accomplished by a standard iterative process in the symbol
calculus.  Then we set
$$
B_j =C_j \sqrt{\psi_Y},\ B_j'=C_j \sqrt{1-\psi_Y}
$$
so that
$$
\sum B_j^* B_j +\sum (B_j')^* B_j'\equiv I \text{ on } \WF'W.
$$
Finally, let the partition $\{A_j\}_{j\in\indexset}$ consist of the $B_j$'s for $j \in \indexY$ and $(B_j')$'s for $j \in \indexint$.

As we are interested in the trace near some energy $E_0$ of a specific
closed branching orbit (with $N$ reflections) as in Theorem \ref{thm:trace}, we fix a nondegenerate closed orbit cylinder $\gamma$ near the 
energy $E_0\in I$. By the dynamical
assumption \eqref{contact}, shrinking $I$
if necessary as discussed
  in the beginning of this section, there exists a time $\shorttime>0$ such that for each $E\in I$ if
$\gamma = \gamma_E \subset \dot\Sigma^E_b(p)$ is a closed branching orbit and
$\pi (\gamma(0)) \in Y$ then $\pi(\gamma(\shorttime))\notin
Y$. Note that owing to the assumption \eqref{contact}, we can take $\shorttime$ as small as desired. Now we use the freedom to shrink the $\delta$ in the definition of the
partition as well as the size of $\psi_Y$ and $I=\supp \chi \ni E_0$  so that the partition of
unity additionally enjoys the following dynamical properties.  Fix
any $k \in \NN$ (which will be taken large later on).
    \begin{enumerate}\renewcommand{\theenumi}{\Alph{enumi}}
    \item \label{p:iii} For all $j \in \indexY$ satisfying
        $\gamma \cap \WF'A_j \neq \emptyset$,  there exist
        $\shorttime$ such that $\pi\big(\XPhi{E}^{N+k}_{\shorttime} (\WF'
        A_j)\big) \cap \supp \psi_Y=\emptyset$ for all $E \in \supp \chi$.
\item \label{p:ii} For all $j \in \indexint$ and $E
      \in \supp \chi$, we have either $\pi(\XPhi{E}^{N+k}_{\shorttime} (\WF'
      A_k)) \subset \{\psi_Y=1\}$ or $\pi\big(\XPhi{E}^{N+k}_{\shorttime} (\WF'A_j)
      \big)\cap Y =\emptyset$.
        \end{enumerate}
      
Property \eqref{p:iii} says that we can refine the
          boundary microlocal cutoffs such that, if they are close to
          the orbit cylinder $\gamma$, their $(N+k)$-branching flow
          with energy in $\supp\chi$ stay away from the boundary in
          some time $\shorttime>0$. Note that for each of the finitely
          many points over $Y$ in a closed branching orbits
          $\gamma_E$ in the orbit cylinder $\gamma$, our dynamical
          hypothesis \eqref{contact} shows that there is a time
          $\Hat{T}_E$ such that the flow starting from
          these points (which undergoes a single branching at time $0$) is disjoint from $\pi^{-1}(Y)$ for all
          $t\in (0,\Hat{T}_E)$.  Now we take
          $\Hat{T}\coloneqq \min_{E\in I} \Hat{T}_E >0$. Thus,
          Property \eqref{p:iii} simply follows from shrinking the
          support of $\psi_Y$ and $\chi$, taking
          $\shorttime<\frac{1}{2}\Hat{T}$ small and using continuity
          of the flow (Corollary~\ref{cor:continuity2}) to construct
          boundary microlocal cutoffs $A_j$'s.

Property \eqref{p:ii} says that the partition of unity is sufficiently
fine that if the time-$\shorttime$ flow of the (micro)support of one
of the microlocal cutoffs touches the interface $Y$, it lives in a
small 
enough neighborhood of the interface to be in the set where 
$\psi_Y=1$. This relies only on the continuity of the branching flow,
    and it can be achieved by shrinking the energy window $I$ and refining the interior part of the
    partition sufficiently after we have fixed the $\psi_Y$ and  $A_j$, $j
    \in \indexY$.

  \begin{figure}[h]
    \centering
    \includegraphics[width=0.44\linewidth]{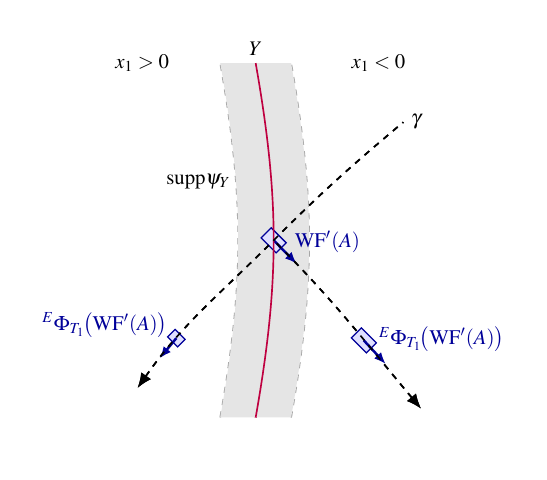}
    \quad 
    \includegraphics[width=0.50\linewidth]{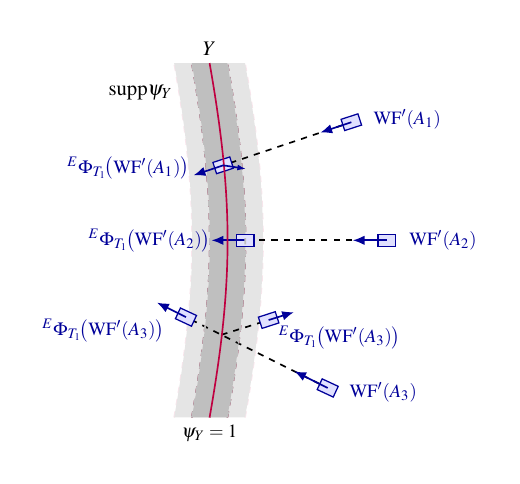}
    \caption{The picture on the left illustrates the microlocal partition property \eqref{p:iii}, where the 		branching flow with energy $E\in\supp\chi$ of the boundary microlocal cutoffs close to
          the orbit cylinder $\gamma$ leave the vicinity of the boundary in
          $\shorttime>0$. The picture on the right illustrates the
          microlocal partition property \eqref{p:ii}, where the time
          $\shorttime$ flow of the microsupport of interior
          microlocalizers either in $\{\psi_Y=1\}$ (e.g.~$A_1$), or
          stay away from $Y$ (e.g.~$A_2,A_3$).} 
    \label{fig:partition}
\end{figure}

\begin{remark}
  Note that the parameter $k$ is inessential here: we will simply
  end by taking this finite number sufficiently large that the
  $O(h^{k_0(N+k+1)-n-0})$ error terms that appear in the trace
  computation below are guaranteed to be smaller than the main term.
  A more detailed analysis would allow hypotheses in our theorem on the dynamics in
  terms of just the $\XPhi{E}^{N+k}_t$-flow for sufficiently large $k$
  rather than on the full $\XPhi{E}_t$ flow: there might be more
  closed orbits for very large $k$ and these would not interfere
  with the leading-order trace asymptotics.
\end{remark}

\subsection{Reduction of the trace}
\label{sec:reductionnew}
We use the microlocal partition of unity introduced in Section \ref{sec:partitionnew} to reduce the trace to pieces which we can compute individually. Note that 
\[\label{trsplit}
\begin{split}
    \Tr [ \chi(P) U(t) ] 
    \equiv \Tr [ W \chi(P) U(t) ]
    &\equiv \sum_{j\in\indexint} \Tr [ A_j^* A_j \chi(P) U(t) ] + \Tr [ \psi_Y W \chi(P) U(t) ] \\
&:= \Tr_\circ + \Tr_\pa.
  \end{split}
\]
This gives a microlocal partition of the trace. We will treat these
terms in separate sections below.  In Section~\ref{sec:int-trace} we
treat $\Tr_\circ$: we
use cyclicity of the trace to
write these terms as
$$
\Tr  A_j \chi(P) U(t)A_j^*
$$
and we will then employ stationary
phase directly.

For the boundary
cutoff contributions $\Tr_\pa$, we begin by using the property that the
$(N+k)$-branching flow (microlocally near the orbit cylinder $\gamma$) from $Y$ does
not come back to $Y$ in time $\shorttime$: we rewrite
(dropping the factor of $W$, which makes no difference)
$$\begin{aligned}
\Tr_\pa
&\equiv \Tr [ \psi_Y \chi_1(P) U(T_1) ( \psi_Y
+\Sigma_{j\in\indexint} A_j^* A_j )\chi(P) U(t-T_1) ]\\
&\equiv \Tr [\chi(P) U(t-T_1)  \psi_Y \chi_1(P) U(T_1) ( \psi_Y +\Sigma_{j\in\indexint} A_j^* A_j )]\\
&\equiv \Tr [\chi(P) U(t-T_1)  \psi_Y \chi_1(P) U(T_1) \psi_Y] + \sum_{j\in\indexint} \Tr [\chi(P) U(t-T_1) \psi_Y \chi_1(P) U(T_1)  A_j^* A_j],
\end{aligned}
$$
where we take the cutoff $\chi_1$ such that $\chi_1 \chi=\chi$ and
$\supp \chi_1\subset I$.

The first term on the RHS is $O(h^{k_0(N+k+1)-n-0})$: this follows
by the Partition Property
\eqref{p:iii}, as there is no $(N+k)$-branching flow starting in $\WF'
A_j^* = \WF' A_j$ and staying in $\supp \psi_Y$ at time $T_1$; we
then apply the same trace estimate as in the proof of
Theorem~\ref{thm:poisson}.
Thus,
\begin{equation}\label{trdelta}
\Tr_\pa =\sum_{j \in \indexint}\Tr [A_j\chi(P) U(t-T_1) \psi_Y \chi_1(P) U(T_1) A_j^* ]+O(h^{k_0(N+k+1)-n-0}).
\end{equation}
Note that by the Partition Property \eqref{p:ii}, the interior
partitions are refined such that their flowouts are either in
$\{\psi_Y\equiv 1\}$ or stay away from $Y$. The leading order contributions to the
trace from these terms will be evaluated in Section \ref{sec:traceatY}.

\subsection{Microlocalized trace in the interiors}
\label{sec:int-trace}
We assume that the function $\hat{\rho}(t)\in\CI_c(\RR)$ is supported
near $T(E_0)$, $E_0 \in \supp \chi$  where $T(E)$ is the period of a
unique periodic branching orbit (with $N$-fold reflection) in
$\dot \Sigma^E_b$. (If there is more than one such orbit cylinder, the argument applies below separately at
each.)

Fix any $E \in \supp \chi$.  By our hypotheses, there is a unique
$N$-fold reflected periodic branching orbit with period $T(E) \in
\supp \hat\rho$.  Choose
$(x^E_0,\xi^E_0) \in \dot \Sigma^E_b$ along this orbit, and suppose $\gamma^E(t) = (X(t,x^E_0,\xi^E_0),\Xi(t,x^E_0,\xi^E_0))\in {}^bT^*X$ is
the orbit, with 
$\gamma^E(0)=(x^E_0,\xi^E_0)$. We simply write $(X(t),\Xi(t))$ when there is
no ambiguity, and moreover omit all the $E$ superscripts in what follows for brevity
of notation, simply remembering that all the quantities have a
smooth parametric dependence on $E$.

We now introduce a general proposition that systematizes the stationary
phase computation of the energy-localized
trace for propagators of the form considered here. For brevity, in
  what follows we only keep track of the leading order in the
  semiclassical parameter $h$ in the trace computation.  Hence we
  employ the notation $\doteq$ to denote ``modulo higher order terms
  in $h$.''  The stationary phase expansions developed below in fact all have
  one-step asymptotic expansions in powers of $h$, hence so do the
  traces we are computing.
Recall that we assume no points along the periodic orbit are
  self-conjugate, so the action $S(t,x,y)$ is locally well defined for $x,y$
  close to the orbit, and is a generating function for the
  locally-defined flow $\XPhi{E}_t^\gamma$ (away from $Y$).  We use
  the notation of Section~\ref{sec:poincare} for the Poincar\'e map
  and the Fermi normal coordinates $\hat x$ to
  the orbit.
\begin{proposition} 
\label{prop:stationary}
With the notation above, consider the operator $K(t)$ with Schwartz kernel
\[K(t,x,y) = e^{\frac{i}{h}S(t,x,y)}\abs{\det\partial_{yx}^2S}^{\frac{1}{2}}a(t,x,y),\]
where $a(t,x,y)$ is supported near $x,y=x_0$, $t=T$. Then for $\hat\rho\in C_c^\infty(\mathbb{R})$ with $\hat\rho$ supported near $t=T$,
\[I_K \coloneqq \int_{\mathbb{R}} e^{iEt/h}\hat\rho(t) \Tr K(t)\, dt \doteq \frac{(2\pi h)^{\frac{n}{2}}e^{i\frac{\pi}{4}\sgn(\mathrm{Hess}_{S(t,\hat x)})}e^{\frac{i}{h}(ET+S_\gamma)}}{\abs{\det(I-\Pcal)}^{\frac{1}{2}}}\int a(T,X(s),X(s))\,ds\]
\end{proposition}
\begin{proof}
We take local coordinates, as described after Lemma \ref{lem:ABC}. Then
\[\int_{\mathbb{R}} e^{iEt/h}\hat\rho(t) \Tr K(t) \,dt =  \int\left(\int e^{\frac{i}{h}(Et+S(t,x,x))} \hat\rho(t) \abs{\det\partial_{yx}^2S}^{\frac{1}{2}} a(t,x,x)\,dt\,d\hat{x}\right)\,dx_1.\]
We now apply the stationary phase lemma to the inside integral. For any fixed $x_1$, the phase function is given by
\[\phi(t,\hat x) = Et+S(t,x,x)\]
where $x = (x_1,\hat x)$ with $x_1$ being a parameter. The critical points are given by
\[\partial_t\phi = E+\partial_tS =0,\quad \partial_{\hat{x}}\phi = \partial_{\hat{x}}[S(t,x,x)] =0.\]
These quantities are zero precisely if there exists a periodic
trajectory of time $t$ with energy $E$ (the latter since $S$ satisfies
the Hamilton--Jacobi equation). As $\gamma$ is the only bicharacteristic with energy $E$, for fixed $x_1$, this happens only at $t=T(E)$, $\hat{x}=0$. The determinant of the Hessian of the phase is 
\begin{align*}
\det\mathrm{Hess}_{\phi(t,\hat x)}|_{\substack{t=T\\ \hat x=0}} = \det\mathrm{Hess}_{S(t,\hat x)}\rvert_{\substack{t=T\\ \hat x=0}} = -(\dot X^1)^2\cdot\det\partial_{yx}^2S \cdot \det(I-\Pcal).
\end{align*}
by the results of Section~\ref{sec:poincare}. It follows that
\begin{align*}
&\int e^{\frac{i}{h}(Et+S(t,x,x))} \hat\rho(t) \abs{\det\partial_{yx}^2S}^{\frac{1}{2}} a(t,x,x)\,dt\,d\hat{x} \\
&\doteq  {(2\pi h)^{\frac{n}{2}} e^{i\frac{\pi}{4}\sgn(\mathrm{Hess}_{S(t,\hat x)})}e^{\frac{i}{h}(ET+S_\gamma)}}{\abs{ \det\mathrm{Hess}_{S(t,\hat x)}\rvert_{\substack{t=T\\ \hat x=0}}}^{-\frac{1}{2}}}\abs{\det\partial_{yx}^2S}^{\frac{1}{2}} a(T,x,x)|_{\hat{x}=0}\\
& = \frac{(2\pi h)^{\frac{n}{2}} e^{i\frac{\pi}{4}\sgn(\mathrm{Hess}_{S(t,\hat x)})}e^{\frac{i}{h}(ET+S_\gamma)}}{|\dot X^1|\left|\det(I-\Pcal)\right|^{\frac{1}{2}}} a(T,x,x)|_{\hat{x}=0}
\end{align*}
where $S_{\gamma} = S(T,x_1,0,x_1,0)$ is the action along $\gamma$ near the point $x_0$. Thus the leading order term in stationary phase gives
\[ I_K \doteq \frac{(2\pi h)^{\frac{n}{2}} e^{i\frac{\pi}{4}\sgn(\mathrm{Hess}_{S(t,\hat x)})}e^{\frac{i}{h}(ET+S_\gamma)}}{\left|\det(I-\Pcal)\right|^{\frac{1}{2}}}\int {|\dot X^1|^{-1}}a(T,(x_1,0),(x_1,0))\,dx_1.\]
Changing variables $x_1 = X^1(s,x_0,\xi_0)$, $dx_1 = |\dot X^1|\,ds$,
\[ I_K \doteq \frac{(2\pi h)^{\frac{n}{2}} e^{i\frac{\pi}{4}\sgn(\mathrm{Hess}_{S(t,\hat x)})}e^{\frac{i}{h}(ET+S_\gamma)}}{\left|\det(I-\Pcal)\right|^{\frac{1}{2}}}\int_0^T a(T,X(s,x_0,\xi_0),X(s,x_0,\xi_0))\,ds.\]
This proves the proposition.
\end{proof}

Now we are ready to consider the contribution of the term
  $\Tr_\circ$ to the trace formula.
First, note that if
  $$
  \XPhi{E}^{N+k}_t(\WF'A_j) \cap \WF'A_j=\emptyset,\ E \in I,
  $$
  then
  $$
\Tr  A_j \chi(P) U(t)A_j^*=O(h^{k_0(N+k+1)-n-0}),
$$
by the proof of the Poisson relation above.  Thus, we consider only the
elements in the partition where the flow is recurrent.  At such points
we will use our parametrix for the propagator.

Consider in particular an operator $A = A_j$, $j \in \indexint$ such
that $  \XPhi{E}_t(\WF'A_j) \cap \WF'A_j\neq \emptyset$; by
hypothesis, for pairs of points in $\WF A_j$
there is a unique branching bicharacteristic of length $t$ and energy
$E$ connecting them.
 By Proposition
\ref{prop:pgtr-m} and Lemma \ref{lemma:pgtr_equiv}, the leading term
of the Schwartz kernel of the resulting microlocal propagator $A 
e^{-itP_h/h} A^*$ is given by\footnote{We abuse notation throughout
  this section by conflating operators with their Schwartz kernels.}  
\begin{equation}
         {(2\pi i h)^{-\frac{n}{2}}} e^{\frac{i}{h}S^{\REF}_\gamma(t,z,w)}
\abs{\Delta_\gamma}^{\frac{1}{2}} e^{-i\frac{\pi}{2}\mu_\gamma} \rtot(t,z,w)  h^{N\kk}  a(z,\partial_z S^\REF_{\gamma}) a(w,-\partial_w S^\REF_{\gamma}) \abs{dz\, dw}^{\frac{1}{2}},
\end{equation}
where $\gamma = \gamma_{t,z,w}$ is the (locally) unique branching
  bicharacteristic with at most $N$ reflections close to the given
  orbit cylinder $\gamma$ such that
$\gamma(0),\gamma(t)\in \opWF A$ with time $t$ and $\pi(\gamma(0))=w$,
$\pi(\gamma(t))=z$.  (Here we are abusing notation since the given
orbit cylinder is $\gamma_{t,z,w}$ for special values of $(t,z,w)$
where $z$ and $w$ are time-$t$ apart on one of the closed orbits;
the
more general $\gamma_{t,z,w}$ extends this family of orbits to nearby
nonperiodic ones.)
Note that by Lemma~\ref{lemma:timeorspace}, we may turn the
  spectral cutoff $\chi(P_h)$ into a (semiclassical) Fourier
  multiplier in time, hence
  $$\begin{aligned}
  \int_{\mathbb{R}}  e^{iEt/h}  \hat\rho(t) \Tr [A \chi(P_h)
  e^{-itP_h/h} A^*] dt &\equiv  \int_{\mathbb{R}}  e^{iEt/h}
  \hat\rho(t) \Tr [A \chi(-hD_t) e^{-itP_h/h} A^*] dt \\
  &=\chi(E) \int_{\mathbb{R}}  e^{iEt/h}
  \hat\rho(t) \Tr [A e^{-itP_h/h} A^*] dt .
\end{aligned}
  $$

By Proposition~\ref{prop:pgtr-m} and Proposition
\ref{prop:stationary}, since $S_\gamma^{\REF}$ parametrizes the flow
and $A$ has real principal symbol,
\begin{equation*}
    \begin{split}
        \int_{\mathbb{R}}  e^{iEt/h}  &\hat\rho(t) \Tr [A \chi(P_h) e^{-itP_h/h} A^*] dt \\
        \doteq & \frac{ i^{-\frac{n}{2}}e^{i\frac{\pi}{4}\sgn(\mathrm{Hess}_{S(t,\hat x)})} e^{-i\frac{\pi}{2}\mu_\gamma} e^{i(ET+S_\gamma)/h}} {\left|\det(I-\Pcal)\right|^{1/2}}\chi(E) \hat{\rho}(T) \rtot  h^{N\kk}
         \int_0^T a^2(X(s),\Xi(s)) \,ds,
    \end{split}
  \end{equation*}
where here and henceforth
  $$
\rtot(E)=\rtot(T(E), x_0, x_0)
$$
(with the notation of  \eqref{eq:ref-coeff-m})
is the product of all the reflections coefficients around the closed
orbit in question.
By \eqref{eq:stx-index}, we have
\[\frac{n}{2} - \frac{\sgn(\mathrm{Hess}_{S(t,\hat x)})}{2} + \mu_\gamma = \ind(\mathrm{Hess}_{S(t,\hat x)}) + \mu_\gamma = \sigma_\gamma,\]
and hence
\[ i^{-n/2}e^{i\frac{\pi}{4}\sgn(\mathrm{Hess}_{S(t,\hat x)})} e^{-i\frac{\pi}{2}\mu_\gamma} = i^{-\left(\frac{n}{2} - \frac{\sgn(\mathrm{Hess}_{S(t,\hat x)})}{2} + \mu_\gamma\right)} = i^{-\sigma_\gamma}.\]
Thus,
\begin{equation}
\label{eq:ml-trace}
    \begin{split}
        \int_{\mathbb{R}} & e^{iEt/h}  \hat\rho(t) \Tr [A \chi(P_h) e^{-itP_h/h} A^*] dt \\
        &\quad \doteq  \frac{i^{-\sigma_\gamma} e^{i(ET+S_\gamma)/h}} {\left|\det(I-\Pcal)\right|^{1/2}}\chi(E) \hat{\rho}(T) \rtot(E)  h^{N\kk}
         \int_0^T a^2(X(s),\Xi(s)) \,ds. 
    \end{split}
\end{equation}
Hence the total interior contribution to the trace is
\begin{equation}\label{interiortrace}
\int e^{iEt/h} \hat \rho(t) \Tr_\circ\big[ \chi(P_h) U(t)\big] \, dt
\doteq
  \frac{ i^{-\sigma_\gamma} e^{i(ET+S_\gamma)/h}}{\left|\det(I-\Pcal)\right|^{1/2}}\chi(E)  \hat{\rho}(T)\, \rtot(E)  h^{N\kk} \sum_{j\in\circ} \int_0^T 
    a_j^2( X(s) ,\Xi(s))\, ds.
\end{equation}

\subsection{Microlocalized trace near the hypersurface} 
\label{sec:traceatY}
We now turn to the contribution to the trace of the boundary term
$\Tr_\pa$ in \eqref{trsplit}. 
By \eqref{trdelta},
$$
\int e^{iEt/h} \hat \rho(t) \Tr_\pa(t) \, dt  =I_\pa+ O(h^{k_0(N+k+1)-n-0})
$$
where
\[ I_{\partial} = \int_{\mathbb{R}} e^{iEt/h}\hat\rho(t) \sum_{k \in
      \indexint}\Tr [A_k \chi(P_h) U(t-\shorttime) \psi_Y \chi_1(P_h)  U(\shorttime) A_k^*] dt. \]
By the Partition Property \eqref{p:ii}, there are two types of interior microlocal cutoffs $\indexint = \indexint_1\cup \indexint_2$:
\begin{enumerate}
    \item $\indexint_1$: $A_k$ such that $\XPhi{E}^{N+k}_{\shorttime} (\WF' A_k) \subset \{\psi_Y=1\}$;
    \item $\indexint_2$: $A_k$ such that $\pi\big(\XPhi{E}^{N+k}_{\shorttime} (\WF'
A_k) \big) \cap Y = \emptyset$.
\end{enumerate}
{(In both cases, the property is to hold for all $E \in I$.)
We  further split $I_{\partial} = I_{\partial_1} + I_{\partial_2}$, where $I_{\partial_i}$ corresponds to the summation over $\circ_i, i=1,2$. For $k\in\indexint_1$, we have
by propagation of singularities \[\Tr [A_k \chi(P_h) U(t-\shorttime) \psi_Y \chi_1(P_h)  U(\shorttime) A_k^*] = \Tr [A_k \chi(P_h) U(t) A_k^*]+O(h^{k_0(N+k+1)-n-0}). \]
By equation \eqref{eq:ml-trace}, then,
\begin{equation}
\label{woodshole1}
    \begin{split}
      I_{\partial_1} \doteq \frac{ i^{-\sigma_\gamma} e^{i(ET+S_\gamma)/h}} {\left|\det(I-\Pcal)\right|^{1/2}}\chi(E) \hat{\rho}(T) \rtot(E)  h^{N\kk}
         \sum_{k \in
      \indexint_1} \int_0^T a_k^2(X(s,x_0,\xi_0),\Xi(s,x_0,\xi_0))\,ds.
    \end{split}
\end{equation}

For $A_k\in\indexint_2$, we need to compute
\[ \int_{\mathbb{R}} e^{iEt/h} \hat\rho(t) \Tr [A_k \chi(P_h) U(t-\shorttime) \psi_Y \chi_1(P_h)  U(\shorttime) A_k^*] dt\]
by applying Proposition \ref{prop:stationary} to the propagator 
$A_k \chi(P_h) U(t-\shorttime) \psi_Y \chi_1(P_h)  U(\shorttime)
A_k^*$.

We now claim that we may drop the $\chi_1(P_h)$ spectral cutoff
  from this expression modulo an $O_{L^2}(h^{\kk(N+k+1)-0})$ error:
  First we note that since, by construction and propagation of
  singularities, $U(\shorttime) A_k^*= O_{L^2}(h^{\kk(N+k+1)-0})$ in a
  neighborhood of $Y$, we may restrict our attention to $X
  \backslash Y$ at the cost of such an error, i.e.\ we may insert a
  microlocalizer $B \in \Psi_h(X)$, supported away from $Y$ and
  microlocally equal to the identity on $\WFh^{\kk(N+k+1)-0}
  U(\shorttime) A_k^*$
  (and commute $\chi(P_h)$ through the
  first propagator) and analyze
  $$
A_k U(t-\shorttime) \chi(P_h)\psi_Y B \chi_1(P_h)  U(\shorttime)
A_k^*.
 $$
Note that while $\WF' B$ may have multiple connected components due to
   potential branching within time $T_1$, only one of them contributes
   to the main singularity due to our assumption on the unique closed
   branching orbit of time $t$.  (In fact we can arrange that at most one reflection can occur owing
   to the
   transversality of $\gamma$ to $Y$.)
 On $\WF' B$,
  $\chi_1(P_h)$ and $\chi(P_h)$ are
  semiclassical pseudodifferential operators, with the former
  microlocally equal to $1$ on the microsupport of the latter. Thus we have 
  \[\chi(P_h)\psi_Y B (I -\chi_1(P_h)) \equiv 0.\]
  Hence
  $$\begin{aligned}
A_k U(t-\shorttime) \chi(P_h)\psi_Y B \chi_1(P_h)  U(\shorttime)
A_k^*&\equiv A_k U(t-\shorttime) \chi(P_h)\psi_Y B U(\shorttime)
A_k^*\\ &= A_k \chi(P_h) U(t-\shorttime) \psi_Y B U(\shorttime)
A_k^*.
\end{aligned}$$
The virtue of this expression is that it is a composition of two
branching propagators, just as we have analyzed in
Section~\ref{sec:multiplereflection}. 
Thus we may finally apply stationary phase as in that section to find that the
  kernel of such a microlocal propagator is given by
\begin{equation}
    \begin{split}
A_k U(t-\shorttime) &\psi_Y \chi_1(P_h)  U(\shorttime) A_k^*\\ 
  &= A_k U(t-\shorttime) \psi_Y B  U(\shorttime)
A_k^*  + O_{L^2}(h^{k_0(N+k+1)-0)})\\ 
&\doteq  {(2\pi i h)^{-\frac{n}{2}}}  e^{\frac{i}{h}S^{\REF}_\gamma(t,z,w)}
\abs{\Delta_\gamma}^{\frac{1}{2}} e^{-i\frac{\pi}{2}\mu_\gamma} \rtot(E)  h^{N\kk} \,  a_k(z,\partial_z S^\REF_{\gamma})
\\
& \quad \quad \quad \psi_{Y}(\pi(\XPhi{E}^\gamma_{T_1}(w,-\partial_w S^\REF_{\gamma}))) a_k(w,-\partial_w S^\REF_{\gamma}) \abs{dz\, dw}^{\frac{1}{2}}.
    \end{split}
\end{equation}
  (Note that the symbol of $B$ is equal to $1$ on the region in
  question, so this plays no role.)
Inserting the other spectral cutoff $\chi(P_h)=\chi(-hD_t)$ when composing with the propagator $U(t)$ (cf.~Lemma~\ref{lemma:timeorspace}), and as before
taking $k$ sufficiently large to make the error term smaller than the
main one, we find that the corresponding leading order of the trace is given by
\begin{equation}
\label{woodshole3}
    \begin{split}
I_{\pa_2} \doteq   \sum_{k \in \indexint_2}\frac{i^{-\sigma_\gamma}e^{i(ET+S_\gamma)/h}}{\left|\det(I-\Pcal)\right|^{1/2}}\chi(E)
\hat{\rho}(T) \rtot(E)  h^{N\kk} \int_0^T
&\psi_{Y}(\pi(\XPhi{E}^\gamma_{T_1}(X(s), \Xi(s))))  \\ 
         & \cdot a_k^2(X(s),\Xi(s)) \,ds.
    \end{split}
  \end{equation} 
  Now we reassemble $I_\pa=I_{\pa_1}+I_{\pa_2}$. We first insert a
  $\psi_{Y}(\pi(\XPhi{E}^\gamma_{T_1}(X(s), \Xi(s))))$ factor in the integral term in
  \eqref{woodshole1}, since this factor is identically $1$ on the
  support of the integrals in the first case. Combining them with
  terms in \eqref{woodshole3}, we obtain, by pulling back under the
  locally well-defined map $\XPhi{E}^\gamma_{T_1}$,
\begin{equation*}
    \begin{split}
I_\pa & \doteq \int e^{iEt/h} \hat \rho(t) \Tr_\pa\big[ \chi(P_h) U(t)\big] \,dt \\
&\doteq  \frac{ i^{-\sigma_\gamma} e^{i(ET+S_\gamma)/h}}{\left|\det(I-\Pcal)\right|^{1/2}}\chi(E)  \hat{\rho}(T) \rtot(E)  h^{N\kk} \sum_{k\in\indexint} \int_0^T \psi_{Y}(\pi(\XPhi{E}^\gamma_{T_1}( X(s) ,\Xi(s)))) a_k^2(X(s),\Xi(s))\,ds\\
& = \frac{i^{-\sigma_\gamma} e^{i(ET+S_\gamma)/h}}{\left|\det(I-\Pcal)\right|^{1/2}}\chi(E)  \hat{\rho}(T) \rtot(E)  h^{N\kk} \sum_{k\in\indexint, j\in\indexY} \int_0^T a_j^2( X(s) ,\Xi(s))\, 
a_k^2(\XPhi{E}^\gamma_{-T_1}(X(s),\Xi(s))) \,ds\\
& \doteq  \frac{ i^{-\sigma_\gamma} e^{i(ET+S_\gamma)/h}}{\left|\det(I-\Pcal)\right|^{1/2}}\chi(E)  \hat{\rho}(T)\, \rtot(E)  h^{N\kk} \sum_{j\in\indexY} \int_0^T a_j^2( X(s) ,\Xi(s))\, ds;
    \end{split}
\end{equation*}
in the last equality we use the property that
$\{A_j^*A_j\}_{j\in\indexint}$ is an interior partition of unity and
the fact that if $\gamma(s)$ is in $\WF' A_j$ with $j \in \indexY$
then $\gamma(s-\shorttime)$ cannot lie in $\supp \psi_Y$ by Partition
Property \eqref{p:iii}.

Finally, adding this boundary contribution to
\eqref{interiortrace} yields the total trace contribution in \eqref{FTtrace} from the orbit cylinder $\gamma$
\begin{equation}
  \begin{aligned}
  I_E & \equiv \int e^{iEt/h} \hat \rho(t) \Tr_\circ\big[ \chi(P_h) U(t)\big] \,
  dt+\int e^{iEt/h} \hat \rho(t) \Tr_\pa\big[ \chi(P_h) U(t)\big] \,  dt\\ &\doteq
  \frac{ i^{-\sigma_\gamma} e^{i(ET+S_\gamma)/h}}{\left|\det(I-\Pcal)\right|^{1/2}}\chi(E)  \hat{\rho}(T)\, \rtot(E)  h^{N\kk} \sum_{j\in\indexset} \int_0^T 
    a_j^2( X(s) ,\Xi(s))\, ds.
\end{aligned}
  \end{equation}
The microlocal partition of unity then gives the leading order of the trace 
\begin{equation}
\label{eq:trace}
    \begin{split}
        \int_{\mathbb{R}} e^{iEt/h} \hat\rho(t) \Tr [\chi(P_h) e^{-itP_h/h}] dt 
        \doteq \frac{ i^{-\sigma_\gamma} e^{i(ET+S_\gamma)/h}} {\left|\det(I-\Pcal)\right|^{1/2}} \rtot(E)  h^{N\kk} \chi(E) \hat{\rho}(T) T^\sharp_\gamma 
    \end{split}
\end{equation}
where $T^\sharp_\gamma$ is the primitive period of $\gamma$.

\appendix
				
\section{Periodic variational problem for mechanical systems with reflection}
\label{appendix:periodic}

In this appendix, we extend the results of \cite{WuYaZo:24} on
variational problems for periodic branching trajectories to periodic
variational problems with varying time. This is used to compute the
Morse indices in the Gutzwiller trace formula.  The
results here are minor variants of those in \cite{WuYaZo:24}, and we
have given references to that paper in lieu of repeating proofs.

Consider \emph{branching loops} (i.e, ``reflected paths'' in the sense
of Definition~2.1 of \cite{WuYaZo:24}): these are projections of
branching bicharacteristics that have the same start- and endpoint in
$X$ but
are not necessarily periodic, i.e., 
$\alpha:[0,T]\to M$
with the property that $\alpha(0)=\alpha(T)$ and with the possibility
of reflections at the hypersurface $Y$. Here, in contrast to
\cite{WuYaZo:24}, we do not fix $T$ ahead of time. If we consider
variations of such loops, with $T$ allowed to vary as well, we then
obtain
\[\alpha(T(\ep),\ep) = \alpha(0,\ep) \implies W(T) + T'_\ep\dot\alpha = W(0),\quad W:=\frac{\partial\alpha}{\partial\ep}.\]
Thus the tangent space of such variations is
\[T_\alpha\Omegaq = \{W\,:\,W(T)-W(0)\in\text{span }\dot\alpha\}.\]
Given $E\in\mathbb{R}$, we consider the functional
\[J[\alpha] = \int_0^T \frac{1}{4}|\dot\alpha|^2_g - V(\alpha(t)) + E\,dt.\]
Note that we have changed the normalization of the classical
  Lagrangian from  that used in 
  \cite{WuYaZo:24} where there is a $1/2$ instead of $1/4$.,  Note also that we will use the terminology
  ``reflected physical path'' from \cite{WuYaZo:24}; these paths are
  precisely the projections to the base of branching bicharacteristics.
\begin{lemma}
A loop $\alpha$ is stationary for $J$ iff $\alpha$ is a reflected
physical path in the sense of Definition~2.5 of \cite{WuYaZo:24}, and $(\alpha,\dot\alpha)$ lies on the energy surface $\frac{1}{4}|\dot\alpha|^2_g+V(\alpha) = E$.
\end{lemma}
\noindent (Cf.\ Lemma~3.2 of \cite{WuYaZo:24}.)

Let $\alpha$ be such a path which is also periodic:
$\alpha(0)=\alpha(T)=p$, $\dot\alpha(0)=\dot\alpha(T)$. Suppose that
$p$ is not conjugate to itself along $\alpha$ (in the sense of
Definition~4.6 of \cite{WuYaZo:24}). Let $\Omega_0(M;p,p)$ denote the space of paths with fixed endpoints at $p$ and $p$. Let
\[\mathcal{J} = \{W\in T_\alpha\Omegaq\,:\,W|_{(0,T)}\text{ is a
    Jacobi field}\}.\]
(For the definition of Jacobi fields in this context, see
  Definition~4.4 of \cite{WuYaZo:24}.)
Note that $\dim\mathcal{J} = n+1$. Indeed, Jacobi fields are uniquely determined (due to non-conjugacy) by their values at $t=0,T$; the tangent condition $W(T)-W(0)\in\text{span }\dot\alpha$ is an $(n-1)$-dimensional condition on the $2n$-dimensional space of Jacobi fields without boundary conditions.
\begin{lemma}
For a periodic loop $\alpha$ such that $J$ is stationary in $T_\alpha\Omegaq$, the second variation $J''$ is a well-defined bilinear form on $T_\alpha\Omegaq$. Moreover,
\[T_\alpha\Omegaq = \mathcal{J}\oplus T_\alpha\Omega_0,\]
with the direct sum orthogonal with respect to $J''$.
\end{lemma}
\noindent(Cf.\ the proof of Theorem~5.2 of \cite{WuYaZo:24}.)

Let
\begin{equation}
\label{eq:sigma-path}
\sigma_\alpha := \ind(J''|_{T_\alpha\Omegaq}).
\end{equation}
The direct sum decomposition above directly gives:
\begin{corollary}
\label{cor:sigma}
\[\sigma_\alpha = \ind(J''|_{\mathcal{J}}) + \ind(J''|_{T_\alpha\Omega_0}) = \ind(J''|_{\mathcal{J}}) + \text{\# of conjugate points along }\alpha.\]
\end{corollary}
Finally, for $(t,x)\in\mathbb{R}\times M$, let $\alpha_{t,x}$ be the loop which starts and ends at $x$ with period $t$. Note that
\[S(t,x,x) + Et = J[\alpha_{t,x}].\]
For $(\delta t,\delta x)\in T_{(t,x)}(\mathbb{R}\times M)$, let $W_{\delta t,\delta x} = \frac{d}{d\ep}|_{\ep=0}[\alpha_{t(\ep),x(\ep)}]$, where $(t(\ep),x(\ep))$ is a path in $\mathbb{R}\times M$ such that $(t(0),x(0)) = (t,x)$, $(t'(0),x'(0)) = (\delta t,\delta x)$. Note that $W_{\delta t,\delta x}\in\mathcal{J}$, and $(\delta t,\delta x)\mapsto W_{\delta t,\delta x}$ is an isomorphism $T_{(t,x)}(\mathbb{R}\times M)\cong \mathcal{J}$. Then
\[J''(W_{\delta t,\delta x},W_{\delta t,\delta x}) = \partial^2_{t,x}[S(t,x,x)]((\delta t,\delta x),(\delta t,\delta x));\]
here both $J''$ and $\partial^2_{t,x}[S(t,x,x)]$ are interpreted as bilinear forms. It then follows that
\[\ind(J''|_{\mathcal{J}}) = \ind(\partial^2_{t,x}[S(t,x,x)]).\]
Note that $\partial^2_{t,x}[S(t,x,x)]$ is degenerate along the direction $(\delta t = 0, \delta x = \dot X)$. Thus, on a complement, we have
\[\ind(\partial^2_{t,x}[S(t,x,x)]) = \ind(\partial^2_{t,\hat x}[S(t,x,x)]).\]
It follows, from Corollary \ref{cor:sigma}, that the index of the space-time Hessian $\partial^2_{t,\hat x}[S(t,x,x)]$ appearing in the final stationary phase can be identified dynamically, via the equation
\begin{equation}
\label{eq:stx-index}
\ind(\partial^2_{t,\hat x}[S(t,x,x)]) = \sigma_\alpha - \text{\# of conjugate points along }\alpha.
\end{equation}

\section{Composition of Van Vleck determinants}
\label{sec:det}

Let $X=Y=Z=\mathbb{R}^{n}$, and
let $S_{1}: X \times Y \rightarrow \mathbb{R}$ and $S_{2}: Y \times Z
\rightarrow \mathbb{R}$ be ``phase functions.'' We suppose, for every
$(x, z) \in X \times Z$, that there exists a unique $y=Y(x, z) \in Y$
such that 
$$
\left.\partial_{y}\left(S_{1}(x, y)+S_{2}(y, z)\right)\right|_{y=Y(x, z)}=0
$$
Suppose as well that $Y$ is smooth in $(x, z)$. Let
$$
\widetilde{S}(x, z)=S_{1}(x, Y(x, z))+S_{2}(Y(x, z), z)
$$
Let $\partial_{x y}^{2} S_{1}$ denote the $n \times n$-matrix whose $(i, j)$ th entry is $\partial_{x_{i} y_{j}}^{2} S_{1}$, and similarly with the other functions/variables. Note that $\partial_{y x}^{2} S_{1}=\left(\partial_{x y}^{2} S_{1}\right)^{\top}$, etc.; in particular the matrices with the order of differentiation changed have the same determinant.

\begin{lemma}\label{lemma:vanvleck}
We have
$$
\left|\operatorname{det} \partial_{x y}^{2} S_{1}(x, y)\right|\left|\operatorname{det} \partial_{y z}^{2} S_{2}(y, z)\right|=\left|\operatorname{det} \partial_{y y}^{2}\left(S_{1}(x, y)+S_{2}(y, z)\right)\right|\left|\operatorname{det} \partial_{x z}^{2} \widetilde{S}(x, z)\right|
$$
when evaluated at $y=Y(x, z)$.
\end{lemma}

The utility of this result in stationary phase computations is as follows: if $a_{1}: X \times Y \rightarrow \mathbb{C}$ and $a_{2}: Y \times Z \rightarrow \mathbb{C}$ are ``amplitude'' functions, and
\begin{gather*}
   u_{1}(x, y)=(2\pi h)^{-n / 2} e^{\frac{i}{h} S_{1}(x, y)} a_{1}(x, y)\left|\operatorname{det} \partial_{x y}^{2} S_{1}(x, y)\right|^{1 / 2}|d x d y|^{1 / 2} \\
u_{2}(y, z)=(2\pi h)^{-n/ 2} e^{\frac{i}{h} S_{2}(y, z)} a_{2}(y, z)\left|\operatorname{det} \partial_{y z}^{2} S_{2}(y, z)\right|^{1 / 2}|d y d z|^{1 / 2} 
\end{gather*}
then the composition $u(x, z):=\int_{Y} u_{1}(x, y) u_{2}(y, z)$ equals
$$
u(x, z)=(2\pi h)^{-n} \tilde{u}(x, z)|d x d z|^{1 / 2}
$$
where
$$\tilde{u}(x, z)=\int_{Y} e^{\frac{i}{h}\left(S_{1}(x, y)+S_{2}(y, z)\right)} a_{1}(x, y) a_{2}(y, z)\left|\operatorname{det} \partial_{x y}^{2} S_{1}(x, y)\right|^{1 / 2}\left|\operatorname{det} \partial_{y z}^{2} S_{2}(y, z)\right|^{1 / 2}|d y|$$
From stationary phase, we have
$$\tilde{u}(x, z)=(2\pi h)^{n / 2} e^{\frac{i}{h} \widetilde{S}(x, z)}\left(\left.\left(e^{i \pi \sigma / 4} \frac{\left|\operatorname{det} \partial_{x y}^{2} S_{1}\right|^{1 / 2}\left|\operatorname{det} \partial_{y z}^{2} S_{2}\right|^{1 / 2}}{\left|\operatorname{det} \partial_{y y}^{2}\left(S_{1}+S_{2}\right)\right|^{1 / 2}} a_{1} a_{2}\right)\right|_{y=Y(x, z)}+O\left(h\right)\right)$$
where $\sigma=\operatorname{sgn} \partial_{y y}^{2}\left(S_{1}+S_{2}\right)$. Thus, modulo Maslov factors, if the claim holds, we can write

$$
u(x, z)=(2\pi h)^{-n / 2}\left(e^{\frac{i}{h} \widetilde{S}(x, z)} a(x, z)\left|\operatorname{det} \partial_{x z}^{2} \widetilde{S}(x, z)\right|^{1 / 2}+O\left(h\right)\right)|d x d z|^{1 / 2}
$$
where $a(x, z)=a_{1}(x, Y(x, z)) a_{2}(Y(x, z), z)$.

\begin{proof}
For each $1 \leq i \leq n$ we have
$$
\partial_{y_{i}} S_{1}(x, Y(x, z))+\partial_{y_{i}} S_{2}(Y(x, z), z)=0 .
$$
Taking the $x_{j}$ partial derivative of the equation yields
$$
\partial_{y_{i} x_{j}}^{2} S_{1}(x, Y(x, z))+\sum_{k=1}^{n}\left(\partial_{y_{i} y_{k}}^{2} S_{1}(x, Y(x, z))+\partial_{y_{i} y_{k}}^{2} S_{2}(Y(x, z), z)\right) \partial_{x_{j}} Y_{k}(x, z)=0
$$
Taking the $z_{j}$ partial derivative instead yields
$$
\sum_{k=1}^{n}\left(\partial_{y_{i} y_{k}}^{2} S_{1}(x, Y(x, z))+\partial_{y_{i} y_{k}}^{2} S_{2}(Y(x, z), z)\right) \partial_{z_{j}} Y_{k}(x, z)+\partial_{y_{i} z_{j}}^{2} S_{2}(x, Y(x, z))=0
$$
In matrix notation, we thus have
\begin{gather}
\label{eq:B1}
\partial_{y x}^{2} S_{1}(x, Y(x, z))=-\left(\partial_{y y}^{2} S_{1}(x, Y(x, z))+\partial_{y y}^{2} S_{2}(Y(x, z), z)\right) \partial_{x} Y(x, z)\\
\label{eq:B2}
\partial_{y z}^{2} S_{2}(Y(x, z), z)=-\left(\partial_{y y}^{2} S_{1}(x, Y(x, z))+\partial_{y y}^{2} S_{2}(Y(x, z), z)\right) \partial_{z} Y(x, z)
\end{gather}
On the other hand, from
$$
\widetilde{S}(x, z)=S_{1}(x, Y(x, z))+S_{2}(Y(x, z), z)
$$
taking the $x_{i}$ derivative yields
$$
\partial_{x_{i}} \widetilde{S}(x, z)=\partial_{x_{i}} S_{1}(x, Y(x, z))+\sum_{k=1}^{n}\left(\partial_{y_{k}} S_{1}(x, Y(x, z))+\partial_{y_{k}} S_{2}(Y(x, z), z)\right) \partial_{x_{i}} Y_{k} .
$$
By the definition of $Y(x, z)$, we have that $\partial_{y_{k}} S_{1}(x, Y(x, z))+\partial_{y_{k}} S_{2}(Y(x, z), z)=0$, and hence
$$
\partial_{x_{i}} \widetilde{S}(x, z)=\partial_{x_{i}} S_{1}(x, Y(x, z)).
$$
It follows by taking the $z_{j}$ derivative that
$$
\partial_{x_{i} z_{j}}^{2} \widetilde{S}(x, z)=\sum_{k=1}^{n} \partial_{x_{i} y_{k}}^{2} S_{1}(x, Y(x, z)) \partial_{z_{j}} Y_{k}
$$
or, in terms of matrices,
\begin{equation}
\label{eq:B3}
\partial_{x z}^{2} \widetilde{S}(x, z)=\partial_{x y}^{2} S_{1}(x, Y(x, z)) \partial_{z} Y(x, z).
\end{equation}
Similar logic yields
$$
\partial_{z_{i}} \widetilde{S}(x, z)=\partial_{z_{i}} S_{2}(Y(x, z), z)
$$
from which we obtain
\begin{equation}
\label{eq:B4}
\partial_{z x}^{2} \widetilde{S}(x, z)=\partial_{z y}^{2} S_{2}(Y(x, z)) \partial_{x} Y(x, z).
\end{equation}
From Equations \eqref{eq:B1}-\eqref{eq:B3}, we note that if either the determinant  $\operatorname{det} \partial_{x y}^{2} S_{1}(x, Y(x, z))$ or the determinant $\operatorname{det} \partial_{y z}^{2} S_{2}(Y(x, z), z)$ is zero, then \eqref{eq:B3} or \eqref{eq:B4} yields that $\operatorname{det} \partial_{x z}^{2} \widetilde{S}(x, z)=0$ as well,
and hence the desired lemma holds. Otherwise, assume $\operatorname{det} \partial_{x y}^{2} S_{1}(x, Y(x, z))$ and $\operatorname{det} \partial_{y z}^{2} S_{2}(Y(x, z), z)$ are both nonzero. Then \eqref{eq:B1} and \eqref{eq:B2} yield
$$
\operatorname{det}\left(\partial_{y y}^{2} S_{1}(x, Y(x, z))+\partial_{y y}^{2} S_{2}(Y(x, z), z)\right),\  \operatorname{det} \partial_{x} Y(x, z),\  \operatorname{det} \partial_{z} Y(x, z)
$$
are all nonzero. Then, taking the absolute value of the determinants of both sides of Equations \eqref{eq:B1}-\eqref{eq:B3}, and multiplying the LHS of  \eqref{eq:B1} and  \eqref{eq:B2} with the RHS of \eqref{eq:B3} and  \eqref{eq:B4} (and vice-versa) yields
$$
\begin{aligned}
& \left|\operatorname{det} \partial_{x y}^{2} S_{1}\right|^{2}\left|\operatorname{det} \partial_{y z}^{2} S_{2}\right|^{2}\left|\operatorname{det} \partial_{x} Y\right|\left|\operatorname{det} \partial_{z} Y\right| \\
& =\left|\operatorname{det} \partial_{y y}^{2}\left(S_{1}+S_{2}\right)\right|^{2}\left|\operatorname{det} \partial_{x z}^{2} \widetilde{S}\right|^{2}\left|\operatorname{det} \partial_{x} Y\right|\left|\operatorname{det} \partial_{z} Y\right|
\end{aligned}
$$
evaluated at $(x, y)=(Y(x, z), z)$. Dividing by $\left|\operatorname{det} \partial_{x} Y\right| \left| \operatorname{det} \partial_{z} Y\right|$ on both sides (which is allowable since these quantities are nonzero in this situation) and taking square roots yields the desired determinant equality.
\end{proof}

\bibliographystyle{alpha}
\bibliography{sing-gutz}

\end{document}